\documentclass[reqno,10pt]{amsart}

\usepackage[top=2.0cm,bottom=2.0cm,left=3cm,right=3cm]{geometry}
\usepackage{amsthm,amsmath,amssymb,dsfont}
\usepackage{mathrsfs,amsfonts,functan,extarrows,mathtools}
\usepackage[colorlinks]{hyperref}
\usepackage{marginnote}
\usepackage{xcolor}
\usepackage{stmaryrd}
\usepackage{esint}
\usepackage{graphicx}
\usepackage{bbm}

\usepackage{tikz}
\usetikzlibrary{positioning}

\newtheorem{theorem}{Theorem}[section]
\newtheorem{proposition}{Proposition}[section]

\newtheorem{lemma}{Lemma}[section]

\numberwithin{equation}{section}
\allowdisplaybreaks

\def\d{\mathrm{d}}

\def\exp{\mathrm{exp}}

\newcounter{wronumber}

\begin{document}
\title{The Incompressible Navier-Stokes-Fourier Limit from Boltzmann-Fermi-Dirac Equation}

\author[Ning Jiang]{Ning Jiang}
\address[Ning Jiang]{\newline School of Mathematics and Statistics, Wuhan University, Wuhan, 430072, P. R. China}
\email{njiang@whu.edu.cn}

\author[Linjie Xiong]{Linjie Xiong}
\address[Linjie Xiong]
{\newline  School of Mathematics, Hunan University, Changsha 410082, P. R. China}
\email{xlj@hnu.edu.cn}

\author[Kai Zhou]{Kai Zhou}
\address[Kai Zhou]
		{\newline School of Mathematics and Statistics, Wuhan University, Wuhan, 430072, P. R. China}
\email{kaizhoucm@whu.edu.cn}
\thanks{\today}

\maketitle

\begin{abstract}
 We study Boltzmann-Fermi-Dirac equation when quantum effects are taken into account in dilute gas dynamics. By employing new estimates on trilinear terms of collision kernels, we prove the global existence of  the classical solution to Boltzmann-Fermi-Dirac equation near equilibrium. Furthermore, the limit from Boltzmann-Fermi-Dirac equation to incompressible Navier-Stokes-Fourier equations is justified rigorously. The corresponding formal analysis was given in the thesis of Zakrevskiy \cite{Zakrevskiy}\\

   \noindent\textsc{Keywords.} Boltzmann-Fermi-Dirac equation, Classical solutions, Navier-Stokes-Fourier Limit.\\

\end{abstract}





\section{Introduction}\label{sec1}
\subsection{The Boltzmann-Fermi-Dirac equation.} The evolution of quantum particles obeying Fermi-Dirac statistics can be described by Boltzmann-Fermi-Dirac equation:
\begin{equation}\label{1eq1}
\left\{
\begin{aligned}
    &\partial_t F + v\cdot \nabla F =C(F),\\
    &F\big|_{t=0}=F_0,
\end{aligned}
\right.
\end{equation}
where $0\leqslant F(t,x,v) \leqslant 1$ is the number density of particles at time $t\geqslant 0$, position $x\in \mathbb{R}^3$, with velocity $v\in \mathbb{R}^3$. The collision integral $C(F)$ takes the form
\begin{equation*}
  \iint\limits_{\mathbb{R}^3\times\mathbb{S}^2}b(v_1-v,\omega)   \Big[\,F^\prime F_1^\prime (1-\delta F)(1-\delta F_1)-FF_1(1-\delta F^\prime)(1-\delta F_1^\prime)\,\Big] \d\omega\d v_1.
\end{equation*}
In the expression above, $\delta=\hbar^3$ is a constant ($\hbar$ is the Planck constant), $F\equiv F(t,x,v)$, $F_1\equiv F(t,x,v_1)$, $F^\prime\equiv F(t,x,v^\prime)$, $F_1^\prime\equiv F(t,x,v_1^\prime)$, and for $\omega\in \mathbb{S}^2$,
\begin{equation}\nonumber
  v^\prime=v+[(v_1-v)\cdot \omega]\omega,\quad v_1^\prime=v_1-[(v_1-v)\cdot \omega]\omega
\end{equation}
are velocities after a collision of two particles with velocities $v$ and $v_1$ before. The particle pairs with the same mass, during the collision, follow the conservation laws of momentum and kinetic energy:
\begin{equation}\nonumber
  v+v_1=v^\prime+v_1^\prime,\qquad |v|^2+|v_1|^2=|v^\prime|^2+|v_1^\prime|^2.
\end{equation}

The collision kernel $b(v_1-v,\omega)$ is an a.e. positive function defined on $\mathbb{R}^3\times\mathbb{S}^2$, which encodes features of the molecular interaction in kinetic theory. Physically, it is assumed to depend only on the modulus of the relative velocity $|v-v_1|$ and on the scalar product $\frac{v_1-v}{|v_1-v|}\cdot\omega$. We assume the collision kernel takes the factor form:
\begin{equation*}
  b(v_1-v,\omega)=|v_1-v|^{\gamma}\hat{b}(\cos\theta),\quad\cos\theta=\frac{v_1-v}{|v_1-v|}\cdot\omega, \quad -3<\gamma\leqslant 1.
\end{equation*}
For $\gamma \geqslant 0$, we call the collision kernel a hard potential; in particular, for $\gamma=0$, we call it a Maxwell collision kernel and for $\gamma=1$, we call the collision hard sphere collision; and for $-3<\gamma<0,$ we call it a soft potential. In Grad angular cutoff, $\hat{b}(\cos\theta)$ satisfies
\begin{equation*}
  \int_{0}^{\pi/2} \hat{b}(\cos\theta)\sin\theta d\theta< +\infty.
\end{equation*}
Throughout the paper, we take $\delta=1$ and the hard sphere collision in \eqref{1eq1} for convenience, namely, the collision kernel has the following explicit expression
\begin{equation}\label{1eq4}
  b(v_1-v,\omega)=|(v_1-v)\cdot\omega|.
\end{equation}

Next, we list some basic properties of Boltzmann-Fermi-Dirac equation \eqref{1eq1}. First, the collision operator $C(F)$ satisfies the conservation laws:
\begin{equation}\label{1eq5}
  \int_{\mathbb{R}^3}C(F)\left(\begin{array}{c}
                                 1 \\
                                 v \\
                                 |v|^2
                               \end{array}\right)\d v=0.
\end{equation}

Next, as an analog of Boltzmann $H$-theorem, the following three assertions are equivalent \cite{Zakrevskiy}:
\begin{itemize}
  \item[(1)] $C(F)=0$;
  \item[(2)] The entropy production rate is zero,
  \begin{equation*}
    \int_{\mathbb{R}^3} C(F)\ln\frac{1-F}{F}\d v=0;
  \end{equation*}
  \item[(3)] $F$ is a Fermi-Dirac distribution,
  \begin{equation*}
    F_{f,u,\theta}(t,x,v)=\frac{1}{1+\exp(\frac{|v-u|^2}{2\theta}-f)}.
  \end{equation*}
\end{itemize}
In (3) above, $\theta=\theta(t,x)>0$ is the temperature, $u=u(t,x)\in \mathbb{R}^3$ is the bulk velocity, and $f(t,x)/\theta(t,x)$ is the total chemical potential. By the appropriate
choice of Galilean frame, the Fermi-Dirac distribution $F$ can be taken with $(f,u,\theta)=(1,0,1)$, we denote it by
\begin{equation}\label{global-equi}
  \displaystyle\mu(v)=\frac{1}{1+e^{\frac{|v|^2}{2}-1}}.
\end{equation}

The main goals of the current paper are the global existence of classical solutions to Boltzmann-Fermi-Dirac equation, and the connection between kinetic theory for Fermi-Dirac statistics and macroscopic fluid equations. In particular, this paper focuses on the incompressible Navier-Stokes scaling, under which the equation \eqref{1eq1} can be rescaled as
\begin{equation}\label{1eq7}
  \left\{
  \begin{aligned}
  \partial_t F_{\epsilon} + \frac{1}{\epsilon}v\cdot \nabla F_{\epsilon} =&\frac{1}{\epsilon^2}C(F_{\epsilon}),\\
  F_{\epsilon}\big|_{t=0}=&F_{\epsilon,0}.
  \end{aligned}
  \right.
\end{equation}
Here $\epsilon$ is the so-called Knudsen number which is the ratio between the mean free path and macroscopic length scale. It can be derived the incompressible Navier-Stokes equations (for details, see Chapter 3 of the thesis of Zakrevskiy \cite{Zakrevskiy}) from \eqref{1eq7} by fluctuating around the global Fermi-Dirac distribution $\mu$ with size $\epsilon$:
\begin{equation*}
  F_{\epsilon}=\mu+\epsilon\mu(1-\mu)g_{\epsilon}.
\end{equation*}
More specifically, putting the perturbation above into \eqref{1eq7} yields:
\begin{equation}\label{1eq8}
  \left\{
  \begin{aligned}
  \partial_t g_{\epsilon} + \frac{1}{\epsilon}v\cdot \nabla_x g_{\epsilon} +\frac{1}{\epsilon^2}Lg_{\epsilon} &=\frac{1}{\epsilon}Q(g_{\epsilon},g_{\epsilon})+T(g_{\epsilon},g_{\epsilon},g_{\epsilon}),\\
  g_{\epsilon}(t,x,v)|_{t=0}&=g_{\epsilon,0}(x,v)\,.
  \end{aligned}
\right.
\end{equation}
From now on, we use the notations
  \begin{equation*}
    f=f(v),\quad f_1=f(v_1),\quad f^\prime=f(v^\prime) ,\quad f_1^\prime=f(v_1^\prime),
  \end{equation*}
for any function $f$, and
\begin{equation*}
  \mathcal{N}=\mu_1^\prime \mu^\prime(1-\mu_1)(1-\mu)=\mu_1\mu(1-\mu_1^\prime)(1-\mu^\prime)\,.
\end{equation*}
Then the linear operator $L$ in \eqref{1eq8} is given by
\begin{equation}\label{1eq9}
  Lg=\nu(v)g-Kg,
\end{equation}
and $K=K_2-K_1$, where the collision frequency $\nu(v)$ is
\begin{equation}\label{1eq10}
  \nu(v)=\displaystyle\iint\limits_{\mathbb{R}^3\times\mathbb{S}^2}|v_1-v|\,|\cos\theta| \frac{\mathcal{N}}{\mu(1-\mu)}\d\omega \d v_1,
\end{equation}
the operator $K_1$ is
\begin{equation}\label{1eq11}
  K_1g=\displaystyle\iint\limits_{\mathbb{R}^3\times\mathbb{S}^2}|v_1-v|\,|\cos\theta| \frac{\mathcal{N}}{\mu(1-\mu)}g_1\d\omega \d v_1,
\end{equation}
and the operator $K_2$ is
\begin{equation}\label{1eq12}
    K_2g=\displaystyle\iint\limits_{\mathbb{R}^3\times\mathbb{S}^2}|v_1-v|\,|\cos\theta| \frac{\mathcal{N}}{\mu(1-\mu)} (g_1^\prime+g^\prime)\d\omega \d v_1.
\end{equation}

Moreover, in \eqref{1eq8}, the bilinear form $Q(f,g)$ is given by
\begin{equation}\label{1eq13}
  \begin{aligned}
  Q(f,g)=&\iint\limits_{\mathbb{R}^3\times\mathbb{S}^2}|v_1-v|\,|\cos\theta|\frac{\mathcal{N}}{\mu(1-\mu)} \bigg\{\bigg(1-\mu_1^\prime-\mu^\prime\bigg)f_1^\prime g^\prime -\bigg(1-\mu_1-\mu\bigg)f_1g\\
  &+\bigg(\mu_1^\prime f_1^\prime +\mu^\prime f^\prime \bigg)g_1-\bigg(f_1^\prime +f^\prime \bigg)\mu_1g_1+\bigg(\mu_1^\prime f_1^\prime +\mu^\prime f^\prime \bigg)g-\bigg(f_1^\prime +f^\prime \bigg)\mu g\bigg\}\d\omega \d v_1\\
  \stackrel{\triangle}{=}&Q_1(f,g)+Q_2(f,g)+\cdots+Q_6(f,g).
\end{aligned}
\end{equation}
and the trilinear form $T(f,g,h)$ is given by
\begin{equation}\label{1eq14}
  \begin{aligned}
  T(f,g,h)=&\iint\limits_{\mathbb{R}^3\times\mathbb{S}^2}|v_1-v|\,|\cos\theta|\frac{\mathcal{N}}{\mu(1-\mu)}\cdot
  \bigg\{\mu_1\mu f_1g\bigg(h_1^\prime +h^\prime\bigg) -\mu_1^\prime\mu^\prime f_1^\prime g^\prime \bigg(h_1+h\bigg)\\
  &+\mu_1\mu_1^\prime\bigg(f_1^\prime g^\prime h_1-f_1gh_1^\prime\bigg) +\mu\mu_1^\prime\bigg(f_1^\prime g^\prime h-f_1gh_1^\prime\bigg)+\mu_1\mu^\prime\bigg(f_1^\prime g^\prime h_1-f_1gh^\prime\bigg) \\
  &+\mu\mu^\prime\bigg(f_1^\prime g^\prime h-f_1gh^\prime\bigg)+f_1g\bigg(\mu_1^\prime h_1^\prime +\mu^\prime h^\prime\bigg) -f_1^\prime g^\prime \bigg(\mu_1h_1+\mu h\bigg)\bigg\}\d\omega \d v_1\\
  \stackrel{\triangle}{=}&T_1(f,g,h)+T_2(f,g,h)+\cdots+T_8(f,g,h).
  \end{aligned}
\end{equation}

\subsection{Well-posedness and hydrodynamic limits.} The Boltzmann-Fermi-Dirac equation, which describes the evolution of rarefied gas with quantum effect, is derived from a modification of the classical Boltzmann equation \cite{ref15}, when the exclusion Pauli principle are taken into account. It is also often called Uehling-Uhlenbeck equation or Nordheim equation. However, different from the classical case, the rigorous derivation of the Boltzmann-Fermi-Dirac equation has not been established. Since the heuristic arguments of Nordheim \cite{ref43}, Uehling and Uhlenbeck \cite{ref46},  rigorous  derivation of the Boltzmann-Fermi-Dirac equation can be found Sphon \cite{ref45}, Erd\"{o}s, Salmhofer and Yau \cite{ref19} and Benedetto, Pulvirenti, Castella and Esposito \cite{ref9}.

For mathematical theory of well-posedness, early results were obtained by Dolbeault \cite{ref18} and Lions \cite{ref37}. They studied the global existence of solutions in mild or distributional sense for the whole space $\mathbb{R}^3$ under some assumptions on the collision kernel. Furthermore, Dolbeault \cite{ref18} obtained that the solution of Boltzmann-Fermi-Dirac equation converges to the solution of the Boltzmann equation as $\delta \to 0$ for very special bounded collision kernel case. Allemand \cite{ref5} extended the results of \cite{ref18} to bounded domains with specular reflection at the boundaries for integrable collision kernels. Alexandre \cite{ref1} obtained another kind of weak solution satisfying the entropy inequality, the so-called $H$-solution. For general initial data, Lu \cite{ref39} studied the global existence and weak stability of weak solution in $\mathbb{T}_x^3$ for very soft potential with a weak angular cutoff. More results are referred to \cite{ref20, ref21, ref38, ref40}.

In the context of classical solution near global equilibrium, much less is known for Boltzmann-Fermi-Dirac equation (after we finished the draft of this paper, we were aware of the just posted paper \cite{Ou-WuL-2021} on this topic). We first review the corresponding results for Boltzmann equation. Ukai \cite{ref47} obtained the first global-in-time smooth solution with cutoff kernel. Later on, Guo \cite{ref27, ref28} developed the so-called nonlinear energy method to get the same type of result for soft potential with $\gamma\geq -3$. For more progress in this direction, we refer to \cite{AMUXY-1, ref11, ref12, GS-JAMS2011, ref29, ref32}.

In the other direction, the hydrodynamic limits from kinetic equations to fluid equations has been very active in recent decades. One of the important feature of kinetic equations is their connection to the fluid equations. The smaller the Knudsen number $\epsilon$ is, the more the dilute gas behaves like a fluid. Mathematically, the so-called hydrodynamic limits are the process that the Knudsen number goes to zero. Depending on the physically scalings, different fluid equations (incompressible of compressible Navier-Stokes, Euler, etc.) can be derived from kinetic equations. 

Bardos and Ukai \cite{ref8} proved the global existence of classical solution $g_\epsilon$ to scaled Boltzmann equation (perturbed around the global Maxwellian with size $\epsilon$) uniformly in $0<\epsilon<1$ for hard potential with cutoff collision kernel. Consequently they justified the limit to incompressible Navier-Stokes equations with small initial data. By employing semigroup approach, Briant \cite{ref10} also proved the same limit on the torus for hard cutoff potential, with convergence rate. Recently, Jiang, Xu and Zhao \cite{ref35} proved again the same limit for a more general class of collision kernel by using non-isotropic norm developed in the series of work \cite{AMUXY-1, AMUXY-2, AMUXY-3, GS-JAMS2011}. For the fluid limits of Boltzmann-Fermi-Dirac equation, Zakrevskiy \cite{Zakrevskiy} formally derived the compressible Euler and Navier-Stokes limits and incompressible Navier-Stokes limits. We also mention that, Filbet, Hu and Jin \cite{ref23} introduced a new scheme for quantum Boltzmann equation to capture the Euler limit by numerical computations.

Starting from the solutions to the limiting fluid equations, Caflisch \cite{ref14} and Nishida \cite{Nishida-1978} proved the compressible Euler limit from the Boltzmann equation in the context of classical solution by the Hilbert expansion, and analytic solutions, respectively. Caflisch's approach was applied to the acoustic limit by Guo, Jang and Jiang \cite{ref30, ref31, ref33} by combining with nonlinear energy method. We also mention some more results using Hilbert expansions \cite{Guo-CPAM2006, Jiang-Xiong-2015}.

The main concern of the current paper is to justify rigorously the formal derivation of Zakrevskiy \cite{Zakrevskiy}. The present paper verifies the incompressible Navier-Stokes limits rigorously in the context of classical solution. Precisely, we first prove the uniform in $\epsilon$ global existence of classical solution around the equilibrium to the scaled Boltzmann-Fermi-Dirac equation for the hard sphere collision. More importantly, we obtain the uniform energy estimate, using which we rigorously prove the limit from Boltzmann-Fermi-Dirac equation to Incompressible Navier-Stokes-Fourier equations by taking limit as $\epsilon\to 0$.

As indicated in \cite{Zakrevskiy}, the equilibrium of the Boltzmann-Fermi-Dirac equation \eqref{1eq1} is the global Fermi-Dirac distribution \eqref{global-equi}.  Thus we choose $\mu(1-\mu)$ as the weight, which leads to the gain of new nonlinear terms in \eqref{1eq8} in comparing with the case of perturbed Boltzmann equation, and brings us some new difficulties.

{\bf Notation.} We introduce some notations for the presentation throughout this paper.
We write $g\in L^2(\mu(1-\mu)\d v)\equiv L_v^2$ if $ \int_{\mathbb{R}^3}|g|^2\mu(1-\mu)\d v<+\infty$, and use $\langle \cdot,\cdot \rangle$ to denote the inner product in the Hilbert space $L^2(\mu(1-\mu)\d v)$, $|\cdot|_{L_v^2}$ the corresponding $L^2$ norm, and sometimes use $\langle g \rangle$ to denote  $\int_{\mathbb{R}^3}g\mu(1-\mu)\d v.$

Similarly, we write $g(x,v)\in H^N\left(\d x;L^2(\mu(1-\mu)\d v)\right)\equiv H_x^NL_v^2$ for integer $N\geqslant 0$ if
\begin{equation*}
  \sum_{|\alpha|\leqslant N}\iint_{\mathbb{R}^3\times\mathbb{R}^3}|\partial_x^\alpha g|^2\mu(1-\mu)\d v\d x<+\infty,
\end{equation*}
for multi-index $\alpha=(\alpha_1,\alpha_2,\alpha_3)$ and $|\alpha|=\sum\limits_{i=1}^{3}\alpha_i$,
and use $(\cdot,\,\cdot)_{H_x^NL_v^2}$ and $(\cdot,\,\cdot)_{H_x^N}$ to denote the inner product in the Hilbert space $H^N\left(\d x;L^2(\mu(1-\mu)\d v)\right)$ and $L^2(\mu(1-\mu)\d v)$ respectively, $\|\cdot\|_{H_x^NL_v^2}$ and $\|\cdot\| _{H_x^N}$ the corresponding norm, sometimes for $N=0$. We drop the subscript $H_x^NL_v^2$ and $H_x^N$ in the inner product and norm. It is also convenient to introduce a weighted inner product as
\begin{equation*}
  \left\langle f, g \right\rangle_{\nu}=\left\langle \nu f, g \right\rangle
\end{equation*}
for any functions $f(v)$ and $g(v)$ in $L^2(\mu(1-\mu)\d v)$, and use $|\cdot|_{\nu}$ for the corresponding weighted $L^2$ norm.

Finally, throughout this paper, let $N\geqslant 4$ be an integer. We define the operator $\mathbf{P}$ as $v$-projection in $L^2(\mu(1-\mu)\d v)$ to the null space $Null(L)$ of $L$. More specifically, by \eqref{2eq2}, for any $g(t,x,v)\in L^2(\mu(1-\mu)\d v)$, there exists constants (in $v$) $a(t,x)\in \mathbb{R}$, $b(t,x)\in \mathbb{R}^3$ and $c(t,x)\in \mathbb{R}$ such that
\begin{equation*}
  \mathbf{P} g=a(t,x)+b(t,x)\cdot v+ c(t,x)|v|^2,
\end{equation*}
then we have the following macro-micro decomposition
\begin{equation}\label{1eq15}
  g=\mathbf{P} g+\{\mathbf{I}-\mathbf{P}\}g,
\end{equation}
$\mathbf{P} g$ is called the macroscopic part and $\{\mathbf{I}-\mathbf{P}\}g$ is called the microscopic part. It's clear that
\begin{equation}\label{1eq16}
  \frac{1}{C}\left\|\partial_x^\alpha \mathbf{P} g\right\|_{\nu}\leqslant \left\|\partial_x^\alpha(a,b,c)\right\| \leqslant C\left\|\partial_x^\alpha \mathbf{P} g\right\|,
\end{equation}
for some constant $C>1$.

Then we can define the temporal energy functional as
\begin{equation*}
  \mathcal{E}_N^2(g)\equiv\mathcal{E}_N^2(g(t,x,v))=\sum\limits_{|\alpha|\leqslant N}\left\|\partial_x^{\alpha}g(t,\cdot,\cdot)\right\|^2,
\end{equation*}
and the microscopic dissipation rate as
\begin{equation*}
  \left(\mathcal{D}_N^{mic}\right)^2(g)(t)\equiv\left(\mathcal{D}_N^{mic}\right)^2(g(t,x,v)) =\sum\limits_{|\alpha|\leqslant N}\left\|\partial_x^{\alpha}\{\mathbf{I}-\mathbf{P}\}g(t,\cdot,\cdot)\right\|_{\nu}^2,
\end{equation*}
the macroscopic dissipation rate as
\begin{equation*}
  \left(\mathcal{D}_N^{mac}\right)^2(g)(t)\equiv\left(\mathcal{D}_N^{mac}\right)^2(g(t,x,v))=\sum\limits_{0<|\alpha|\leqslant N}\left\|\partial_x^{\alpha}\mathbf{P}g(t,\cdot,\cdot)\right\|_{\nu}^2.
\end{equation*}
Clearly, \eqref{1eq16} implies
\begin{equation}\label{1eq17}
  \mathcal{D}_N^{mac}\leqslant C \mathcal{E}_N.
\end{equation}

\subsection{Main results.} Our main results are as follows: the first theorem is about the global existence of the Boltzmann-Fermi-Dirac equation \eqref{1eq8} uniform with respect to the Knudsen number $\epsilon$.
\begin{theorem}\label{th1}
  Let $0<\epsilon<1$. Assume $0\leqslant F_{\epsilon,0}(x,v)=\mu+\epsilon\mu(1-\mu) g_{\epsilon,0}\leqslant 1$, then there are constants $\delta_0>0$, $c_0>0$, independent of $\epsilon$, such that if $\mathcal{E}_N(g_{\epsilon,0})\leqslant \delta_0$, then there exists a unique global solution \begin{equation*}
    g_{\epsilon}\in L^{\infty}\left([0,+\infty);\,H^N(\d x;\,L^2(\mu(1-\mu)\d v))\right)
  \end{equation*}
  to the Cauchy problem \eqref{1eq8}. Furthermore, $0\leqslant F_{\epsilon}=\mu+\epsilon\mu(1-\mu)g_{\epsilon}\leqslant 1$, and we have the following global energy estimate
  \begin{equation}\label{1eq18}
    \sup_{t\geqslant 0}\mathcal{E}_N^2(g_\epsilon)(t)+c_0 \int_{0}^{\infty} \frac{1}{\epsilon^2}\left(\mathcal{D}_N^{mic}\right)^2(g_\epsilon)(t)dt+c_0 \int_{0}^{\infty} \left(\mathcal{D}_N^{mac}\right)^2(g_\epsilon)(t)dt \leqslant \mathcal{E}_N^2(g_{\epsilon,0}).
  \end{equation}
\end{theorem}

The second theorem is about the limit to the  incompressible Navier-Stokes-Fourier equations:
\begin{equation}\label{1eq19}
\left\{
  \begin{array}{l}
    E_2\partial_{t} \mathrm{u}+E_2\mathrm{u} \cdot \nabla_{x} \mathrm{u}+\nabla_{x} p=\nu_{*} \Delta_{x} \mathrm{u}, \\
    \nabla_{x} \cdot \mathrm{u}=0, \\
    C_A\partial_{t} \theta+C_A\mathrm{u} \cdot \nabla_{x} \theta=\kappa_{*} \Delta_{x} \theta,
  \end{array}
\right.
\end{equation}
where
\begin{equation*}
  \nu_{*}=\left\langle\beta_{L}\left(|v|\right) v_{1}^{2} v_{2}^{2}\right\rangle, \qquad
  \kappa_{*}=\left\langle\alpha_{L}\left(|v|\right)\left(\frac{|v|^{2}}{2}-K_{A}\right)^{2} v_{1}^{2}\right\rangle.
\end{equation*}
One can find the definition of constants $E_2$, $C_A$, $K_A$, $K_g$, positive functions $\alpha_{L}\left(|v|\right)$ and $\beta_{L}\left(|v|\right)$ appeared here in section \ref{sec4}.
\begin{theorem}\label{th2}
   Let $0<\epsilon<1$, and $\delta_0>0$ be as in Theorem \ref{th1}. For any $\left(\rho_{0}, \mathrm{u}_{0}, \theta_{0}\right)\in H^{N}\left(\d x\right)$ with $\left\|\left(\rho_{0}, \mathrm{u}_{0}, \theta_{0}\right)\right\|_{H_x^N} < \frac{\delta_0}{2}$, and $g_{\epsilon, 0} \in Null(L)^{\bot}$ with $\left\|\widetilde{g}_{\epsilon, 0}\right\|_{H_x^NL_v^2}<\frac{\delta_0}{2}$, let
   \begin{equation*}
     g_{\epsilon, 0}(x, v)=\left\{\rho_{0}(x)+\mathrm{u}_{0}(x) \cdot v+\theta_{0}(x)\left(\frac{|v|^{2}}{2}-K_g\right)\right\} +\widetilde{g}_{\epsilon, 0}(x, v).
   \end{equation*}
   Let $g_{\epsilon}$ be the family of solutions to the Boltzmann-Fermi-Dirac equation \eqref{1eq8}  constructed in Theorem \ref{th1}. Then,
   \begin{equation*}
     g_{\epsilon} \rightarrow \mathrm{u} \cdot v+\theta\left(\frac{|v|^{2}}{2}-K_A\right), \quad \text { as }\epsilon \rightarrow 0,
   \end{equation*}
   where the convergence is weak-$\star$ for $t$, strongly in $H^{N-\eta}\left(\d x\right)$ for any $\eta>0,$ and weakly in $L^{2}(\mu(1-\mu)\d v),$ and $(\mathrm{u}, \theta) \in C\left([0, \infty) ; H^{N-1}\left(\d x\right)\right) \cap L^{\infty}\left([0, \infty) ; H^{N}\left(\d x\right)\right)$ is the solution of the incompressible Navier-Stokes-Fourier equation \eqref{1eq19} with initial data:
   \begin{equation*}
        \mathrm{u}\big|_{t=0}(x)=\mathcal{P} \mathrm{u}_{0}(x), \quad \theta \big|_{t=0}=\frac{K_g\theta_0(x)-\rho_0(x)}{K_g+1},
  \end{equation*}
  where $\mathcal{P}$ is the Leray projection. Furthermore, the convergence of the moments holds: as $\epsilon \rightarrow 0$
  \begin{equation*}
    \begin{aligned}
        \mathcal{P}\left\langle g_{\epsilon}, v \right\rangle &\rightarrow \mathrm{u}, &\quad \mathrm{ in }\; C\left([0, \infty) ; H^{N-1-\eta}\left(\d x\right)\right) \cap L^{\infty}\left([0, \infty) ; H^{N-\eta}\left(\d x\right)\right), \\
        \left\langle g_{\epsilon},\frac{1}{C_A}\left(\frac{|v|^{2}}{2}-K_A\right) \right\rangle &\rightarrow \theta, &\quad \mathrm{ in }\; C\left([0, \infty) ; H^{N-1-\eta}\left(\d x\right)\right) \cap L^{\infty}\left([0, \infty) ; H^{N-\eta}\left(\d x\right)\right)
    \end{aligned}
  \end{equation*}
for any $\eta>0$.
\end{theorem}

 The paper is organized as follows: we give some basic estimates for the operators $L$, $Q$ and $T$ in the next section; Then prove the global existence of solution to scaled Boltzmann-Fermi-Dirac equation \eqref{1eq7} in section \ref{sec3}; In the last section, we prove the incompressible Navier-Stokes-Fourier limit.

\newpage
\section{Some basic estimates}\label{sec2}
This section is a preparation for energy estimates including microscopic and macroscopic estimates in next section. The operators $L$, $Q$ and $T$ defined by \eqref{1eq9}, \eqref{1eq13} and \eqref{1eq14} are studied respectively.

\subsection{Properties of the linear operator $L$}
The most crucial point in energy estimate is that $L$ is locally coercive, namely, there holds
\begin{equation*}
  \left\langle Lg,g\right\rangle \geqslant \lambda|\{\mathbf{I}-\mathbf{P}\}g|_{\nu}^2,\quad\forall g\in D(L),
\end{equation*}
for some constant $\lambda>0$, where
\begin{equation*}
  D(L)=\left\{g\in L^2(\mu(1-\mu)\d v)\,\big|\,\nu g\in L^2(\mu(1-\mu)\d v)\right\}
\end{equation*}
is the domain of $L$. For convenience, we list some results without proof since they are parallel to the results in perturbed Boltzmann equation \cite{ref24}.

\begin{proposition}\label{prop1}
For the operator $L=\nu-K$, we have the following results:
\begin{itemize}
  \item There are two constants $C_1,\;C_2>0$ such that
  \begin{equation}\label{2eq1}
    C_1(1+|v|)\leqslant \nu(v) \leqslant C_2(1+|v|).
  \end{equation}
  \item The operator
  \begin{equation*}
    K:\,L^2\left(\mu(1-\mu)\d v\right)\longrightarrow L^2\left(\mu(1-\mu)\d v\right)
  \end{equation*}
  is compact.
  \item The linearized collision operator $L$ is symmetric, nonnegative, and its null space is
  \begin{equation}\label{2eq2}
    Null(L)=\operatorname{span}\left\{1,\,v,\,|v|^2\right\}.
  \end{equation}
  \item There exists a constant $\lambda>0$ such that
  \begin{equation}\label{2eq3}
    \langle Lg\;, g \rangle \geqslant \lambda \left|\{\mathbf{I}-\mathbf{P}\}g\right|_{\nu}^2.
  \end{equation}
\end{itemize}
\end{proposition}
We denote the global Maxwellian by
  \begin{equation*}
    M(v)=e^{-\frac{|v|^2}{2}}.
  \end{equation*}
and the collision frequency in perturbed Boltzmann equation by
\begin{equation*}
  \nu_M=\iint\limits_{\mathbb{R}^3\times\mathbb{S}^2}|v_1-v|\,|\cos\theta| M(v_1)\d\omega \d v_1.
\end{equation*}
\eqref{2eq1} holds since $\nu$ is just replacing $M$ in $\nu_M$ with $\frac{\mathcal{N}}{\mu(1-\mu)}$. The compactness of $K$ is a consequence of compactness criterion in \cite{ref44}, and \eqref{2eq3} is the corollary of the three properties above. For more details, see \cite{ref24, ref27, Zakrevskiy}.

\subsection{Estimates for the nonlinear operators $Q$ and $T$}
In this subsection, we devote to $L^2$-type estimates for the bilinear form $Q(f,g)$ and the trilinear form $T(f,g,h)$, which plays an important role in the nonlinear problem. Before stating the estimates, we need a useful Proposition. Recall $v^\prime$ and $v_1^\prime$ as in \eqref{1eq2}, using the pole coordinates with the direction of $v_1-v$ as the positive direction of $Z$-axis, then
\begin{equation*}
  \omega=(\omega_1,\omega_2,\omega_3)=(\sin\theta\cos\varphi,\,\sin\theta\sin\varphi,\,\cos\theta),\quad \theta\in \,[0,\pi],\;\varphi\in\,[0,2\pi],
\end{equation*}
thus the rotation through $\frac{\pi}{2}$:
\begin{equation*}
  \begin{aligned}
  \omega&\mapsto\omega_\bot=(-\cos\theta\cos\varphi,-\cos\theta\sin\varphi,\sin\theta),\\
  \theta&\mapsto\frac{\pi}{2}-\theta,\quad \varphi\mapsto\varphi-\pi,
  \end{aligned}
\end{equation*}
yields:
\begin{proposition}\label{prop2}
  Let $H$ be a measurable function with arguments $v^\prime$ and $v_1^\prime$ defined in \eqref{1eq2}, then
  \begin{equation}\label{2eq4}
  \int_{\mathbb{S}^{2}}|\cos\theta|H(v_1^\prime,v^\prime)\d\omega= \int_{\mathbb{S}^{2}}|\cos\theta|H(v^\prime,v_1^\prime)\d\omega.
  \end{equation}
\end{proposition}

Now we establish the estimates for $Q$ and $T$. We always use the fact that $M$ and $\mu(1-\mu)$ are bounded by each other and the relation \eqref{2eq1}.

\begin{lemma}\label{lem1}
  Let $f,\,g,\,h$ and $\widetilde{h}$ be smooth functions, then for $Q(f,g)$ we have
  \begin{equation}\label{2eq5}
    \left|\big(Q(f,g),\widetilde{h}\big)\right|\leqslant C\int_{\mathbb{R}^3} \Big(|\nu^{1/2}f|_{L_v^2}|g|_{L_v^2}+|f|_{L_v^2}|\nu^{1/2}g|_{L_v^2}\Big)\left|\nu^{1/2}\widetilde{h}\right|_{L_v^2}\d x ,
  \end{equation}
  and
  \begin{equation}\label{2eq6}
    \left\|\big\langle Q(\overline{f},\overline{g}),\widetilde{h}\big\rangle\right\|\leqslant C\sup\limits_{x,v}\left|(\mu(1-\mu))^{\frac{1}{4}}\widetilde{h}\right|\sup\limits_{x}|f|_{L_v^2}\|g\|;
  \end{equation}
  for $T(f,\,g,\,h)$ we have
  \begin{equation}\label{2eq7}
    \left|\big(T(f,g,h),\widetilde{h}\big)\right|\leqslant C\int_{\mathbb{R}^3}\left(|\nu^{1/2}f|_{L_v^2}|g|_{L_v^2}+|f|_{L_v^2}|\nu^{1/2}g|_{L_v^2}\right) |h|_{L_v^2}\left|\nu^{1/2}\widetilde{h}\right|_{L_v^2}\d x,
  \end{equation}
  and
  \begin{equation}\label{2eq8}
    \left\|\left\langle T(\overline{f},\overline{g},\overline{h}),\widetilde{h}\right\rangle\right\|\leqslant  C\sup\limits_{x,v}\left|(\mu(1-\mu))^{\frac{1}{4}}\widetilde{h}\right|\sup\limits_{x}|f|_{L_v^2}\sup\limits_{x}|g|_{L_v^2}\|h\|.
  \end{equation}
  where $(\overline{f},\overline{g})$ and $(\overline{f},\overline{g},\overline{h})$ are permutations of $(f,g)$ and $(f,g,h)$ respectively.
\end{lemma}
\begin{proof}
  For simplicity, we suppress the $x$-dependence in $f(x,v),\,g(x,v),\,h(x,v)$ and $\widetilde{h}(x,v)$, to prove \eqref{2eq5}, it's sufficient to establish the same inequality for $Q_i,\,i=1,2,\cdots,6$. Recall that
  \begin{equation*}
    Q_1(f,g)=\iint\limits_{\mathbb{R}^3\times\mathbb{S}^2}|v_1-v|\,|\cos\theta|\frac{\mathcal{N}}{\mu(1-\mu)} \bigg(1-\mu_1^\prime-\mu^\prime\bigg)f_1^\prime g^\prime \d\omega \d v_1,
  \end{equation*}
  and
  \begin{equation*}
    Q_2(f,g)=-\iint\limits_{\mathbb{R}^3\times\mathbb{S}^2}|v_1-v|\,|\cos\theta|\frac{\mathcal{N}}{\mu(1-\mu)} \bigg(1-\mu_1-\mu\bigg)f_1g\d\omega \d v_1,
  \end{equation*}
  note that
  \begin{equation*}
    \frac{\mathcal{N}}{\mu(1-\mu)} \bigg(1-\mu_1^\prime-\mu^\prime\bigg),\quad\frac{\mathcal{N}}{\mu(1-\mu)} \bigg(1-\mu_1-\mu\bigg)
  \end{equation*}
  is bounded by $M_1$, thus, similar to Lemma 2.3 in \cite{ref26}, we have for $i=1,2$,
  \begin{equation}\label{2eq9}
    \left|\big\langle Q_i(f,g),\widetilde{h}\big\rangle\right|
    \leqslant C\left(|\nu^{1/2}f|_{L_v^2}|g|_{L_v^2}+|f|_{L_v^2}|\nu^{1/2}g|_{L_v^2}\right)|\nu^{1/2}\widetilde{h}|_{L_v^2}.
  \end{equation}

  In the following, we will deal with the nonlinear terms that are different from those in the case of perturbed Boltzmann equation. For
  \begin{equation*}
    Q_3(f,g)=\iint\limits_{\mathbb{R}^3\times\mathbb{S}^2}|v_1-v|\,|\cos\theta|\frac{\mathcal{N}}{\mu(1-\mu)} \bigg(\mu_1^\prime f_1^\prime+\mu^\prime f^\prime \bigg)f\d\omega \d v_1
  \end{equation*}
  we have
  \begin{equation*}
    \begin{aligned}
    \left|\big\langle Q_3(f,g),\widetilde{h}\big\rangle\right|\leqslant &C \iiint\limits_{\mathbb{R}^3\times\mathbb{R}^3\times\mathbb{S}^2}|v_1-v|\,|\cos\theta| M_1MM_1^\prime |f_1^\prime|\,|f|\,|\widetilde{h}|\d\omega \d v_1 \d v\\
    &+C\iiint\limits_{\mathbb{R}^3\times\mathbb{R}^3\times\mathbb{S}^2}|v_1-v|\,|\cos\theta| M_1MM^\prime |f^\prime |\,|f|\,|\widetilde{h}|\d\omega \d v_1 \d v,
    \end{aligned}
  \end{equation*}
  Due to Proposition \ref{prop2}, we need only to estimate the second term on the right hand side above. Using Cauchy-Schwarz inequality to get
  \begin{equation}\label{2eq10}
    \begin{aligned}
    \left|\big\langle Q_3(f,g),\widetilde{h}\big\rangle\right|\leqslant&\left(\iiint_{\mathbb{R}^3\times\mathbb{R}^3\times\mathbb{S}^2}|v_1^\prime-v^\prime| M_1^{\frac{1}{2}}M^{\frac{1}{2}}M_1^\prime |f_1^\prime|^2\d\omega \d v_1^\prime \d v^\prime\right)^{\frac{1}{2}}\\
    &\times\left(\iiint_{\mathbb{R}^3\times\mathbb{R}^3\times\mathbb{S}^2}|v_1-v| M_1^{\frac{3}{2}}M^{\frac{3}{2}}M_1^\prime|f|^2\,|\widetilde{h}|^2\d\omega \d v_1 \d v\right)^{\frac{1}{2}}\\
    \leqslant & C\left(\int_{\mathbb{S}^2}\int_{\mathbb{R}_{v_1^\prime}^3}|f_1^\prime|^2 M_1^\prime\d v_1^\prime \left(\int_{\mathbb{R}_{v^\prime}^3}|v_1^\prime-v^\prime|{M_1^\prime}^\frac{1}{2} {M^\prime}^\frac{1}{2} \d v^\prime\right) \d\omega\right)^{\frac{1}{2}}|g|_{L_v^2}|\nu^{1/2}\widetilde{h}|_{L_v^2}\\
    \leqslant & C|\nu^{1/2}f|_{L_v^2}|g|_{L_v^2}|\nu^{1/2}\widetilde{h}|_{L_v^2}.
    \end{aligned}
  \end{equation}
  where we have used that $|v_1-v|=|v_1^\prime-v^\prime|$ and $|v_1-v| M_1^\frac{1}{2}\leqslant C (1+|v|)$.

  The estimate for $Q_4$ is same as $Q_3$, we also have
  \begin{equation}\label{2eq11}
    \left|\big\langle Q_4(f,g),\widetilde{h}\big\rangle\right|\leqslant C|\nu^{1/2}f|_{L_v^2}|g|_{L_v^2}|\nu^{1/2}\widetilde{h}|_{L_v^2}.
  \end{equation}

  The estimate for $Q_5$ and $Q_6$ need more care. Recall that
  \begin{equation*}
    Q_5(f,g)=\iint\limits_{\mathbb{R}^3\times\mathbb{S}^2}|v_1-v|\,|\cos\theta|\frac{\mathcal{N}}{\mu(1-\mu)} \bigg(\mu_1^\prime f_1^\prime+\mu^\prime f^\prime \bigg)g\d\omega \d v_1
  \end{equation*}
  By Proposition \ref{prop2}, using Cauchy-Schwarz inequality we have
  \begin{equation*}
    \begin{aligned}
    \left|\big\langle Q_5(f,g),\widetilde{h}\big\rangle\right| \leqslant & C\iiint\limits_{\mathbb{R}^3\times\mathbb{R}^3\times\mathbb{S}^2}|v_1-v|\,|\cos\theta| MM_1M_1^\prime |f_1^\prime|\,|g|\,|\widetilde{h}|\d\omega \d v_1\d v\\
    &\quad+C\iiint\limits_{\mathbb{R}^3\times\mathbb{R}^3\times\mathbb{S}^2}|v_1-v|\,|\cos\theta|MM_1M^\prime |f^\prime |\,|g|\,|\widetilde{h}|\d\omega \d v_1\d v\\
    \leqslant & C\int\limits_{\mathbb{R}^3}|g|\,|\widetilde{h}|M\d v\left(\int_{\mathbb{R}^3}M_1\d v_1\right)^{\frac{1}{2}}\left(\iint_{\mathbb{R}^3\times\mathbb{S}^2} |v_1-v|^2\,|\cos\theta|^2M_1{M^\prime}^2|f^\prime |^2\d v_1\right)^{\frac{1}{2}},
    \end{aligned}
  \end{equation*}
  using the notation
  \begin{equation}\label{2eq12}
    V_{\|}=(V\cdot\omega)\omega,\qquad V_{\bot}=V-(V\cdot\omega)\omega,
  \end{equation}
  then the variable changing $v_1-v\mapsto V$ and the fact (see \cite{ref24})
  \begin{equation}\label{2eq13}
    d\omega dV= \frac{2dV_{\|}dV_{\bot}}{|V_{\|}|^2},
  \end{equation}
  implies
  \begin{equation}\label{2eq14}
    \begin{aligned}
    \left|\big\langle Q_5(f,g),\widetilde{h}\big\rangle\right| \leqslant&C\int_{\mathbb{R}_v^3}|g|\,|\widetilde{h}|M\d v\\
    &\times\left(\int\limits_{\mathbb{R}_{V_{\|}}^3}|V_{\|}|^2\left[\int_{\mathbb{R}_{V_{\bot}}^2} M(v+V)M(v+V_{\|})\d V_{\bot}\right]|f(v+V_{\|})|^2\frac{M(v+V_{\|})\d V_{\|}}{|V_{\|}|^2} \right)^{\frac{1}{2}}\\
    \leqslant &C \int_{\mathbb{R}_v^3}|g|\,|\widetilde{h}|M\d v \left(\int_{\mathbb{R}_{V_{\|}}^3}\left[\int_{\mathbb{R}_{V_{\bot}}^2} M^{\frac{1}{4}}(V_{\bot})\d V_{\bot}\right]|f(v+V_{\|})|^2M(v+V_{\|})\d V_{\|} \right)^{1/2}\\
    \leqslant &C |f|_{L_v^2}|g|_{L_v^2}|\widetilde{h}|_{L_v^2}.
    \end{aligned}
  \end{equation}
  here we have used the fact
  \begin{equation*}
    \frac{1}{2}|v+V|^2+\frac{1}{2}|v+V_{\|}|^2\geqslant \frac{1}{8}|V_{\bot}|^2.
  \end{equation*}

  Similar to $Q_5$, we also have
  \begin{equation}\label{2eq15}
    \left|\big\langle Q_6(f,g),\widetilde{h}\big\rangle\right|\leqslant C|f|_{L_v^2}|g|_{L_v^2}|\widetilde{h}|_{L_v^2}.
  \end{equation}
  Then by summing up \eqref{2eq9}-\eqref{2eq11}, \eqref{2eq14}-\eqref{2eq15} and further integrating over $\mathbb{R}_x^3$, we complete the proof of \eqref{2eq5}.

  To prove \eqref{2eq6}, we need only to put the weight function $\nu(v)$ on the function $\widetilde{h}$ when we deal with $\left\langle Q_1(f,g),\widetilde{h}\right\rangle,\,\left\langle Q_2(f,g),\widetilde{h}\right\rangle,\,\cdots,\left\langle Q_6(f,g),\widetilde{h}\right\rangle$. We can bound $\left\langle Q(f,g),\widetilde{h}\right\rangle$ by
  \begin{equation*}
    \begin{aligned}
    &C\left[\int_{\mathbb{R}^3}|f(x,v)|^2\mu(1-\mu)\d v\right]^{1/2}\left[\int_{\mathbb{R}^3}|g(x,v)|^2\mu(1-\mu)\d v\right]^{1/2} \left[\int_{\mathbb{R}^3}\nu^2(v)|\widetilde{h}(x,v)|^2\mu(1-\mu)\d v\right]^{1/2}\\
    &\leqslant C \sup\limits_{x,v}\left|\left(\mu(1-\mu)\right)^{\frac{1}{4}}\widetilde{h}(x,v)\right|\left[\int_{\mathbb{R}^3}|f(x,v)|^2\mu(1-\mu)\d v\right]^{1/2} \left[\int_{\mathbb{R}^3}|g(x,v)|^2\mu(1-\mu)\d v\right]^{1/2}.
    \end{aligned}
  \end{equation*}
  We conclude that \eqref{2eq6} holds by integrating over $x$ and taking $L^2$ and $L^\infty$ norm in $x$ for the last two factors.

    We then turn to study the trilinear form $T(f,g,h)$. For $T_1(f,g,h)$, we have
  \begin{equation*}
  \begin{aligned}
  |T_1(f,g,h)|=&\left|\iint_{\mathbb{R}^3\times\mathbb{S}^2}|v_1-v|\,|\cos\theta| \frac{\mathcal{N}\mu_1}{1-\mu} f_1g \bigg(h_1^\prime+h^\prime\bigg)\d\omega \d v_1\right|\\
  \leqslant & C\iint\limits_{\mathbb{R}^3\times\mathbb{S}^2}|v_1-v|\,|\cos\theta| MM_1^2|f_1gh_1^\prime|\d\omega \d v_1\\
  &+C\iint\limits_{\mathbb{R}^3\times\mathbb{S}^2}|v_1-v|\,|\cos\theta|MM_1^2|f_1 g h^\prime|\d\omega \d v_1\\
  &\stackrel{\triangle}{=}I_1^1+I_1^2.
  \end{aligned}
  \end{equation*}
  By Proposition \ref{prop2}, it's sufficient to estimate $I_1^2$. We use Cauchy-Schwarz inequality to get
\begin{equation*}
\begin{aligned}
  I_1^2\leqslant &\left(\iint_{\mathbb{R}^3\times\mathbb{S}^2}|v_1-v|\,|\cos\theta| M^{\frac{1}{2}}M_1^2|f_1g|^2\d\omega \d v_1\right)^{1/2}\times\\
  &\left(\iint_{\mathbb{R}^3\times\mathbb{S}^2}|v_1-v|\,|\cos\theta| M^{\frac{3}{2}}M_1^2|h^\prime|^2\d\omega \d v_1\right)^{1/2}.
\end{aligned}
\end{equation*}
Note that $|v_1-v|M^{\frac{1}{2}}M_1$ is bounded, the first factor above is bounded by $|f|_{L_v^2}|g|$. The second factor need more careful estimate, note that $MM_1=M^\prime M_1^\prime$, the variable changing $v_1-v\mapsto V$ and \eqref{2eq12}, \eqref{2eq13} imply the second factor is bounded by
\begin{equation*}
\begin{aligned}
  &\left(\iint_{\mathbb{R}^3\times\mathbb{R}^2}|V_{\|}|M^{1/2}M(v+V)|h(v+V_{\|})|^2\frac{2M(v+V_{\|})}{|V_{\|}|^2}\d V_{\bot}\d V_{\|}\right)^{1/2}\\
  \leqslant & C\left(\int_{\mathbb{R}_{V_{\|}}^3}\frac{1}{|V_{\|}|}M^{\frac{1}{4}}(V_{\|})|h(v+V_{\|})|^2M(v+V_{\|})\d V_{\|} \int_{\mathbb{R}_{V_{\bot}}^2}M^{\frac{1}{4}}(V_{\bot})\d V_{\bot}\right)^{1/2}\\
  \leqslant & C |h|_{L_v^2},
\end{aligned}
\end{equation*}
here we have used $M^{1/2}(v)M(v+V)\leqslant M^{\frac{1}{4}}(V)$. Then
\begin{equation*}
  I_1^2 \leqslant C |f|_{L_v^2}|h|_{L_v^2}|g|,
\end{equation*}
hence
\begin{equation}\label{2eq16}
\left|\left\langle T_1(f,g,h), \widetilde{h}\right\rangle\right|\leqslant C |f|_{L_v^2}|g|_{L_v^2}|h|_{L_v^2}|\widetilde{h}|_{L_v^2}.
\end{equation}

For $T_2(f,g,h)$, we have
\begin{equation*}
\begin{aligned}
  |T_2(f,g,h)|=&\left|\iint_{\mathbb{R}^3\times\mathbb{S}^2}|v_1-v|\,|\cos\theta| \frac{\mathcal{N}\mu_1^\prime\mu^\prime}{\mu(1-\mu)} f_1^\prime g^\prime  \bigg(h_1+h\bigg)\d\omega \d v_1\right|\\
  \leqslant & C\iint\limits_{\mathbb{R}^3\times\mathbb{S}^2}|v_1-v|\,|\cos\theta| M_1M^\prime M_1^\prime|h_1||f_1^\prime g^\prime |\d\omega \d v_1\\
  &+C\iint\limits_{\mathbb{R}^3\times\mathbb{S}^2}|v_1-v|\,|\cos\theta| M_1M^\prime M_1^\prime|h||f_1^\prime g^\prime |\d\omega \d v_1\\
  &\stackrel{\triangle}{=}I_2^1+I_2^2.
\end{aligned}
\end{equation*}
Next we use different methods to estimate $\left\langle I_2^1, \widetilde{h}\right\rangle$ and $\left\langle I_2^2, \widetilde{h}\right\rangle$ respectively. For $\left\langle I_2^1, \widetilde{h}\right\rangle$, we use Cauchy-Schwarz inequality to get
\begin{equation}\label{2eq17}
  \begin{aligned}
  \left|\left\langle I_2^1, \widetilde{h}\right\rangle\right|\leqslant & C
  \iiint\limits_{\mathbb{R}^3\times\mathbb{R}^3\times\mathbb{S}^2}|v_1-v|\,|\cos\theta| |f_1^\prime g^\prime ||h_1||\widetilde{h}|MM_1M^\prime M_1^\prime\d\omega \d v_1\d v\\
  \leqslant & C
  \left(\iiint_{\mathbb{R}^3\times\mathbb{R}^3\times\mathbb{S}^2} |f_1^\prime|^2|g^\prime |^2 M^\prime M_1^\prime\d\omega\d v_1^\prime \d v^\prime\right)^{1/2}\times\\
  &\left(\iiint_{\mathbb{R}^3\times\mathbb{R}^3\times\mathbb{S}^2} |h_1|^2|\widetilde{h}|^2MM_1\d\omega \d v_1\d v\right)^{1/2}\\
  \leqslant & C |f|_{L_v^2}|g|_{L_v^2}|h|_{L_v^2}|\widetilde{h}|_{L_v^2},
  \end{aligned}
\end{equation}
here we have used that $|v_1-v|MM_1$ is bounded and $\d\omega \d v_1^\prime \d v^\prime=\d\omega \d v_1\d v$. We further estimate $I_2^2$ then $\left\langle I_2^2, \widetilde{h}\right\rangle$. Using Cauchy-Schwarz inequality again to get
\begin{equation*}
  \begin{aligned}
  I_2^2\leqslant & C \left(\iint_{\mathbb{R}^3\times\mathbb{S}^2}|v_1-v|\,|\cos\theta| M_1{M^\prime}^\frac{3}{4} {M_1^\prime}^\frac{5}{4}|f_1^\prime|^2\d\omega \d v_1\right)^{1/2}\times\\
  &\left(\iint_{\mathbb{R}^3\times\mathbb{S}^2}|v_1-v|\,|\cos\theta| M_1{M^\prime}^\frac{5}{4} {M_1^\prime}^\frac{3}{4}|g^\prime |^2\d\omega \d v_1\right)^{1/2}|h|.
  \end{aligned}
\end{equation*}
 By Proposition \ref{prop2}, we need only to estimate the second factor on the right hand side above, then using variable changing $v_1-v\mapsto V$ and \eqref{2eq12}, \eqref{2eq13}, it's bounded by
\begin{equation*}
  \begin{aligned}
  &\left(\iint_{\mathbb{R}^3\times\mathbb{S}^2}|v_1-v|\,|\cos\theta| M^{\frac{1}{4}}M_1^{\frac{5}{4}}|f^\prime |^2M^\prime\d\omega \d v_1\right)^{1/2}\\
  \leqslant & C \left(\iint_{\mathbb{R}^3\times\mathbb{R}^2}|V_{\|}| M^{\frac{1}{4}}M^{\frac{5}{4}}(v+V)|f(v+V_{\|})|^2\frac{2M(v+V_{\|})}{|V_{\|}|^2}\d V_{\bot}\d V_{\|}\right)^{1/2}\\
  \leqslant & C\left(\int_{\mathbb{R}^3}\frac{1}{|V_{\|}|}M^{\frac{1}{8}}(V_{\|})|f(v+V_{\|})|^2M(v+V_{\|})\d V_{\|} \int_{\mathbb{R}^2}M^{\frac{1}{8}}(V_{\bot})\d V_{\bot}\right)^{1/2}\\
  \leqslant & C|f|_{L_v^2},
  \end{aligned}
\end{equation*}
 where we have used the fact $M^{\frac{1}{4}}M^{\frac{5}{4}}(v+V)\leqslant M^{\frac{1}{4}}(V)$. The estimate for the second factor is same as above. Then we get
\begin{equation}\label{2eq18}
    \left|\left\langle I_2^2, \widetilde{h}\right\rangle\right|\leqslant C|f|_{L_v^2}|g|_{L_v^2}|h|_{L_v^2}|\widetilde{h}|_{L_v^2},
\end{equation}
hence the estimates \eqref{2eq17} and \eqref{2eq18} imply that
\begin{equation}\label{2eq19}
  \left|\left\langle T_2(f,g,h), \widetilde{h}\right\rangle\right|\leqslant C |f|_{L_v^2}|g|_{L_v^2}|h|_{L_v^2}|\widetilde{h}|_{L_v^2}.
\end{equation}

For $T_3(f,g,h)$, we have
\begin{equation*}
\begin{aligned}
  |T_3(f,g,h)|=&\left|\iint_{\mathbb{R}^3\times\mathbb{S}^2}|v_1-v|\,|\cos\theta| \frac{\mathcal{N}}{\mu(1-\mu)}\mu_1\mu_1^\prime \bigg(f_1^\prime g^\prime h_1- f_1gh_1^\prime\bigg)\d\omega \d v_1\right|\\
  \leqslant & C\iint\limits_{\mathbb{R}^3\times\mathbb{S}^2}|v_1-v|\,|\cos\theta| M_1^2M_1^\prime|f_1^\prime g^\prime ||h_1|\d\omega \d v_1\\
  &+C\iint\limits_{\mathbb{R}^3\times\mathbb{S}^2}|v_1-v|\,|\cos\theta| M_1^2M_1^\prime|f_1g||h_1^\prime|\d\omega \d v_1\\
  &\stackrel{\triangle}{=}I_3^1+I_3^2.
\end{aligned}
\end{equation*}
Then for $I_3^1$, we use Cauchy-Schwarz inequality to get
\begin{equation}\label{2eq20}
  \begin{aligned}
  \left|\left\langle I_3^1, \widetilde{h}\right\rangle\right|\leqslant & C
  \left(\iiint_{\mathbb{R}^3\times\mathbb{R}^3\times\mathbb{S}^2} |v_1-v|\,|\cos\theta|M_1 |f_1^\prime|^2|g^\prime |^2 M^\prime M_1^\prime\d\omega \d v_1^\prime \d v^\prime\right)^{1/2}\\
  &\left(\iiint_{\mathbb{R}^3\times\mathbb{R}^3\times\mathbb{S}^2} |v_1-v|\,|\cos\theta|M_1 |h_1|^2|\widetilde{h}|^2MM_1\d\omega \d v_1 \d v\right)^{1/2}\\
  \leqslant & C \left(\iiint_{\mathbb{R}^3\times\mathbb{R}^3\times\mathbb{S}^2} (\nu(v_1^\prime)+\nu(v^\prime))|f_1^\prime|^2|g^\prime |^2M^\prime M_1^\prime\d\omega \d v_1^\prime \d v^\prime\right)^{1/2}\\
  &\left(\iiint_{\mathbb{R}^3\times\mathbb{R}^3\times\mathbb{S}^2}\nu(v)|h_1|^2|\widetilde{h}|^2MM_1\d\omega \d v_1 \d v\right)^{1/2}\\
  \leqslant &C \left(|\nu^{1/2}f|_{L_v^2}|g|_{L_v^2}+|f|_{L_v^2}|\nu^{1/2}g|_{L_v^2}\right) |h|_{L_v^2}|\nu^{1/2}\widetilde{h}|_{L_v^2},
  \end{aligned}
\end{equation}
where we have used $|v_1-v|M_1\leqslant C(1+|v|)\leqslant C(\nu(v_1^\prime)+\nu(v^\prime))$.
For $I_3^2$, by Proposition \ref{prop2}, using Cauchy-Schwarz inequality again to get
\begin{equation*}
  \begin{aligned}
  I_3^2\leqslant & C \left(\iint_{\mathbb{R}^3\times\mathbb{S}^2}|f_1|^2M_1 \d\omega \d v_1\right)^{1/2} \left(\iint_{\mathbb{R}^3\times\mathbb{S}^2}(|v_1-v||\cos\theta| )^2{M^\prime}^2M_1^3 |h^\prime|^2 \d\omega \d v_1\right)^{1/2}|g|\\
  \leqslant & C|f|_{L_v^2}|g|\left(\iint_{\mathbb{R}^3\times\mathbb{R}^2}|V_{\|}|^2 M(v+V_{\|})M^3(v+V)|h(v+V_{\|})|^2\frac{M(v+V_{\|})}{|V_{\|}|^2}\d V_{\bot}\d V_{\|}\right)^{1/2}\\
  \leqslant & C|f|_{L_v^2}|h|_{L_v^2}|g|,
  \end{aligned}
\end{equation*}
here we have used $M(v+V_{\|})M^3(v+V)\leqslant M^{\frac{1}{2}}(V_{\bot})$. Then
\begin{equation}\label{2eq21}
    \left|\left\langle I_3^2, \widetilde{h}\right\rangle\right|\leqslant C|f|_{L_v^2}|g|_{L_v^2}|h|_{L_v^2}|\widetilde{h}|_{L_v^2},
\end{equation}
hence by \eqref{2eq20} and \eqref{2eq21} we get
\begin{equation}\label{2eq22}
  \left|\left\langle T_3(f,g,h), \widetilde{h}\right\rangle\right|\leqslant C\left(|\nu^{1/2}f|_{L_v^2}|g|_{L_v^2}+|f|_{L_v^2}|\nu^{1/2}g|_{L_v^2}\right) |h|_{L_v^2}|\nu^{1/2}\widetilde{h}|_{L_v^2}.
\end{equation}

For $T_4(f,g,h)$, we have
\begin{equation*}
\begin{aligned}
  |T_4(f,g,h)|=&\left|\iint_{\mathbb{R}^3\times\mathbb{S}^2}|v_1-v|\,|\cos\theta| \frac{\mathcal{N}\mu_1^\prime}{1-\mu}\bigg(f_1^\prime g^\prime h -f_1gh_1^\prime\bigg)\d\omega \d v_1\right|\\
  \leqslant & C\iint\limits_{\mathbb{R}^3\times\mathbb{S}^2}|v_1-v|\,|\cos\theta| MM_1M_1^\prime|f_1^\prime g^\prime ||h|\d\omega \d v_1\\
  &+C\iint\limits_{\mathbb{R}^3\times\mathbb{S}^2}|v_1-v|\,|\cos\theta| MM_1M_1^\prime|f_1g||h_1^\prime|\d\omega \d v_1\\
  \stackrel{\triangle}{=}&I_4^1+I_4^2.
\end{aligned}
\end{equation*}

Using the same way to treat $I_4^1$ and $I_4^2$ as to treat $I_2^2$, we get
\begin{equation*}
  I_4^1\leqslant C|f|_{L_v^2}|g|_{L_v^2}|h|,\qquad I_4^2\leqslant C|f|_{L_v^2}|h|_{L_v^2}|g|.
\end{equation*}
Hence
\begin{equation}\label{2eq23}
  \begin{aligned}
  \left|\left\langle T_4(f,g,h), \widetilde{h}\right\rangle\right|\leqslant  C
   |f|_{L_v^2}|g|_{L_v^2}|h|_{L_v^2}|\widetilde{h}|_{L_v^2}.
  \end{aligned}
\end{equation}

For the rest terms in $T(f,g,h)$, we have the following relations:
\begin{itemize}
  \item The estimate for $T_5(f,g,h)$ is similar to the estimate for $T_3(f,g,h)$, we have
  \begin{equation}\label{2eq24}
  \left|\left\langle T_5(f,g,h), \widetilde{h}\right\rangle\right|\leqslant C\left(|\nu^{1/2}f|_{L_v^2}|g|_{L_v^2}+|f|_{L_v^2}|\nu^{1/2}g|_{L_v^2}\right) |h|_{L_v^2}|\nu^{1/2}\widetilde{h}|_{L_v^2}.
\end{equation}
  \item The estimate for $T_6(f,g,h)$ is similar to the estimate for $T_4(f,g,h)$, we have
  \begin{equation}\label{2eq25}
  \left|\left\langle T_6(f,g,h), \widetilde{h}\right\rangle\right|\leqslant  C
   |f|_{L_v^2}|g|_{L_v^2}|h|_{L_v^2}|\widetilde{h}|_{L_v^2}.
\end{equation}
  \item The estimate for $T_7(f,g,h)$ is similar to the estimate for $T_1(f,g,h)$, we have
  \begin{equation}\label{2eq26}
  \left|\left\langle T_7(f,g,h), \widetilde{h}\right\rangle\right|\leqslant  C
   |f|_{L_v^2}|g|_{L_v^2}|h|_{L_v^2}|\widetilde{h}|_{L_v^2}.
   \end{equation}
  \item The estimate for $T_8(f,g,h)$ is similar to the estimate for $T_2(f,g,h)$, we have
  \begin{equation}\label{2eq27}
  \left|\left\langle T_8(f,g,h), \widetilde{h}\right\rangle\right|\leqslant C\left(|\nu^{1/2}f|_{L_v^2}|g|_{L_v^2}+|f|_{L_v^2}|\nu^{1/2}g|_{L_v^2}\right) |h|_{L_v^2}|\nu^{1/2}\widetilde{h}|_{L_v^2}.
  \end{equation}
\end{itemize}
Thus the relations above combining with \eqref{2eq16}, \eqref{2eq19}, \eqref{2eq22} and \eqref{2eq23} imply that
\begin{equation}\label{2eq28}
  \left|\left\langle T(f,g,h), \widetilde{h}\right\rangle\right|\leqslant C\left(|\nu^{1/2}f|_{L_v^2}|g|_{L_v^2}+|f|_{L_v^2}|\nu^{1/2}g|_{L_v^2}\right) |h|_{L_v^2}|\nu^{1/2}\widetilde{h}|_{L_v^2},
\end{equation}
we get \eqref{2eq7} by further integrating over $x$. To prove \eqref{2eq8}, we need only to put the weight function $\nu(v)$ on the function $\widetilde{h}$ in the estimates for $T_1(f,g,h),\cdots,T_8(f,g,h)$, we can bound $\left\langle T(f,g,h), \widetilde{h}\right\rangle$ by
  \begin{equation*}
    \begin{aligned}
    &C\left[\int_{\mathbb{R}^3}|f(x,v)|^2\mu(1-\mu)\d v\right]^{1/2}\left[\int_{\mathbb{R}^3}|g(x,v)|^2\mu(1-\mu)\d v\right]^{1/2} \\
    &\times\left[\int_{\mathbb{R}^3}|h(x,v)|^2\mu(1-\mu)\d v\right]^{1/2}\left[\int_{\mathbb{R}^3}\nu^2(v)|\widetilde{h}(x,v)|^2\mu(1-\mu)\d v\right]^{1/2}\\
    \leqslant &C \sup\limits_{x,v}\left(\left(\mu(1-\mu)\right)^{\frac{1}{4}}|\widetilde{h}(x,v)|\right)\left[\int_{\mathbb{R}^3}|f(x,v)|^2\mu(1-\mu)\d v\right]^{1/2} \\
    &\times\left[\int_{\mathbb{R}^3}|g(x,v)|^2\mu(1-\mu)\d v\right]^{1/2}\left[\int_{\mathbb{R}^3}|h(x,v)|^2\mu(1-\mu)\d v\right]^{1/2}
    \end{aligned}
  \end{equation*}
  We conclude that \eqref{2eq8} holds by integrating over $\mathbb{R}_x^3$ and taking $L^2$ and $L^\infty$ norm in $x$ for the last three factors.
\end{proof}

As a corollary, we want to give higher $x$-derivative estimates of nonlinear terms. If we apply the macro-micro decomposition as in \eqref{1eq15} to $f$ and $g$, and plug them into \eqref{2eq5} and \eqref{2eq7}, then by the Sobolev embedding $H^2(\d x)\hookrightarrow L^\infty(\d x)$ we have
\begin{equation}\label{2eq29}
    \begin{aligned}
    \left|\big(Q(\overline{f},\overline{g}),\widetilde{h}\big)\right| \leqslant &C\left[
    \|f\|_{H_x^2L_v^2}\Big(\|\mathbf{P}g\|_{L_{x,v}^2} +\|\nu^{1/2}\{\mathbf{I}-\mathbf{P}\}g\|_{L_{x,v}^2}\Big)\right.\\
    &\left.\qquad+ \|\nu^{1/2}\{\mathbf{I}-\mathbf{P}\}f\|_{H_x^2L_v^2}\|g\|_{L_{x,v}^2}\right] \|\nu^{1/2}\widetilde{h}\|_{L_{x,v}^2},
    \end{aligned}
  \end{equation}
and
\begin{equation}\label{2eq30}
  \begin{aligned}
    \left|\big(T(\overline{f},\overline{g},h),\widetilde{h}\big)\right| \leqslant &C\left[
    \|f\|_{H_x^2L_v^2}\Big(\|\mathbf{P}g\|_{L_{x,v}^2} +\|\nu^{1/2}\{\mathbf{I}-\mathbf{P}\}g\|_{L_{x,v}^2}\Big)\right.\\
    &\left.\qquad+ \|\nu^{1/2}\{\mathbf{I}-\mathbf{P}\}f\|_{H_x^2L_v^2} \|g\|_{L_{x,v}^2}\right]\|h\|_{H_x^2L_v^2} \|\nu^{1/2}\widetilde{h}\|_{L_{x,v}^2},
  \end{aligned}
  \end{equation}
where $(\overline{f},\overline{g})$ is a permutation of $(f,g)$. If in the estimate of $\big(T(\overline{f},\overline{g},h),\widetilde{h}\big)$, we apply the Sobolev embedding $H^2(\d x)\hookrightarrow L^\infty(\d x)$ to $f$ and $g$, we also have
\begin{equation}\label{2eq31}
  \begin{aligned}
    \left|\big(T(\overline{f},\overline{g},h),\widetilde{h}\big)\right| \leqslant &C\Big(\|f\|_{H_x^2L_v^2}\|g\|_{H_x^2L_v^2}+ \|f\|_{H_x^2L_v^2}\|\nu^{1/2}\{\mathbf{I}-\mathbf{P}\}g\|_{H_x^2L_v^2}\\ &\qquad\qquad+\|\nu^{1/2}\{\mathbf{I}-\mathbf{P}\}f\|_{H_x^2L_v^2}\|g\|_{H_x^2L_v^2}\Big) \|h\|_{L_{x,v}^2} \|\nu^{1/2}\widetilde{h}\|_{L_{x,v}^2},
  \end{aligned}
\end{equation}

\section{The existence of global solutions}\label{sec3}
We use the standard strategy to obtain the global existence of solution to \eqref{1eq8}: construct the local solutions for a sequence of iterating approximate equations and take limit to get local solution to \eqref{1eq8}; then obtain a uniform energy estimate; finally, the use of the continuation argument. The uniform energy estimate consists of two parts: one part is to obtain microscopic dissipation rate, the more complicated part is to deal with macroscopic dissipation rate.
\subsection{Construction of local solutions}
In order to prove the existence of nonlinear problem \eqref{1eq8}, we shall firstly study the following linear Cauchy problem:
\begin{equation}\label{3eq1}
\left\{
  \begin{aligned}
  \partial_t g + \frac{1}{\epsilon}v\cdot \nabla_x g +\frac{1}{\epsilon^2}Lg &=\frac{1}{\epsilon}Q(f,f)+T(f,f,f),\\
  g(t,x,v)|_{t=0}&=g_0(x,v),
  \end{aligned}
\right.
\end{equation}
where $f$ is a given function. For \eqref{3eq1}, we have the following result:
\begin{lemma}\label{lem5}
  For $0<\epsilon<1$, assume that
  \begin{equation*}
    g_0\in H^N(\d x;\,L^2(\mu(1-\mu)\d v)),
  \end{equation*}
  and that for some $T>0$, $f$ satisfies
  \begin{equation*}
    \sup_{0\leqslant t \leqslant T}\mathcal{E}_N^2(f)(t)+\int_{0}^{T} \left(\mathcal{D}_N^{mic}\right)^2(f)(t)dt < +\infty.
  \end{equation*}
  Then, the Cauchy problem \eqref{3eq1} admits a unique solution:
  \begin{equation*}
    g(t, x, v)\in L^{\infty}\left([0,T];\,H^N(\d x;\,L^2(\mu(1-\mu)\d v))\right)
  \end{equation*}
\end{lemma}
\begin{proof}
  We prove the existence of a solution to the Cauchy problem \eqref{3eq1} by the Hahn-Banach theorem. We define the linear operator $\mathcal{T}$ by
  \begin{equation*}
    \mathcal{T}g\equiv \partial_t g + \frac{1}{\epsilon}v\cdot \nabla_x g +\frac{1}{\epsilon^2}Lg.
  \end{equation*}
  Then, we rewrite \eqref{3eq1} into the following form:
  \begin{equation*}
    \mathcal{T}g = \frac{1}{\epsilon}Q(f,f)+T(f,f,f), \quad g(0)=g_0.
  \end{equation*}

  For $h\in C^{\infty}\left([0,T];\,\mathcal{S}(\mathbb{R}_{x,v}^6)\right)$ with $h(T)=0$, we define $\mathcal{T}^*$ through
  \begin{equation*}
    (g,\mathcal{T}^*h)_{L^2\left([0,T];\,H_x^NL_v^2\right)} =(\mathcal{T}g,h)_{L^2\left([0,T];\,H_x^NL_v^2\right)},
  \end{equation*}
  so that $\mathcal{T}^*$ is the adjoint of the operator $\mathcal{T}$ in the Hilbert space $L^2\left([0,T];\,H^N(\d x;\,L^2(\mu(1-\mu)\d v))\right)$.

  We set
  \begin{equation*}
    \mathbb{W}=\left\{w=\mathcal{T}^*h\,\big|\, h\in C^{\infty}\left([0,T];\,\mathcal{S}(\mathbb{R}_{x,v}^6)\right) \text{ with } h(T)=0\right\},
  \end{equation*}
  which is a dense subspace of $L^2\left([0,T];\,H^N(\d x;\,L^2(\mu(1-\mu)\d v))\right)$. We also have that
  \begin{equation*}
    \mathcal{T}^*h=-\partial_t h -\frac{1}{\epsilon}v\cdot \nabla_x h +\frac{1}{\epsilon^2}Lh.
  \end{equation*}

  Then
  \begin{equation*}
    (h,\mathcal{T}^*h)_{H_x^NL_v^2}=-\frac{1}{2}\frac{\d}{\d t}\|h(t)\|_{H_x^NL_v^2}-\frac{1}{\epsilon}(v\cdot \nabla_x h,h)_{H_x^NL_v^2}+ \frac{1}{\epsilon^2}(Lh,h)_{H_x^NL_v^2}.
  \end{equation*}
  Note that the second term above vanishes, and by the estimate \eqref{2eq3}, we have that
  \begin{equation*}
    \int_{t}^{T}(h,\mathcal{T}^*h)_{H_x^NL_v^2}dt\geqslant \frac{1}{2}\|h(t)\|_{H_x^NL_v^2}^2+\frac{\lambda}{\epsilon^2} \int_{t}^{T}\left(\mathcal{D}_N^{mic}\right)^2(h)(s) ds.
  \end{equation*}
  Thus, for all $0<t<T$,
  \begin{equation*}
    \begin{aligned}
    &\|h(t)\|_{H_x^NL_v^2}^2+\frac{2\lambda}{\epsilon^2} \int_{t}^{T}\left(\mathcal{D}_N^{mic}\right)^2(h)(s) ds\\
    \leqslant & 2\|\mathcal{T}^*h\|_{L^2\left([t,T];\,H_x^NL_v^2\right)} \|h\|_{L^2\left([t,T];\,H_x^NL_v^2\right)}.
    \end{aligned}
  \end{equation*}
  This implies that
  \begin{equation}\label{3eq2}
    \begin{aligned}
    &\|h(t)\|_{L^{\infty}\left([0,T];\,H_x^NL_v^2\right)}^2+\frac{2\lambda}{\epsilon^2} \int_{t}^{T}\left(\mathcal{D}_N^{mic}\right)^2(h)(s) ds\\
    \leqslant & 2\|\mathcal{T}^*h\|_{L^2\left([0,T];\,H_x^NL_v^2\right)} \|h\|_{L^2\left([0,T];\,H_x^NL_v^2\right)}\\
    \leqslant & 2\sqrt{T}\|\mathcal{T}^*h\|_{L^2\left([0,T];\,H_x^NL_v^2\right)} \|h\|_{L^\infty\left([0,T];\,H_x^NL_v^2\right)}.
    \end{aligned}
  \end{equation}
  This immediately gives
  \begin{equation}\label{3eq3}
    \|h\|_{L^\infty\left([0,T];\,H_x^NL_v^2\right)}\leqslant 2\sqrt{T}\|\mathcal{T}^*h\|_{L^2\left([0,T];\,H_x^NL_v^2\right)}.
  \end{equation}
  Furthermore, \eqref{3eq2} can be written as follows: let $\mathcal{Y}=\|h\|_{L^\infty\left([0,T];\,H_x^NL_v^2\right)}$,
  \begin{equation*}
  \begin{aligned}
    \frac{\lambda}{\epsilon^2} \int_{t}^{T}\left(\mathcal{D}_N^{mic}\right)^2(h)(s) ds \leqslant&\sqrt{T}\|\mathcal{T}^*h\|_{L^2\left([0,T];\,H_x^NL_v^2\right)} \mathcal{Y}-\frac{1}{2}\mathcal{Y}^2\\
    \leqslant & \frac{T}{2}\|\mathcal{T}^*h\|_{L^2\left([0,T];\,H_x^NL_v^2\right)}^2.
  \end{aligned}
  \end{equation*}
  Thus,
  \begin{equation}\label{3eq4}
    \frac{1}{\epsilon}\left(\int_{t}^{T}\left(\mathcal{D}_N^{mic}\right)^2(h)(s) ds\right)^{1/2}\leqslant  \frac{\sqrt{T}}{\sqrt{2\lambda}}\|\mathcal{T}^*h\|_{L^2\left([0,T];\,H_x^NL_v^2\right)}.
  \end{equation}

  Next, we define a functional $G$ on $\mathbb{W}$ as follows:
  \begin{equation*}
    G(w)=\big(\frac{1}{\epsilon}Q(f,f)+T(f,f,f),h\big)_{L^2\left([0,T];\,H_x^NL_v^2\right)} +\big(g_0,h(0)\big)_{H_x^NL_v^2},
  \end{equation*}
  note that $0<\epsilon<1$, using \eqref{3eq19} and \eqref{3eq26}, we have the estimate
  \begin{equation*}
    \begin{aligned}
    |G(w)|\leqslant& \frac{1}{\epsilon}\left\{\int_{0}^{T}\left(\mathcal{E}_N(f)+\mathcal{E}_N^2(f)\right) \left(\mathcal{E}_N(f)+\mathcal{D}_N^{mic}(f) \right) \mathcal{D}_N^{mic}(h) dt\right\}+\|g_0\|_{H_x^NL_v^2}\|h(0)\|_{H_x^NL_v^2} \\
    \leqslant &\frac{C}{\epsilon} \sup_{t\in[0, T]}\Big(\mathcal{E}_N(f)(t)+\mathcal{E}_N^2(f)(t)\Big)
    \left\{\sqrt{T}\sup_{t\in[0, T]}\mathcal{E}_N(f)(t)+ \left(\int_{0}^{T}\left(\mathcal{D}_N^{mic}\right)^2(f)(t)dt\right)^{1/2}\right\}\\
    &\times\left(\int_{0}^{T}\left(\mathcal{D}_N^{mic}\right)^2(h(t))dt\right)^{1/2} +\|g_0\|_{H_x^NL_v^2} \|h\|_{L^\infty\left([0,T];\,H_x^NL_v^2\right)}.
    \end{aligned}
  \end{equation*}
  Finally \eqref{3eq3} and \eqref{3eq4} imply that
  \begin{equation*}
    \begin{aligned}
    |G(w)|&\leqslant C_{f,g_0}\|\mathcal{T}^*h\|_{L^2\left([0,T];\,H_x^NL_v^2\right)}\\
    &\leqslant C \|w\|_{L^2\left([0,T];\,H_x^NL_v^2\right)},
    \end{aligned}
  \end{equation*}
  where
  \begin{equation*}
    \begin{aligned}
    C_{f,g_0}=2\sqrt{T}\|g_0\|_{H_x^NL_v^2}+C\sqrt{T}\sup_{t\in[0, T]}\Big(\mathcal{E}_N(f)(t)+\mathcal{E}_N^2(f)(t)\Big) &\left\{\sqrt{T}\sup_{t\in[0, T]}\mathcal{E}_N(f)(t)+ \right.\\ &\left.\quad\left(\int_{0}^{T}\left(\mathcal{D}_N^{mic}\right)^2(f)(t)dt\right)^{1/2}\right\}.
    \end{aligned}
  \end{equation*}

  Thus, $G$ is a continuous linear functional on $\big(\mathbb{W};\,\|\cdot\|_{L^2\left([0,T];\,H_x^NL_v^2\right)}\big)$. So by the Hahn-Banach Theorem, $G$ can be extended from $\mathbb{W}$ to $L^2\left([0,T];\,H^N(\d x;\,L^2(\mu(1-\mu)\d v))\right)$. From the Riesz representation theorem, there exists
  \begin{equation*}
    g\in L^2\left([0,T];\,H^N(\d x;\,L^2(\mu(1-\mu)\d v))\right),
  \end{equation*}
  such that for any $w\in \mathbb{W}$,
  \begin{equation*}
    G(w)=(g,w)_{L^2\left([0,T];\,H_x^NL_v^2\right)}.
  \end{equation*}
  Thus fixing $f\in L^\infty\left([0,T];\,H^N(\d x;\,L^2(\mu(1-\mu)\d v))\right)$, by the definitions of the operator $\mathcal{T}^*$ and $G$, we have for any $h\in C^{\infty}\left([0,T];\,\mathcal{S}(\mathbb{R}_{x,v}^6)\right)$ with $h(T)=0$, there exists a unique $g\in L^2\left([0,T];\,H^N(\d x;\,L^2(\mu(1-\mu)\d v))\right)$ such that,
  \begin{equation}\label{3eq5}
    \begin{aligned}
    (\mathcal{T}g,\widehat{h})_{L^2\left([0,T];\,L_x^2L_v^2\right)} =&(\mathcal{T}g,h)_{L^2\left([0,T];\,H_x^NL_v^2\right)}\\
    =&(g,\mathcal{T}^*h)_{L^2\left([0,T];\,H_x^NL_v^2\right)} \\
    =&\big(\frac{1}{\epsilon}Q(f,f)+T(f,f,f),\widehat{h}\big)_{L^2\left([0,T];\,L_x^2L_v^2\right)} +\big(g_0,\widehat{h}(0)\big)_{L_x^2L_v^2}.
    \end{aligned}
  \end{equation}
  where
  \begin{equation*}
    \widehat{h}=\Lambda_x^{2N}h\in C^{\infty}\left([0,T];\,\mathcal{S}(\mathbb{R}_{x,v}^6)\right)\text{ with } \widehat{h}(T)=0,
  \end{equation*}
  and $\Lambda=(1-\Delta_x)^{1/2}$, $\Lambda_x^{2N}$ is used for changing the inner product from
  \begin{equation*}
    H^N(\d x;\,L^2(\mu(1-\mu)\d v)) \text{ to } L^2(\d x;L^2(\mu(1-\mu)\d v).
  \end{equation*}

  Since $\Lambda_x^{2N}$ is an isomorphism on
  \begin{equation*}
    h: h\in C^{\infty}\left([0,T];\,\mathcal{S}(\mathbb{R}_{x,v}^6)\right)
  \end{equation*}
  with $h(T)=0$,
  \begin{equation*}
    g\in L^2\left([0,T];\,H^N(\d x;\,L^2(\mu(1-\mu)\d v))\right)
  \end{equation*}
  is a solution of the Cauchy problem \eqref{3eq1}. Furthermore, using \eqref{3eq5}, we also have
  \begin{equation*}
    \sup_{0\leqslant t \leqslant T}\mathcal{E}_N^2(g)(t)+\frac{1}{\epsilon^2}\int_{0}^{T} \left(\mathcal{D}_N^{mic}\right)^2(g)(t)dt < \widetilde{C}_{f,g_0},
  \end{equation*}
  here
  \begin{equation*}
    \widetilde{C}_{f,g_0}=C\sup_{t\in[0, T]}\Big(\mathcal{E}_N^2(f)(t)+\mathcal{E}_N^4(f)(t)\Big) \left\{T\sup_{t\in[0, T]}\mathcal{E}_N^2(f)(t)+\int_{0}^{T}\left(\mathcal{D}_N^{mic}\right)^2(f)(t)dt\right\} +\mathcal{E}_N^2(g_0)
  \end{equation*}
\end{proof}

Now we concentrate on the local existence of solutions for fully nonlinear problem. Consider the following iterative scheme:
\begin{equation}\label{3eq6}
\left\{
  \begin{aligned}
  \partial_t g^{n+1} + \frac{1}{\epsilon}v\cdot \nabla_x g^{n+1} +\frac{1}{\epsilon^2}Lg^{n+1} &=\frac{1}{\epsilon}Q(g^{n},g^{n})+T(g^{n},g^{n},g^{n}),\\
  g^{n+1}(t,x,v)|_{t=0}&=g_0(x,v),
  \end{aligned}
\right.
\end{equation}
with $g^0=0$.

\begin{lemma}\label{lem6}
  There exist constants $0<\delta_0 \leqslant 1$ and $0<T\leqslant 1$ such that for any $0<\epsilon<1$, if $g_0\in H^N(\d x;\,L^2(\mu(1-\mu)\d v))$ with $\mathcal{E}_N(g_0)\leqslant \delta_0$, then the iteration problem \eqref{3eq6} admits a sequence of solutions $\{g^n\}_{n\geqslant 1}$ satisfying
  \begin{equation}\label{3eq7}
     \sup_{t \in [0,T]}\mathcal{E}_N^2(g^n)(t)+\frac{1}{\epsilon^2}\int_{0}^{T}\left(\mathcal{D}_N^{mic}\right)^2(g^n)(t) dt \leqslant 4\mathcal{E}_N^2(g_0)
   \end{equation}
\end{lemma}
\begin{proof}
  It is enough to prove \eqref{3eq7} by induction, since for the linear Cauchy problem \eqref{3eq6}, given $g^n$ satisfying \eqref{3eq7}, the existence of $g^{n+1}$ is assured by Lemma \ref{lem5}. Applying $\partial_x^\alpha$ to \eqref{3eq6}, then taking inner product with $\partial_x^\alpha g^{n+1}$ in $L^2(\mu(1-\mu)\d v \d x)$, using \eqref{2eq3}, \eqref{3eq19} and \eqref{2eq24}, we get
  \begin{equation*}
    \begin{aligned}
    &\frac{1}{2}\frac{\d}{\d t}\mathcal{E}_N^2(g^{n+1})+\frac{\lambda}{\epsilon^2} \left(\mathcal{D}_N^{mic}\right)^2(g^{n+1})\\
    \leqslant & \frac{C}{\epsilon}\left(\mathcal{E}_N(g^{n})+\mathcal{E}_N^2(g^{n})\right) \left\{\mathcal{D}_N^{mac}(g^{n})+\mathcal{D}_N^{mic}(g^{n})\right\} \mathcal{D}_N^{mic}(g^{n+1})\\
    \leqslant & \frac{C^2\eta}{2\epsilon^2}\left(\mathcal{D}_N^{mic}\right)^2(g^{n+1}) +\frac{1}{2\eta}\left(\mathcal{E}_N^2(g^{n})+\mathcal{E}_N^4(g^{n})\right) \left(\left(\mathcal{D}_N^{mac}\right)^2(g^{n})+\left(\mathcal{D}_N^{mic}\right)^2(g^{n})\right),
    \end{aligned}
  \end{equation*}
  in the last step we have used Cauchy-Schwarz inequality with $\eta>0$. Thus let $C^2\eta=\lambda$, we get
  \begin{equation*}
    \begin{aligned}
    &\frac{\d}{\d t}\mathcal{E}_N^2(g^{n+1})+\frac{\lambda}{\epsilon^2}\left(\mathcal{D}_N^{mic}\right)^2(g^{n+1})\\
    \leqslant &\frac{C^2}{\lambda}\left(\mathcal{E}_N^2(g^{n})+\mathcal{E}_N^4(g^{n})\right) \left(\left(\mathcal{D}_N^{mac}\right)^2(g^{n})+\left(\mathcal{D}_N^{mic}\right)^2(g^{n})\right).
    \end{aligned}
  \end{equation*}
  Note that we have \eqref{3eq17}, integrating on $[0, T]$ with $T \leqslant 1$,
  \begin{equation*}
    \begin{aligned}
    &\sup_{t \in [0,T]}\mathcal{E}_N^2(g^{n+1})(t) +\frac{1}{\epsilon^2}\int_{0}^{T}\left(\mathcal{D}_N^{mic}\right)^2(g^{n+1})(t)dt\\
    \leqslant &\mathcal{E}_N^2(g_0)
    +C^2\sup_{t \in [0,T]}\left(\mathcal{E}_N^2(g^{n})+\mathcal{E}_N^4(g^{n})\right) \cdot \left\{
    T\sup_{t \in [0,T]}\mathcal{E}_N^2(g^{n}(t))+\int_{0}^{T}\left(\mathcal{D}_N^{mic}\right)^2(g^{n}(t))dt\right\},
    \end{aligned}
  \end{equation*}
  we chose $\delta_0>0$ such that
  \begin{equation*}
    1+16C^2\delta_0^2\left(1+4\delta_0^2\right)\leqslant 4,
  \end{equation*}
  to complete the proof.
\end{proof}

Finally, we prove the convergence of $\{g^n\}$ using the uniform estimate \eqref{3eq7}.

\begin{theorem}[Local existence]\label{lex}
   There exist constants $0<\delta_0 \leqslant 1$ and $0<T\leqslant 1$ such that for any $0<\epsilon< 1$, if $g_{\epsilon,0}\in H^N(\d x;\,L^2(\mu(1-\mu)\d v))$ with $\mathcal{E}_N(g_{\epsilon,0})\leqslant \delta_0,$ then there is a unique solution
   \begin{equation*}
     g_{\epsilon}(t, x, v)\in L^{\infty}\left([0,T];\,H^N(\d x;\,L^2(\mu(1-\mu)\d v))\right)
   \end{equation*}
   to the Boltzmann-Fermi-Dirac equation \eqref{3eq8} such that
   \begin{equation}\label{3eq8}
     \sup_{t \in [0,T]}\mathcal{E}_N^2(g_{\epsilon})(t)+\frac{1}{\epsilon^2}\int_{0}^{T}
     \left(\mathcal{D}_N^{mic}\right)^2(g_{\epsilon})(t)dt \leqslant 4\mathcal{E}_N^2(g_{\epsilon,0}).
   \end{equation}
   Moreover, $\mathcal{E}_N(g_{\epsilon})(t)$ : $\left[0, T\right] \rightarrow \mathbb{R}$ is continuous. If $0\leqslant F(0, x, v)=\mu+\mu(1-\mu)g_0(x,v)\leqslant 1$, then
   \begin{equation*}
     0\leqslant F(t, x, v)=\mu+\mu(1-\mu)g(t,x,v) \leqslant 1.
   \end{equation*}
\end{theorem}
\begin{proof}
  We prove that $\{g^n\}$ is a Cauchy sequence in $L^{\infty}\left([0,T];\,L^2(\mu(1-\mu)\d v\d x)\right)$. Setting $w^n=g^{n+1}-g^{n}$, we deduce from \eqref{3eq6}
  \begin{equation*}
    \left\{
  \begin{aligned}
  &\partial_t w^n + \frac{1}{\epsilon}v\cdot \nabla_x w^n +\frac{1}{\epsilon^2}Lw^n =\left(\frac{1}{\epsilon}Q(g^n,g^n)+T(g^n,g^n,g^n)\right)- \left(\frac{1}{\epsilon}Q(g^{n-1},g^{n-1})+T(g^{n-1},g^{n-1},g^{n-1})\right),\\
  &w^n|_{t=0}=0.
  \end{aligned}
  \right.
  \end{equation*}
  Since for any $h\in L^2(\mu(1-\mu))$,
  \begin{align*}
    &\left\langle\left(\frac{1}{\epsilon}Q(g^n,g^n)+T(g^n,g^n,g^n)\right)- \left(\frac{1}{\epsilon}Q(g^{n-1},g^{n-1})+T(g^{n-1},g^{n-1},g^{n-1})\right), \mathbf{P}h\right\rangle=0  \\
    &\left(\frac{1}{\epsilon}Q(g^n,g^n)+T(g^n,g^n,g^n)\right)- \left(\frac{1}{\epsilon}Q(g^{n-1},g^{n-1})+T(g^{n-1},g^{n-1},g^{n-1})\right)\\
    =&\frac{1}{\epsilon^3}\left[Q(\epsilon w^{n-1},\epsilon g^n)+T(\epsilon w^{n-1},\epsilon g^n,\epsilon g^n)\right]\\ &+ \frac{1}{\epsilon^3}\left[Q(\epsilon g^{n-1},\epsilon w^{n-1})+T(\epsilon g^{n-1},\epsilon w^{n-1},\epsilon g^{n})\right]+T(g^{n-1},g^{n-1},w^{n-1}),
  \end{align*}
  then we have
  \begin{equation*}
    \begin{aligned}
    &\left(\left(\frac{1}{\epsilon}Q(g^n,g^n)+T(g^n,g^n,g^n)\right)- \left(\frac{1}{\epsilon}Q(g^{n-1},g^{n-1})+T(g^{n-1},g^{n-1},g^{n-1})\right), w^n\right)\\
    =&\frac{1}{\epsilon^3}\Big(Q(\epsilon w^{n-1},\epsilon g^n)+T(\epsilon w^{n-1},\epsilon g^n,\epsilon g^n), \{\,\mathbf{I}-\mathbf{P}\,\}w^n\Big)\\
    &+\frac{1}{\epsilon^3}\left(Q(\epsilon g^{n-1},\epsilon w^{n-1})+T(\epsilon g^{n-1},\epsilon w^{n-1},\epsilon g^{n}), \{\,\mathbf{I}-\mathbf{P}\,\}w^n\right)\\
    &+\Big(T(g^{n-1},g^{n-1},w^{n-1}), \{\,\mathbf{I}-\mathbf{P}\,\}w^n\Big).
    \end{aligned}
  \end{equation*}
  We need to estimate the three terms on the right hand side above. Note that $0<\epsilon<1$, by \eqref{2eq29}, \eqref{2eq30} and the Cauchy-Schwarz inequality with $\eta>0$, the first term is bounded by
  \begin{equation}\label{3eq9}
    \begin{aligned}
    &\frac{C}{\epsilon}\left\{\mathcal{E}_0(w^{n-1})\mathcal{D}_N^{mic}(g^n)+ \mathcal{E}_N(g^n)\left(\mathcal{E}_0(w^{n-1})+\mathcal{D}_0^{mic}(w^{n-1})\right) \right\}\left(1+\epsilon\mathcal{E}_N(g^n)\right)\mathcal{D}_0^{mic}(w^n)\\
    &\leqslant  C_{\eta}\left\{\mathcal{E}_0^2(w^{n-1})\left(\mathcal{D}_N^{mic}\right)^2(g^n)+ \mathcal{E}_N^2(g^n)\left(\mathcal{E}_0^2(w^{n-1})+\left(\mathcal{D}_0^{mic}\right)^2(w^{n-1})\right) \right\}\left(1+\mathcal{E}_N^2(g^n)\right)\\
    &+\frac{\eta}{\epsilon^2}\left(\mathcal{D}_0^{mic}\right)^2(w^n),
    \end{aligned}
  \end{equation}
  the second term is bounded by
  \begin{equation}\label{3eq10}
    \begin{aligned}
    &\frac{C}{\epsilon}\left\{\mathcal{E}_0(w^{n-1})\mathcal{D}_N^{mic}(g^{n-1})+ \mathcal{E}_N(g^{n-1})\left(\mathcal{E}_0(w^{n-1})+\mathcal{D}_0^{mic}(w^{n-1})\right) \right\}\left(1+\epsilon\mathcal{E}_N(g^n)\right)\mathcal{D}_0^{mic}(w^n)\\
    &\leqslant C_{\eta}\left\{\mathcal{E}_0^2(w^{n-1})\left(\mathcal{D}_N^{mic}\right)^2(g^{n-1})+ \mathcal{E}_N^2(g^{n-1})\left(\mathcal{E}_0^2(w^{n-1})+\left(\mathcal{D}_0^{mic}\right)^2(w^{n-1})\right)\right\} \left(1+\mathcal{E}_N^2(g^n)\right)\\
    &+\frac{\eta}{\epsilon^2}\left(\mathcal{D}_0^{mic}\right)^2(w^n),
    \end{aligned}
  \end{equation}
  and the third term of is bounded by
  \begin{equation}\label{3eq11}
    \begin{aligned}
    &C\left\{\mathcal{E}_N(g^{n-1})+\mathcal{D}_N^{mic}(g^{n-1})\right\} \mathcal{E}_N(g^{n-1})\mathcal{E}_0(w^{n-1})\mathcal{D}_0^{mic}(w^n)\\
    &\leqslant C_{\eta}\left\{\mathcal{E}_N^2(g^{n-1})+\left(\mathcal{D}_N^{mic}\right)^2(g^{n-1})\right\} \mathcal{E}_N^2(g^{n-1})\mathcal{E}_0^2(w^{n-1})+\eta\left(\mathcal{D}_0^{mic}\right)^2(w^n).
    \end{aligned}
  \end{equation}
  Thus we have the estimate
  \begin{equation}\label{3eq12}
    \begin{aligned}
    &\left|\Big(\left(\frac{1}{\epsilon}Q(g^n,g^n)+T(g^n,g^n,g^n)\right)- \left(\frac{1}{\epsilon}Q(g^{n-1},g^{n-1})+T(g^{n-1},g^{n-1},g^{n-1})\right), w^n\Big)_{L_{x,v}^2}\right|\\
    \leqslant &C_{\eta}\bigg\{\mathcal{E}_0^2(w^{n-1})\left(1+\mathcal{E}_N^2(g^n)\right) \left[\mathcal{E}_N^2(g^n)+\left(\mathcal{D}_N^{mic}\right)^2(g^n)\right]+\\
    &\qquad\mathcal{E}_0^2(w^{n-1})\left(1+\mathcal{E}_N^2(g^n)+\mathcal{E}_N^2(g^{n-1})\right) \left[\mathcal{E}_N^2(g^{n-1})+\left(\mathcal{D}_N^{mic}\right)^2(g^{n-1})\right]+\\
    &\qquad\left(\mathcal{D}_N^{mic}\right)^2(w^{n-1})\left(1+\mathcal{E}_N^2(g^n)\right) \left[\mathcal{E}_N^2(g^n)+\mathcal{E}_N^2(g^{n-1})\right]\bigg\}+\frac{3\eta}{\epsilon^2}\left(\mathcal{D}_0^{mic}\right)^2(w^n).
    \end{aligned}
  \end{equation}
  Let $3\eta=\frac{\lambda}{2}$, we get
  \begin{equation*}
    \begin{aligned}
    &\frac{1}{2}\frac{\d}{\d t}\|w^n\| +\frac{\lambda}{2\epsilon^2}\left(\mathcal{D}_N^{mic}\right)^2(w^n)\\
    \leqslant &C_{\eta}\bigg\{\mathcal{E}_0^2(w^{n-1})\left(1+\mathcal{E}_N^2(g^n)\right) \left[\mathcal{E}_N^2(g^n)+\left(\mathcal{D}_N^{mic}\right)^2(g^n)\right]+\\
    &\qquad\mathcal{E}_0^2(w^{n-1})\left(1+\mathcal{E}_N^2(g^n)+\mathcal{E}_N^2(g^{n-1})\right) \left[\mathcal{E}_N^2(g^{n-1})+\left(\mathcal{D}_N^{mic}\right)^2(g^{n-1})\right]+\\
    &\qquad\left(\mathcal{D}_N^{mic}\right)^2(w^{n-1})\left(1+\mathcal{E}_N^2(g^n)\right) \left[\mathcal{E}_N^2(g^n)+\mathcal{E}_N^2(g^{n-1})\right]\bigg\},
    \end{aligned}
  \end{equation*}
  it follows that
  \begin{equation*}
    \begin{aligned}
    &\|w^n\|_{L^\infty\left([0,T];\,{L_x^2L_v^2}\right)}^2+ \frac{1}{\epsilon^2}\int_{0}^{T}\left(\mathcal{D}_N^{mic}\right)^2(w^n)dt\\
    \leqslant & C \sup_{t\in [0,T]}\mathcal{E}_0^2(w^{n-1})\sup_{t\in [0,T]}\left(1+\mathcal{E}_N^2(g^n)+\mathcal{E}_N^2(g^{n-1})\right)\times\\
    &\left[T\sup_{t\in [0,T]}\mathcal{E}_N^2(g^n)+\int_{0}^{T}\left(\mathcal{D}_N^{mic}\right)^2(g^n)dt        +T\sup_{t\in [0,T]}\mathcal{E}_N^2(g^{n-1})+\int_{0}^{T}\left(\mathcal{D}_N^{mic}\right)^2(g^{n-1})dt\right]\\
    &+C\sup_{t\in [0,T]}\left(1+\mathcal{E}_N^2(g^n)\right)\sup_{t\in [0,T]} \left[\mathcal{E}_N^2(g^n)+\mathcal{E}_N^2(g^{n-1})\right]
    \int_{0}^{T}\left(\mathcal{D}_N^{mic}\right)^2(w^{n-1})dt.
    \end{aligned}
  \end{equation*}

  By using \eqref{3eq7} with $0<\delta_0<1$  small enough we get that for any $0<\epsilon <1$,
  \begin{equation*}
  \begin{aligned}
    &\|w^n\|_{L^\infty\left([0,T];\,{L_{x,v}^2}\right)}^2 +\frac{1}{\epsilon^2}\int_{0}^{T}\left(\mathcal{D}_N^{mic}\right)^2(w^n)dt \\
    \leqslant &\frac{1}{2}\left\{\|w^{n-1}\|_{L^\infty\left([0,T];\,{L_{x,v}^2}\right)}^2 +\frac{1}{\epsilon^2}\int_{0}^{T}\left(\mathcal{D}_N^{mic}\right)^2(w^{n-1})dt\right\}.
  \end{aligned}
  \end{equation*}
  Thus we have proved that $\{g^n\}$ is a Cauchy sequence in $L^\infty\left([0,T];\,L^2(\mu(1-\mu)\d v\d x)\right)$.

  The interpolation combining with the estimate \eqref{3eq7} implies that $\{g^n\}$ is a Cauchy sequence in $L^\infty\left([0,T];\,H^{N-\eta}(\d x;\,L^2(\mu(1-\mu)\d v))\right)$ for any $\eta> 0$ and the limit is in the functional space $L^\infty\left([0,T];\,H^N(\d x;\,L^2(\mu(1-\mu)\d v))\right)$. Finally the estimate \eqref{3eq8} follows from weak lower semicontinuity.

  Finally we prove the uniqueness of the local solutions. Let $g_1$ and $g_2$ be two local solutions of \eqref{3eq8}. Then
  \begin{equation*}
    g=g_1-g_2\in L^\infty\left([0,T];\,H^N(\d x;\,L^2(\mu(1-\mu)\d v))\right)
  \end{equation*}
  and it satisfies
  \begin{equation*}
    \left\{
    \begin{aligned}
        \partial_t g + \frac{1}{\epsilon}v\cdot \nabla g +\frac{1}{\epsilon^2}Lg &=\frac{1}{\epsilon}Q(g,g_1)+T(g,g_1,g)+\frac{1}{\epsilon}Q(g_2,g)+T(g_2,g,g)+T(g_2,g_2,g),\\
        g(t,x,v)|_{t=0}&=0,
    \end{aligned}
  \right.
  \end{equation*}
  a similar procedure proves that $g\equiv0$, which implies the uniqueness of the solutions.
\end{proof}

In the rest subsections, we devote ourselves to the global energy estimate. The global existence of Cauchy problem \eqref{3eq8} is obtained by a standard continuation argument of local solutions.

\subsection{Microscopic energy estimate}
Firstly, we study the estimate on the microscopic part $\{\mathbf{I}-\mathbf{P}\}g$ in the function space $H^{N}\left(\d x ; L^{2}\left(\mu(1-\mu)\d v\right)\right)$. For notational simplification, we drop the sub-index $\varepsilon$ of the local solution $g_\varepsilon$ to the Cauchy problem \eqref{3eq8}, and also drop $g$ in the notations $ \mathcal{E}_{N}, \mathcal{D}_N^{mac}, \mathcal{D}_N^{mic}$. Actually, we shall establish
\begin{lemma}\label{lem7}
  Let $g\in L^\infty\left([0,T];\,H^N(\d x;\,L^2(\mu(1-\mu)\d v))\right)$ be a solution of the equation \eqref{3eq8} constructed in Theorem \ref{lex}. Then, there exists a constant $C>0$ independent of $\epsilon$ such that the following estimate holds:
  \begin{equation}\label{3eq13}
    \frac{\d}{\d t}\mathcal{E}_N^2+\frac{1}{\epsilon^2}\left(\mathcal{D}_N^{mic}\right)^2 \leqslant C \left\{\frac{1}{\epsilon}\left(\mathcal{E}_N+\mathcal{E}_N^2\right)\left(\mathcal{D}_N^{mic}\right)^2 +\left(\mathcal{E}_N^2+\mathcal{E}_N^4\right)\left(\mathcal{D}_N^{mac}\right)^2\right\}.
  \end{equation}
\end{lemma}

\begin{proof}
  Applying $\partial_x^{\alpha}$ to \eqref{3eq8} and taking inner product with $\partial_x^{\alpha}g$ in $L^2(\d x; L^2(\mu(1-\mu)\d v))$ for $|\alpha|\leqslant N$, we get
  \begin{equation}\label{3eq14}
    \begin{aligned}
    \frac{1}{2}\frac{\d}{\d t}\sum_{|\alpha|\leqslant N}\left(\partial_x^\alpha g, \partial_x^\alpha g\right)+&\frac{1}{\epsilon^2}\sum_{|\alpha|\leqslant N}\Big(L\partial_x^{\alpha}g,\partial_x^{\alpha}g\Big) \\ &=\frac{1}{\epsilon^3}\sum_{|\alpha|\leqslant N}\Big(\partial_x^{\alpha}\left(Q(\epsilon g,\epsilon g)+T(\epsilon g,\epsilon g,\epsilon g)\right), \partial_x^{\alpha}g\Big).
    \end{aligned}
  \end{equation}
  where the inner product including $v\cdot\nabla g$ vanishes by integration by parts. By the coercivity of $L$ \eqref{2eq3} we have
  \begin{equation}\label{3eq15}
    \Big(L\partial_x^{\alpha}g,\partial_x^{\alpha}g\Big)\geqslant \lambda \|\{\mathbf{I}-\mathbf{P}\}\partial_x^{\alpha}g\|_{\nu}^2.
  \end{equation}
  Then we treat the right hand side of \eqref{3eq2}. Fixing $\alpha>0$, since
  \begin{equation*}
    \Big(\partial_x^{\alpha}\left(Q(\epsilon g,\epsilon g)+T(\epsilon g,\epsilon g,\epsilon g)\right), \partial_x^{\alpha}\mathbf{P}g\Big)=0,
  \end{equation*}
  we have
  \begin{equation}\label{3eq16}
    \begin{aligned}
    \Big(\partial_x^{\alpha}\left(Q(\epsilon g,\epsilon g)+T(\epsilon g,\epsilon g,\epsilon g)\right), \partial_x^{\alpha}g\Big)=&\sum_{\alpha_1+\alpha_2=\alpha}C_{\alpha}^{\alpha_1} \Big(Q(\epsilon\partial_x^{\alpha_1}g,\epsilon\partial_x^{\alpha_2}g), \partial_x^{\alpha}\{\mathbf{I}-\mathbf{P}\}g\Big)\\
    &+\sum_{\alpha_3+\alpha_4+\alpha_5=\alpha} C_{\alpha_3\alpha_4\alpha_5} \Big(T(\epsilon\partial_x^{\alpha_3}g,\epsilon\partial_x^{\alpha_4}g,\epsilon\partial_x^{\alpha_5}g), \partial_x^{\alpha}\{\mathbf{I}-\mathbf{P}\}g\Big).
    \end{aligned}
  \end{equation}
  Without loss of generality, let's assume $|\alpha_1|\leqslant N/2$, then $|\alpha_2|>0$. From \eqref{2eq29}, let $(\overline{f},\overline{g})$ be $(f,g)$, then replacing $(f,g,\widetilde{h})$ by $(\epsilon\partial_x^{\alpha_1}g,\epsilon\partial_x^{\alpha_2}g,\partial_x^{\alpha}\{\mathbf{I}-\mathbf{P}\}g)$ we deduce
  \begin{equation}\label{3eq17}
    \left|\Big(Q(\epsilon\partial_x^{\alpha_1}g,\epsilon\partial_x^{\alpha_2}g), \partial_x^{\alpha}\{\mathbf{I}-\mathbf{P}\}g\Big)\right|\leqslant C\epsilon^2 \mathcal{E}_N \left\{\mathcal{D}_N^{mac}+\mathcal{D}_N^{mic}\right\}\mathcal{D}_N^{mic}
  \end{equation}
  For the second term on the right hand side of \eqref{3eq16}, we establish the estimate
  \begin{equation}\label{3eq18}
    \left|\Big(T(\epsilon\partial_x^{\alpha_3}g,\epsilon\partial_x^{\alpha_4}g,\epsilon\partial_x^{\alpha_5}g), \partial_x^{\alpha}\{\mathbf{I}-\mathbf{P}\}g\Big)\right|
    \leqslant C\epsilon^3 \mathcal{E}_N^2 \left\{\mathcal{D}_N^{mac}+\mathcal{D}_N^{mic}\right\} \mathcal{D}_N^{mic}(g).
  \end{equation}
  for the following three cases:

\noindent(A) Two of $|\alpha_3|$, $|\alpha_4|$, $|\alpha_5|$ are equal to 0.

(a) $|\alpha_3|=0$ and $|\alpha_5|=0$, then $|\alpha_4|>0$. Using \eqref{2eq30} with $(\overline{f},\overline{g})=(f,g)$, we get \eqref{3eq18}.

(b) $|\alpha_4|=0$ and $|\alpha_5|=0$, then $|\alpha_3|>0$. Using \eqref{2eq30} with $(\overline{f},\overline{g})=(g,f)$, we get \eqref{3eq18}.

(c) $|\alpha_3|=0$ and $|\alpha_4|=0$, then $|\alpha_5|>0$. Using \eqref{2eq31} with $(\overline{f},\overline{g})=(f,g)$, we get \eqref{3eq18}.

\noindent(B) Only one of $|\alpha_3|$, $|\alpha_4|$, $|\alpha_5|$ is equal to 0.

(a) $|\alpha_3|=0$. Using \eqref{2eq31} with $(\overline{f},\overline{g})=(f,g)$ if $|\alpha_4|\leqslant |\alpha_5|$, or using \eqref{2eq30} with $(\overline{f},\overline{g})=(f,g)$ if $|\alpha_4|> |\alpha_5|$, to get \eqref{3eq18}.

(b) $|\alpha_4|=0$. Using \eqref{2eq31} with $(\overline{f},\overline{g})=(f,g)$ if $|\alpha_3|\leqslant |\alpha_5|$, or using \eqref{2eq30} with $(\overline{f},\overline{g})=(g,f)$ if $|\alpha_3|> |\alpha_5|$, to get \eqref{3eq18}.

(c) $|\alpha_5|=0$. Using \eqref{2eq30} with $(\overline{f},\overline{g})=(f,g)$ if $|\alpha_3|\leqslant |\alpha_4|$, or using \eqref{2eq30} with $(\overline{f},\overline{g})=(g,f)$ if $|\alpha_3|> |\alpha_4|$, to get \eqref{3eq18}.

\noindent(C) All of $|\alpha_3|$, $|\alpha_4|$, $|\alpha_5|$ are greater than 0. We claim that at least two of the $|\alpha_3|$, $|\alpha_4|$, $|\alpha_5|$ are less than or equal to $|\alpha|/2$. This case is similar to case (A).

  Thus, plugging \eqref{3eq17} and \eqref{3eq18} into \eqref{3eq16} yields for $0<|\alpha|\leqslant N$
  \begin{equation}\label{3eq19}
    \left|\Big(\partial_x^{\alpha}\left(Q(\epsilon g,\epsilon g)+T(\epsilon g,\epsilon g,\epsilon g)\right), \partial_x^{\alpha}g\Big)\right|\leqslant C\epsilon^2 \Big(\mathcal{E}_N +\epsilon\mathcal{E}_N^2\Big) \left\{\mathcal{D}_N^{mac}+\mathcal{D}_N^{mic}\right\}\mathcal{D}_N^{mic}(g).
  \end{equation}

  Next, for $|\alpha|=0$, similar to \eqref{3eq16}, we have
  \begin{equation}\label{3eq20}
    \Big(Q(\epsilon g,\epsilon g)+T(\epsilon g,\epsilon g,\epsilon g), g\Big)=\Big(Q(\epsilon g,\epsilon g)+T(\epsilon g,\epsilon g,\epsilon g), \{\mathbf{I}-\mathbf{P}\}g\Big).
  \end{equation}
  Splitting $g=\mathbf{P}g+\{\mathbf{I}-\mathbf{P}\}g$, to further decompose $\Big(Q(\epsilon g,\epsilon g),\{\mathbf{I}-\mathbf{P}\}g\Big)$ into
  \begin{equation}\label{3eq21}
    \Big(Q(\epsilon\mathbf{P}g,\epsilon\mathbf{P}g), \{\mathbf{I}-\mathbf{P}\}g\Big)+\Big(Q(\epsilon\{\mathbf{I}-\mathbf{P}\}g,\epsilon\mathbf{P}g), \{\mathbf{I}-\mathbf{P}\}g\Big)+\Big(Q(\epsilon g,\epsilon\{\mathbf{I}-\mathbf{P}\}g),\{\mathbf{I}-\mathbf{P}\}g\Big).
  \end{equation}
  Recall the estimate \eqref{2eq29}, replacing $f$ by $\epsilon\{\mathbf{I}-\mathbf{P}\}g$ and $g$ by $\epsilon g$ there we get
  \begin{equation}\label{3eq22}
  \begin{aligned}
    &\left|\Big(Q(\epsilon\{\mathbf{I}-\mathbf{P}\}g,\epsilon\mathbf{P}g), \{\mathbf{I}-\mathbf{P}\}g\Big)+\Big(Q(\epsilon g,\epsilon\{\mathbf{I}-\mathbf{P}\}g),\{\mathbf{I}-\mathbf{P}\}g\Big)\right|\\
    \leqslant &C\epsilon^2\mathcal{E}_N \left\{\mathcal{D}_N^{mac}+\mathcal{D}_N^{mic}\right\}\mathcal{D}_N^{mic}.
  \end{aligned}
  \end{equation}
  On the other hand, putting
  \begin{equation*}
    \mathbf{P} g=a(t,x)+b(t,x)\cdot v+ c(t,x)|v|^2,
  \end{equation*}
  into the first term in \eqref{3eq21} and using the H\"{o}lder's inequality to get
  \begin{equation}\label{3eq23}
    \begin{aligned}
    \left|\Big(Q(\epsilon\mathbf{P}g,\epsilon\mathbf{P}g),\{\mathbf{I}-\mathbf{P}\}g\Big) \right|
    \leqslant& C \epsilon^2\int_{\mathbb{R}^3}|\mathbf{P}g|_{L_v^2}|\nu^{1/2}\mathbf{P}g|_{L_v^2}|\nu^{1/2} \{\mathbf{I}-\mathbf{P}\}g|_{L_v^2}dx\\
    \leqslant &C \epsilon^2\left(\int_{\mathbb{R}^3}|\mathbf{P}g|_{L_v^2}^4dx\right)^{1/4} \left(\int_{\mathbb{R}^3}|\nu^{1/2}\mathbf{P}g|_{L_v^2}^4dx\right)^{1/4}\|\nu^{1/2}\{\mathbf{I}-\mathbf{P}\}g\|_{L_x^2L_v^2} \\
    \leqslant &C\epsilon^2\|\mathbf{P}g\|_{H_x^1L_v^2}\|\nabla_x \mathbf{P}g\|_{L_x^2L_v^2} \|\nu^{1/2}\{\mathbf{I}-\mathbf{P}\}g\|_{L_x^2L_v^2}\\
    \leqslant &C\epsilon^2\mathcal{E}_N \mathcal{D}_N^{mac}\mathcal{D}_N^{mic},
    \end{aligned}
  \end{equation}
  where the third step is due to \eqref{3eq16}, in detail,
  \begin{equation*}
  \begin{aligned}
    \left(\int_{\mathbb{R}^3}|\nu^{1/2}\mathbf{P}g|_{L_v^2}^4dx\right)^{1/2}\leqslant
    & C \left\|a^2+|b|^2+c^2\right\|\\
    \leqslant &C \left(\|a(t,\cdot)\|_{L^6\left(\d x\right)}+ \|b(t,\cdot)\|_{L^6\left(\d x\right)}+\|c(t,\cdot)\|_{L^6\left(\d x\right)}\right)\times\\
    &\left(\|a(t,\cdot)\|_{L^3\left(\d x\right)}+ \|b(t,\cdot)\|_{L^3\left(\d x\right)}+\|c(t,\cdot)\|_{L^3\left(\d x\right)}\right)\\
    \leqslant & C \left(\|\nabla_x a(t,\cdot)\|_{L^2\left(\d x\right)}+ \|\nabla_x b(t,\cdot)\|_{L^2\left(\d x\right)}+\|\nabla_x c(t,\cdot)\|_{L^2\left(\d x\right)}\right)\times\\
    &\left(\|a(t,\cdot)\|_{H^1\left(\d x\right)}+ \|b(t,\cdot)\|_{H^1\left(\d x\right)}+\|c(t,\cdot)\|_{H^1\left(\d x\right)}\right)\\
    \leqslant & C \|\nabla_x \mathbf{P}g\|_{L_{x,v}^2}\|\mathbf{P}g\|_{H_x^1\left(L_v^2\right)}.
  \end{aligned}
  \end{equation*}
  Plugging \eqref{3eq22} and \eqref{3eq23} into \eqref{3eq21} we get
  \begin{equation}\label{3eq24}
    \Big(Q(\epsilon g,\epsilon g),\{\mathbf{I}-\mathbf{P}\}g\Big)\leqslant C \epsilon^2\mathcal{E}_N \left\{\mathcal{D}_N^{mac}+\mathcal{D}_N^{mic}\right\}\mathcal{D}_N^{mic}.
  \end{equation}
  Repeating the process of \eqref{3eq21}-\eqref{3eq24} to get
  \begin{equation}\label{3eq25}
    \Big(T(\epsilon g,\epsilon g,\epsilon g),\{\mathbf{I}-\mathbf{P}\}g\Big)\leqslant C \epsilon^3\mathcal{E}_N^2 \left\{\mathcal{D}_N^{mac}+\mathcal{D}_N^{mic}\right\}\mathcal{D}_N^{mic}.
  \end{equation}
  Hence combining \eqref{3eq20} with \eqref{3eq24} and \eqref{3eq25} we obtain
  \begin{equation}\label{3eq26}
    \left|\Big(Q(\epsilon g,\epsilon g)+T(\epsilon g,\epsilon g,\epsilon g), g\Big)\right|\leqslant C \epsilon^2\left(\mathcal{E}_N+\epsilon\mathcal{E}_N^2\right) \left\{\mathcal{D}_N^{mac}+\mathcal{D}_N^{mic}\right\}\mathcal{D}_N^{mic}.
  \end{equation}
  We plug \eqref{3eq15}, \eqref{3eq19} and \eqref{3eq26} into \eqref{3eq2} and use Cauchy-Schwarz inequality with $\eta >0$ to get
  \begin{equation}\label{3eq27}
    \begin{aligned}
    \frac{1}{2}\frac{\d}{\d t}\mathcal{E}_N^2+\frac{\lambda}{\epsilon^2}\left(\mathcal{D}_N^{mic}\right)^2
    \leqslant& \frac{C}{\epsilon}\left(\mathcal{E}_N+\epsilon\mathcal{E}_N^2\right) \left\{\mathcal{D}_N^{mac}+\mathcal{D}_N^{mic}\right\} \mathcal{D}_N^{mic}\\
    \leqslant & \frac{C}{\epsilon}\left(\mathcal{E}_N+\mathcal{E}_N^2\right)\left(\mathcal{D}_N^{mic}\right)^2+ \frac{\eta}{\epsilon^2}\left(\mathcal{D}_N^{mic}\right)^2+ \frac{C}{\eta}\left(\mathcal{D}_N^{mac}\right)^2\left(\mathcal{E}_N^2+\mathcal{E}_N^4\right).
    \end{aligned}
  \end{equation}
  By taking $\eta=\frac{\lambda}{2}$ we complete the proof.
\end{proof}

\subsection{Macroscopic energy estimate}
Now we study the macroscopic part $\mathbf{P}g$, where $g$ is a solution of the equation \eqref{3eq8}.
The key idea is to show the macroscopic part of the solution, $a$, $b$, $c$, are bounded by the microscopic part $\{\mathbf{I}-\mathbf{P}\}g$. Therefore, we plug the decomposition $g=\,\mathbf{P} g\,+\{\mathbf{I}-\mathbf{P}\}g$ into the Boltzmann-Fermi-Dirac equation \eqref{3eq8} to get the time evolution of the macroscopic part $\mathbf{P}g$:
\begin{equation}\label{3eq28}
  \left\{\partial_t +\frac{1}{\epsilon}v\cdot\nabla_x\right\} \left(\mathbf{P}g\right)=-\left\{\partial_t +\frac{1}{\epsilon}v\cdot\nabla_x+\frac{1}{\epsilon^2}L\right\}\{\,\mathbf{I}-\mathbf{P}\,\}g +\frac{1}{\epsilon}Q(g,g)+T(g,g,g).
\end{equation}

In order to study the time evolution of $\mathbf{P}g$ precisely, we write
\begin{equation}\label{3eq29}
  \mathbf{P} g=a(t,x)+b(t,x)\cdot v+ c(t,x)|v|^2,
\end{equation}
and set the thirteen moments
\begin{equation}\label{3eq30}
  \big\{e_k(v)\big\}_{k=1}^{13}=\left\{\,1,\,v_i,\,v_iv_j,\,v_i|v|^2\,\right\},\;i,j=1,2,3.
\end{equation}
Then for fixed $t,x$, put the expansion \eqref{3eq29} into \eqref{3eq28} to get the following macroscopic equations:
\begin{align}
    \partial_{t} a &=-\partial_{t} r^{(0)}+\frac{1}{\epsilon}m^{(0)}+\frac{1}{\epsilon^2}\ell^{(0)}+h^{(0)} \label{3eq31}\\
    \partial_{t} b_{i}+\frac{1}{\epsilon}\partial_{i} a &=-\partial_{t} r_{i}^{(1)}+\frac{1}{\epsilon}m_{i}^{(1)}+\frac{1}{\epsilon^2}\ell_{i}^{(1)}+h_{i}^{(1)}\label{3eq32} \\
    \partial_{t} c+\frac{1}{\epsilon}\partial_{i} b_{i} &=-\partial_{t} r_{i}^{(2)}+\frac{1}{\epsilon}m_{i}^{(2)}+\frac{1}{\epsilon^2}\ell_{i}^{(2)}+h_{i}^{(2)} \label{3eq33}\\
    \frac{1}{\epsilon}\partial_{i} b_{j}+\frac{1}{\epsilon}\partial_{j} b_{i} &=-\partial_{t} r_{i j}^{(2)}+\frac{1}{\epsilon}m_{ij}^{(2)}+\frac{1}{\epsilon^2}\ell_{ij}^{(2)}+h_{i j}^{(2)}, \quad i \neq j \label{3eq34}\\
    \frac{1}{\epsilon}\partial_{i} c &=-\partial_{t} r_{i}^{(3)}+\frac{1}{\epsilon}m_{i}^{(3)}+\frac{1}{\epsilon^2}\ell_{i}^{(3)}+h_{i}^{(3)},\label{3eq35}
\end{align}
by collecting the coefficients on both sides with respect to the moments in \eqref{3eq30}, where $r^{(0)}, r_{i}^{(1)},\cdots, h_{i j}^{(2)}, h_{i}^{(3)}$ are coefficients of $\{\,\mathbf{I}-\mathbf{P}\,\}g$, $-v\cdot\nabla_x\{\,\mathbf{I}-\mathbf{P}\,\}g$, $-L\{\,\mathbf{I}-\mathbf{P}\,\}g$ and $\frac{1}{\epsilon}Q(g,g)+T(g,g,g)$ with respect to the moments in \eqref{3eq30} respectively.

On the other hand, we introduce some constants:
\begin{equation}\label{3eq36}
  \begin{array}{ll}
    \displaystyle p_0^0=\int_{\mathbb{R}^3}\mu\d v,&\displaystyle\qquad p_2^0 =\int_{\mathbb{R}^3}|v|^2\mu\d v,\\
    \displaystyle p_0^1=\int_{\mathbb{R}^3}\mu(1-\mu)\d v, &\displaystyle\qquad p_2^1 =\int_{\mathbb{R}^3}|v|^2\mu(1-\mu)\d v,\\
    \displaystyle p_4^1=\int_{\mathbb{R}^3}|v|^4\mu(1-\mu)\d v. &
  \end{array}
\end{equation}
By taking inner product with $1,\,v,\,|v|^2$ in $ L^{2}\left(\mu(1-\mu)\d v\right)$ respectively, equation \eqref{3eq28} yields the local conservation laws:
\begin{equation}\label{3eq37}
\left\{
  \begin{aligned}
  &\partial_ta-\frac{p_2^1}{p_0^1p_4^1-(p_2^1)^2}\frac{1}{\epsilon}\nabla\cdot\left\langle|v|^{2} v , \{\mathbf{I}-\mathbf{P}\}g\right\rangle=0,\\
  &\partial_t b_i+\frac{1}{\epsilon}\partial_i\left(a+\frac{p_4^1}{p_2^1}c\right)+\frac{3}{p_2^1}\frac{1}{\epsilon}\nabla\cdot\left\langle vv_{i} , \{\mathbf{I}-\mathbf{P}\}g\right\rangle=0,\\
  &\partial_tc+\frac{1}{3\epsilon}\nabla\cdot b+\frac{p_0^1}{p_0^1p_4^1-(p_2^1)^2}\frac{1}{\epsilon}\nabla\cdot\left\langle|v|^{2} v , \{\mathbf{I}-\mathbf{P}\}g\right\rangle=0.
  \end{aligned}
\right.
\end{equation}

Next, we claim that the coefficients of the linear terms $\{\,\mathbf{I}-\mathbf{P}\,\}g$, $-v\cdot\nabla_x\{\,\mathbf{I}-\mathbf{P}\,\}g$, $-L\{\,\mathbf{I}-\mathbf{P}\,\}g$ and the nonlinear term $\frac{1}{\epsilon}Q(g,g)+T(g,g,g)$ involved in the right-hand side of macroscopic system \eqref{3eq31}-\eqref{3eq35} can be bounded by the microscopic dissipation rate.
\begin{lemma}\label{lem8}
    It holds that
    \begin{equation}\label{3eq38}
      \sum\limits_{|\alpha| \leqslant N-1} \sum\limits_{i j}\left\|\partial_{x}^{\alpha}\left[r^{(0)}, r_{i}^{(1)}, r_{i}^{(2)}, r_{i j}^{(2)}, r_{i}^{(3)}\right]\right\| \leqslant C \sum\limits_{|\alpha| \leqslant N-1}\left\|\partial_{x}^{\alpha} \{\mathbf{I}-\mathbf{P}\}g\right\|,
    \end{equation}
    \begin{equation}\label{3eq39}
      \sum\limits_{|\alpha| \leqslant N-1} \sum\limits_{i j}\left\|\partial_{x}^{\alpha}\left[m^{(0)}, m_{i}^{(1)}, m_{i}^{(2)}, m_{ij}^{(2)}, m_{i}^{(3)}\right]\right\| \leqslant C \sum\limits_{|\alpha| \leqslant N}\left\|\partial_{x}^{\alpha} \{\mathbf{I}-\mathbf{P}\}g\right\|,
    \end{equation}
    \begin{equation}\label{3eq40}
      \sum\limits_{|\alpha| \leqslant N-1} \sum\limits_{i j}\left\|\partial_{x}^{\alpha}\left[\ell^{(0)}, \ell_{i}^{(1)}, \ell_{i}^{(2)}, \ell_{ij}^{(2)}, \ell_{i}^{(3)}\right]\right\| \leqslant C \sum\limits_{|\alpha| \leqslant N-1}\left\|\partial_{x}^{\alpha} \{\mathbf{I}-\mathbf{P}\}g\right\|,
    \end{equation}
    and
    \begin{equation}\label{3eq41}
      \sum\limits_{|\alpha| \leqslant N} \sum\limits_{i j}\left\|\partial_{x}^{\alpha}\left[h^{(0)}, h_{i}^{(1)}, h_{i}^{(2)}, h_{i j}^{(2)}, h_{i}^{(3)}\right]\right\| \leqslant \frac{C}{\epsilon}\left(\mathcal{E}_N+\mathcal{E}_N^2\right) \left(\mathcal{D}_N^{mac}+\mathcal{D}_N^{mic}\right).
    \end{equation}
\end{lemma}
\begin{proof}
Let $\{\widetilde{e_i}(v)\}_{i=1}^{13}$ be its corresponding orthonormal basis such that for some constants $\xi^{ij}$
  \begin{equation*}
    \widetilde{e_i}(v)=\sum_{j=1}^{13}\xi^{ij}e_i(v).
  \end{equation*}

  For the linear term, the coefficients $m^{(0)}, m_{i}^{(1)}, m_{i}^{(2)}, m_{ij}^{(2)}, m_{i}^{(3)}$ of the projection of $-v\cdot\nabla_x\{\mathbf{I}-\mathbf{P}\}g$ take the form
  \begin{equation*}
    -\sum_{i,n=1}^{13}\xi^{ij}\xi^{in}\int_{\mathbb{R}^3}\{v\cdot\nabla_x\} \{\mathbf{I}-\mathbf{P}\}g\cdot e_n(v)\mu(1-\mu)\d v.
  \end{equation*}
  Same is true after we take $\partial_x^{\alpha}$. Let $|\alpha| \leqslant N-1$, \eqref{3eq39} holds for that
  \begin{equation*}
    \begin{aligned}
    &\left\|\int_{\mathbb{R}^3}v\cdot\nabla_x\{\mathbf{I}-\mathbf{P}\}\partial_x^{\alpha}g\cdot e_n(v)\mu(1-\mu)\d v\right\|^2\\
    \leqslant & \int_{\mathbb{R}^3}|e_n(v)|^2|v|^2 \mu(1-\mu)\d v\iint\limits_{\mathbb{R}^3\times \mathbb{R}^3}|\{\mathbf{I}-\mathbf{P}\}\nabla\partial_x^{\alpha}g|^2\mu(1-\mu)\d v\d x\\
    \leqslant & C \sum\limits_{|\alpha| \leqslant N}\left\|\partial_{x}^{\alpha} \{\mathbf{I}-\mathbf{P}\}g\right\|^{2},
    \end{aligned}
  \end{equation*}
  Note that $\nu(v)\leqslant C(1+|v|)$ and $K$ is bounded from $L^2(\mu(1-\mu)\d v)$ to itself, then the same process as above gives the estimates \eqref{3eq38} and \eqref{3eq40} for the coefficients of the rest linear terms.

  For the coefficients $h^{(0)}$, $h_{i}^{(1)}$, $h_{i}^{(2)}$, $h_{i j}^{(2)}$, $h_{i}^{(3)}$, by \eqref{2eq6} and \eqref{2eq8}, we have
  \begin{equation*}
    \left\|\left\langle\partial_x^\alpha\left[\frac{1}{\epsilon}Q(g,g)+T(g,g,g)\right], e_n\right\rangle\right\| \leqslant \frac{C}{\epsilon}\left(1+\epsilon\mathcal{E}_N\right)\mathcal{E}_N \left(\mathcal{D}_N^{mac}+\mathcal{D}_N^{mic}\right),
  \end{equation*}
  thus \eqref{3eq41} holds.
\end{proof}

Similarly, those terms containing the microscopic part $\{\mathbf{I}-\mathbf{P}\}g$ in the conservation laws \eqref{3eq37} can be also bounded by the microscopic dissipation rate.
\begin{lemma}\label{lem9}
  It holds that
  \begin{equation}\label{3eq42}
    \sum\limits_{|\alpha| \leqslant N-1}\left\|\partial_{x}^{\alpha} \nabla_{x} \cdot\left[\left\langle|v|^{2} v, \{\mathbf{I}-\mathbf{P}\}g\right\rangle,\left\langle v \otimes v, \{\mathbf{I}-\mathbf{P}\}g\right\rangle\right]\right\|^{2}\leqslant C \sum\limits_{0<|\alpha|\leqslant N}\left\|\partial_{x}^{\alpha} \{\mathbf{I}-\mathbf{P}\}g\right\|^{2}.
  \end{equation}
\end{lemma}

\vskip\baselineskip

Before we establish the crucial macroscopic energy estimate, we shall introduce the notation:
\begin{equation}\label{3eq43}
    \begin{aligned}
    G_{\alpha,i}^1(t)&=\Big\langle \partial_i\partial_x^\alpha a, \partial_x^\alpha b_i \Big\rangle, \quad G_{\alpha,i}^2(t)=\left\langle \partial_x^\alpha r_i^{(1)}, \partial_i\partial_x^\alpha a \right\rangle,\quad G_{\alpha,i}^3(t)=\left\langle \partial_x^\alpha r_i^{(3)}, \partial_i\partial_x^\alpha c \right\rangle,\\
    G_{\alpha,i}^4(t)&=-\sum_{j\neq i}\left\langle \partial_x^\alpha r_j^{(2)}, \partial_i\partial_x^\alpha b_i \right\rangle+\sum_{j\neq i}\left\langle \partial_x^\alpha r_{ji}^{(2)}, \partial_j\partial_x^\alpha b_i \right\rangle+2\left\langle \partial_x^\alpha r_i^{(2)}, \partial_i\partial_x^\alpha b_i \right\rangle,\\
    \end{aligned}
\end{equation}
and
\begin{equation}\label{3eq44}
  G_\alpha(t)=\sum_{i=1}^3\left[G_{\alpha,i}^1(t)+G_{\alpha,i}^2(t)+G_{\alpha,i}^3(t)+G_{\alpha,i}^4(t)\right].
\end{equation}

\begin{lemma}\label{lem10}
  Let $g$ be a solution of the scaled Boltzmann-Fermi-Dirac equation \eqref{3eq32}. Then there exists a positive constant $\widetilde{C}$ independent of $\epsilon$ such that the following estimate holds:
  \begin{equation}\label{3eq45}
  \begin{aligned}
    &\epsilon\frac{\d}{\d t}\sum_{|\alpha|\leqslant N-1}G_\alpha(t)
    +\sum_{|\alpha|\leqslant N-1}\left\|\nabla_x \partial_x^{\alpha}(a,b,c)\right\|^2\\
    \leqslant &\widetilde{C}\left\{\frac{1}{\epsilon^2}\left(\mathcal{D}_N^{mic}\right)^2+\left(\mathcal{E}_N^2+\mathcal{E}_N^4\right) \left(\left(\mathcal{D}_N^{mac}\right)^2+\left(\mathcal{D}_N^{mic}\right)^2\right)\right\}
  \end{aligned}
  \end{equation}
  for any $t\in [0,T]$.
\end{lemma}

\begin{proof}
  We fix a constant $\eta \in (0,1)$ to be determined later and let $|\alpha|\leqslant N-1$. Note that from \eqref{3eq31}-\eqref{3eq35}, we deduce $b=(b_1,b_2,b_3)$ satisfies the following elliptic-type equation:
  \begin{equation}\label{3eq46}
  \begin{aligned}
        -\Delta_{x} b_{j}-\partial_{j} \partial_{j} b_{j}=& \sum_{i \neq j} \partial_{j}\left[-\epsilon\partial_{t} r_{i}^{(2)}+m_{i}^{(2)}+\frac{1}{\epsilon}\ell_{i}^{(2)}+\epsilon h_{i}^{(2)}\right]\\
        &-\sum_{i \neq j} \partial_{i}\left[-\epsilon\partial_{t} r_{i j}^{(2)} +m_{ij}^{(2)}+\frac{1}{\epsilon}\ell_{ij}^{(2)}+\epsilon h_{i j}^{(2)}\right]-2 \partial_{j}\left[-\epsilon\partial_{t} r_{j}^{(2)}+m_{j}^{(2)}+\frac{1}{\epsilon}\ell_{j}^{(2)}+\epsilon h_{j}^{(2)}\right].
    \end{aligned}
  \end{equation}
  Thus we start with estimates on $b$.

  {\it Estimates on $b$.}
  Applying $\partial_{x}^{\alpha}$ to \eqref{3eq46}, then taking inner product with $\partial_{x}^{\alpha}b_j$ in $L^2(\d x)$,
  \begin{equation}\label{3eq47}
    \begin{aligned}
    &\left\|\nabla_{x} \partial_{x}^{\alpha} b_{j}\right\|^{2}+\left\|\partial_{j} \partial_{x}^{\alpha} b_{j}\right\|^{2}=\epsilon R_{r,b}-\sum_{i \neq j}\left(\partial_{x}^{\alpha}\left[\frac{1}{\epsilon}\ell_{i}^{(2)}+m_{i}^{(2)}+ \epsilon h_{i}^{(2)}\right], \partial_{j} \partial_{x}^{\alpha} b_{j}\right)\\
    +&\sum_{i \neq j} \left(\partial_{x}^{\alpha}\left[\frac{1}{\epsilon}\ell_{ij}^{(2)}+m_{ij}^{(2)}+ \epsilon h_{i j}^{(2)}\right], \partial_{i} \partial_{x}^{\alpha} b_{j}\right)+ 2\left(\partial_{x}^{\alpha}\left[\frac{1}{\epsilon}\ell_{j}^{(2)}+m_{j}^{(2)}+\epsilon h_{j}^{(2)}\right], \partial_{j} \partial_{x}^{\alpha} b_{j}\right).
    \end{aligned}
    \end{equation}
  Integrating by parts
  \begin{equation*}
    \begin{aligned}
    &\epsilon R_{r,b}=-\epsilon\left(\partial_{t}\left[\sum_{i \neq j} \partial_{j} \partial_{x}^{\alpha} r_{i}^{(2)}-\sum_{i \neq j} \partial_{i} \partial_{x}^{\alpha} r_{i j}^{(2)}-2 \partial_{j} \partial_{x}^{\alpha} r_{j}^{(2)}\right], \partial_{x}^{\alpha} b_{j}\right)\\
    &=-\epsilon\frac{d}{d t}\left(\sum_{i \neq j} \partial_{j} \partial_{x}^{\alpha} r_{i}^{(2)}-\sum_{i \neq j} \partial_{i} \partial_{x}^{\alpha} r_{i j}^{(2)}-2 \partial_{j} \partial_{x}^{\alpha} r_{j}^{(2)}, \partial_{x}^{\alpha} b_{j}\right)\\
    &+\epsilon\left(\sum_{i \neq j} \partial_{j} \partial_{x}^{\alpha} r_{i}^{(2)}-\sum_{i \neq j} \partial_{i} \partial_{x}^{\alpha} r_{i j}^{(2)}-2 \partial_{j} \partial_{x}^{\alpha} r_{j}^{(2)}, \partial_{x}^{\alpha} \partial_{t} b_{j}\right).
    \end{aligned}
  \end{equation*}

  Note the first term on the right hand side above is just $-\epsilon\frac{\d}{\d t}G_{\alpha,j}^4(t)$, and from the conservation law \eqref{3eq37}, the second term is bounded by
\begin{equation*}
    \begin{aligned}
    &\eta\epsilon^2\left\|\partial_{x}^{\alpha} \partial_{t} b_{j}\right\|^{2}+\frac{1}{4 \eta}\left\|\sum_{i \neq j} \partial_{j} \partial_{x}^{\alpha} r_{i}^{(2)}-\sum_{i \neq j} \partial_{i} \partial_{x}^{\alpha} r_{i j}^{(2)}-2 \partial_{j} \partial_{x}^{\alpha} r_{j}^{(2)}\right\|^{2} \\
    &\leqslant  \eta\left\|\partial_{x}^{\alpha} \partial_{j}(a+\frac{p_4^1}{p_2^1}c)+\partial_{x}^{\alpha} \nabla_{x} \cdot\left( vv_{i} , \{\mathbf{I}-\mathbf{P}\}g\right)\right\|^{2}+\frac{C}{\eta}\left[\sum_{i \neq j}\left\|\partial_{j} \partial_{x}^{\alpha} r_{i}^{(2)}\right\|^{2}+\sum_{i \neq j}\left\|\partial_{i} \partial_{x}^{\alpha} r_{i j}^{(2)}\right\|^{2}+\left\|\partial_{j} \partial_{x}^{\alpha}r_{j}^{(2)}\right\|^{2}\right]\\
    &\leqslant  C \eta\left[\left\|\partial_{j} \partial_{x}^{\alpha}(a, c)\right\|^{2}+\left\| \partial_{x}^{\alpha} \nabla_{x} \cdot\left( vv_{i} , \{\mathbf{I}-\mathbf{P}\}g\right)\right\|^{2}\right]+\frac{C}{\eta} \sum_{i j}\left\|\nabla_{x} \partial_{x}^{\alpha}\left[r_{i}^{(2)}, r_{i j}^{(2)}, r_{j}^{(2)}\right]\right\|^{2}.
    \end{aligned}
\end{equation*}
The sum of the rest terms on the right hand side of \eqref{3eq47} is bounded by
\begin{equation*}
    \begin{aligned}
        &\eta\left\|\nabla_{x} \partial_{x}^{\alpha} b_{j}\right\|^{2}+\frac{C}{\eta}\left\{\sum_{i \neq j} \left\|\partial_{x}^{\alpha}\left[\frac{1}{\epsilon}\ell_{i}^{(2)}+m_{i}^{(2)}+\epsilon h_{i}^{(2)}\right]\right\|^{2}+\right.\\
        &\left.\hspace{3cm}\sum_{i \neq j} \left\|\partial_{x}^{\alpha}\left[\frac{1}{\epsilon}\ell_{ij}^{(2)}+m_{ij}^{(2)}+ \epsilon h_{i j}^{(2)}\right]\right\|^{2} +\left\|\partial_{x}^{\alpha}\left[\frac{1}{\epsilon}\ell_{ij}^{(2)}+m_{ij}^{(2)}+ \epsilon h_{i j}^{(2)} \right]\right\|^{2}\right\} \\
        &\leqslant \eta\left\|\nabla_{x} \partial_{x}^{\alpha} b_{j} \right\|^{2}+\frac{C}{\eta} \sum_{i j} \left\|\partial_{x}^{\alpha}\left[\frac{1}{\epsilon}\ell_{i}^{(2)}, \frac{1}{\epsilon}\ell_{ij}^{(2)}, m_{i}^{(2)}, m_{ij}^{(2)},\epsilon h_{i}^{(2)},\epsilon h_{i j}^{(2)}\right]\right\|^{2}.
    \end{aligned}
\end{equation*}
Putting all estimates into \eqref{3eq47}, by Lemmas \ref{lem8} and \ref{lem9} we get
\begin{equation}\label{3eq48}
\begin{aligned}
    \epsilon\frac{d}{d t} \sum_{j=1}^{3} G_{\alpha,j}^4(t)+\left\|\nabla_{x} \partial_{x}^{\alpha} b\right\|^{2} \leqslant&  C \eta \left\{\left\|\nabla_{x} \partial_{x}^{\alpha}(a, b, c)\right\|^{2}+ \left\|\partial_{x}^{\alpha}\{\mathbf{I}-\mathbf{P}\}g\right\|^{2}\right\}\\
    +&\frac{C}{\eta}\left\{\frac{1}{\epsilon^2}\left(\mathcal{D}_N^{mic}\right)^2+\left(\mathcal{E}_N^2+\mathcal{E}_N^4\right) \left(\left(\mathcal{D}_N^{mac}\right)^2+\left(\mathcal{D}_N^{mic}\right)^2\right)\right\}.
\end{aligned}
\end{equation}

{\it Estimates on c.} For each $i=1,2,3$, it follows from \eqref{3eq35} that
\begin{equation}\label{3eq49}
    \begin{aligned}
    &\left\|\partial_{i} \partial_{x}^{\alpha} c\right\|^{2}\\
    =&\epsilon\left(\partial_{i} \partial_{x}^{\alpha} c, \partial_{i} \partial_{x}^{\alpha} c\right) \\
    =&-\epsilon\left(\partial_{t} \partial_{x}^{\alpha} r_{i}^{(3)}, \partial_{i} \partial_{x}^{\alpha}c\right) +\left(\partial_{x}^{\alpha}\left[\frac{1}{\epsilon}\ell_{i}^{(3)}+m_{i}^{(3)}+ \epsilon h_{i}^{(3)}\right], \partial_{i} \partial_{x}^{\alpha} c\right) \\
    =&-\epsilon\frac{d}{d t}\left(\partial_{x}^{\alpha} r_{i}^{(3)}, \partial_{i} \partial_{x}^{\alpha} c\right)+\epsilon\left(\partial_{x}^{\alpha} r_{i}^{(3)}, \partial_{i} \partial_{x}^{\alpha} \partial_{t} c\right)+\left(\partial_{x}^{\alpha}\left[\frac{1}{\epsilon}\ell_{i}^{(3)}+m_{i}^{(3)}+ \epsilon h_{i}^{(3)}\right], \partial_{i} \partial_{x}^{\alpha} c\right).
    \end{aligned}
\end{equation}

The first term on the right hand side above is just $-\epsilon\frac{\d}{\d t}G_{\alpha,i}^3(t)$. From the conservation law \eqref{3eq37}, the second term is bounded by
\begin{equation*}
    \begin{aligned}
   & -\epsilon\left(\partial_{i} \partial_{x}^{\alpha} r_{i}^{(3)}, \partial_{x}^{\alpha} \partial_{t} c\right)\\
    & \leqslant \eta\epsilon^2\left\|\partial_{x}^{\alpha} \partial_{t} c\right\|^{2}+\frac{1}{4 \eta}\left\|\partial_{i} \partial_{x}^{\alpha} r_{i}^{(3)}\right\|^{2} \\
    &\leqslant \eta\left\|\frac{1}{3}\partial_{x}^{\alpha}\nabla\cdot b+\frac{p_0^1}{p_0^1p_4^1-(p_2^1)^2}\partial_{x}^{\alpha}\nabla\cdot\left(|v|^{2} v , \{\mathbf{I}-\mathbf{P}\}g\right) \right\|^{2}+\frac{1}{4 \eta}\left\|\partial_{i} \partial_{x}^{\alpha} r_{i}^{(3)}\right\|^{2} \\
    &\leqslant C \eta\left\|\partial_{x}^{\alpha} \nabla_{x} \cdot b\right\|^{2}+C \eta\left\|\partial_{x}^{\alpha}\nabla\cdot\left(|v|^{2} v , \{\mathbf{I}-\mathbf{P}\}g\right)\right\|^{2}+\frac{1}{4 \eta}\left\|\partial_{i} \partial_{x}^{\alpha} r_{i}^{(3)}\right\|^{2}.
\end{aligned}
\end{equation*}
The third term is bounded by
\begin{equation*}
\eta\left\|\partial_{i} \partial_{x}^{\alpha} c\right\|^{2} +\frac{1}{4\eta}\left\|\partial_{x}^{\alpha}\left[\frac{1}{\epsilon}\ell_{i}^{(3)}+m_{i}^{(3)}+ \epsilon h_{i}^{(3)}\right]\right\|^{2}.
\end{equation*}
Note Lemmas \ref{lem8} and \ref{lem9}, plugging estimates on the three terms into \eqref{3eq49}, then taking summation for $i$ over $i \in\{1,2,3\}$ yields
\begin{equation}\label{3eq50}
    \begin{aligned}
    \epsilon\frac{d}{d t} G_\alpha^3(t)+\left\|\nabla_{x} \partial_{x}^{\alpha} c\right\|^{2} \leqslant& C \eta \left\{\left\|\nabla_{x} \partial_{x}^{\alpha}(b, c)\right\|^{2}+ \left\|\partial_{x}^{\alpha} \nabla_{x}\{\mathbf{I}-\mathbf{P}\}g\right\|^{2}\right\}\\
    &+\frac{C}{\eta}\left\{\frac{1}{\epsilon^2}\left(\mathcal{D}_N^{mic}\right)^2+\left(\mathcal{E}_N^2+\mathcal{E}_N^4\right) \left(\left(\mathcal{D}_N^{mac}\right)^2+\left(\mathcal{D}_N^{mic}\right)^2\right)\right\}.
\end{aligned}
\end{equation}

{\it Estimates on $a$.} For each $i=1,2,3$, it follows from \eqref{3eq32} that
\begin{equation}\label{3eq51}
    \begin{aligned}
    &\left\|\partial_{i} \partial_{x}^{\alpha} a\right\|^{2}
    =\left(-\epsilon\partial_{x}^{\alpha}\partial_{t} b_{i}-\epsilon\partial_{x}^{\alpha} \partial_{t}r_{i}^{(1)} +\partial_{x}^{\alpha}\left[\frac{1}{\epsilon}\ell_{i}^{(1)}+m_{i}^{(1)}+ \epsilon h_{i}^{(1)}\right], \partial_{i} \partial_{x}^{\alpha} a\right) \\
    =&-\epsilon\frac{d}{d t}\left[\left(\partial_{x}^{\alpha} b_{i}, \partial_{i} \partial_{x}^{\alpha} a\right)+\left(\partial_{x}^{\alpha} r_{i}^{(1)}, \partial_{i} \partial_{x}^{\alpha} a\right)\right] \\
    +&\epsilon\left(\partial_{x}^{\alpha} b_{i}, \partial_{i} \partial_{x}^{\alpha} \partial_{t} a\right)+\epsilon\left(\partial_{x}^{\alpha} r_{i}^{(1)}, \partial_{i} \partial_{x}^{\alpha} \partial_{t} a\right)+\left(\partial_{x}^{\alpha}\left[\frac{1}{\epsilon}\ell_{i}^{(1)}+m_{i}^{(1)}+ \epsilon h_{i}^{(1)}\right], \partial_{i} \partial_{x}^{\alpha} a\right).
\end{aligned}
\end{equation}
Similarly as before, we need to estimate the four terms on the right-hand side of \eqref{3eq51}. The first term is $-\epsilon\frac{\d}{\d t}\left[G_\alpha^1(t)+G_\alpha^2(t)\right]$. From the conservation law \eqref{3eq37},
\begin{equation*}
    \begin{aligned}
    &\epsilon\left(\partial_{x}^{\alpha} \partial_{i} b_{i}, \partial_{x}^{\alpha} \partial_{t} a\right)+\epsilon\left(\partial_{x}^{\alpha} \partial_{i} r_{i}^{(1)}, \partial_{x}^{\alpha} \partial_{t} a\right)\\
    =&\left(\partial_{x}^{\alpha} \partial_{i} b_{i}, \frac{p_2^1}{p_0^1p_4^1-(p_2^1)^2}\partial_{x}^{\alpha} \nabla_{x} \cdot\left(|v|^{2} v , \{\mathbf{I}-\mathbf{P}\}g\right)\right)\\
    &+\left(\partial_{x}^{\alpha} \partial_{i} r_{i}^{(1)}, \frac{p_2^1}{p_0^1p_4^1-(p_2^1)^2} \partial_{x}^{\alpha} \nabla_{x} \cdot\left(|v|^{2} v , \{\mathbf{I}-\mathbf{P}\}g\right)\right) \\
    \leqslant & \eta\left(\left\|\partial_{x}^{\alpha} \partial_{i} b_{i}\right\|^{2}+\left\|\partial_{x}^{\alpha} \partial_{i} r_{i}^{(1)}\right\|^{2}\right)+\frac{C}{\eta}\left\|\partial_{x}^{\alpha} \nabla_{x} \cdot\left(|v|^{2} v , \{\mathbf{I}-\mathbf{P}\}g\right)\right\|^{2}.
\end{aligned}
\end{equation*}
The final term is bounded by
\begin{equation*}
    \eta\left\|\partial_{i} \partial_{x}^{\alpha} a\right\|^{2} +\frac{C}{\eta}\left\|\partial_{x}^{\alpha}\left[\frac{1}{\epsilon}\ell_{i}^{(1)}, m_{i}^{(1)} , \epsilon h_{i}^{(1)}\right]\right\|^{2}.
\end{equation*}
Collecting all the estimates above, by Lemmas \ref{lem8} and \ref{lem9} we get
\begin{equation}\label{3eq52}
    \begin{aligned}
    &\epsilon\frac{d}{d t}\left[G_\alpha^1(t)+G_\alpha^2(t)\right]+\left\|\nabla_{x} \partial_{x}^{\alpha} a\right\|^{2} \leqslant \eta\left\|\nabla_{x} \partial_{x}^{\alpha}(a, b)\right\|^{2}+\frac{C}{\eta}\left\|\partial_{x}^{\alpha} \nabla_{x}\{\mathbf{I}-\mathbf{P}\}g\right\|^{2}\\
    +&\frac{C}{\eta}\left\{\frac{1}{\epsilon^2}\left(\mathcal{D}_N^{mic}\right)^2+\left(\mathcal{E}_N^2+\mathcal{E}_N^4\right) \left(\left(\mathcal{D}_N^{mac}\right)^2+\left(\mathcal{D}_N^{mic}\right)^2\right)\right\}.
    \end{aligned}
\end{equation}

Finally we add up the inequalities \eqref{3eq48}, \eqref{3eq50} and \eqref{3eq52} and take summation for $\alpha$ over $|\alpha|\leqslant N-1$ to obtain
\begin{equation}\label{3eq53}
    \begin{aligned}
    \epsilon\frac{d}{d t} & \sum_{|\alpha| \leqslant N-1} \sum_{i=1}^{3}G_\alpha^i(t) +\sum_{|\alpha| \leqslant N-1}\left\|\nabla_{x} \partial_{x}^{\alpha}(a, b, c)\right\|^{2} \\
    \leqslant  C \eta& \sum_{|\alpha| \leqslant N-1}\left\|\nabla_{x} \partial_{x}^{\alpha}(a, b, c)\right\|^{2}+\frac{C}{\eta}\left\{\frac{1}{\epsilon^2}\left(\mathcal{D}_N^{mic}\right)^2+\left(\mathcal{E}_N^2+\mathcal{E}_N^4\right) \left(\left(\mathcal{D}_N^{mac}\right)^2+\left(\mathcal{D}_N^{mic}\right)^2\right)\right\}
    \end{aligned}
\end{equation}
We complete the proof by choosing $\eta \in (0,1)$, such that $C\eta =\frac{1}{2}$.
\end{proof}

\subsection{Uniform estimate and the global existence}
We derive the uniform energy estimate by the microscopic estimate \eqref{3eq1} and the macroscopic estimate \eqref{3eq45}. For constant $d >0$ suitably large to be determined later, we define
\begin{equation*}
    \mathcal{E}_{N,d}^2(t)\equiv d\, \mathcal{E}_{N}^2(t)+\epsilon\sum_{|\alpha|\leqslant N-1}G_\alpha(t),
\end{equation*}
multiplying \eqref{3eq1} by $d$, and then adding it to \eqref{3eq45}, since $\left(\mathcal{D}_N^{mac}\right)(g)$ is equivalent with $\sum\limits_{|\alpha| \leqslant N-1}\left\|\nabla_{x} \partial_{x}^{\alpha}(a, b, c)\right\|$, we have
\begin{equation*}
    \begin{aligned}
    \frac{\d}{\d t}\mathcal{E}_{N,d}^2
    +\left\{\frac{d}{\epsilon^2}\left(\mathcal{D}_N^{mic}\right)^2+\left(\mathcal{D}_N^{mac}\right)^2\right\}
    \leqslant &Cd \left\{\frac{1}{\epsilon}\left(\mathcal{E}_N+\mathcal{E}_N^2\right)\left(\mathcal{D}_N^{mic}\right)^2 +\left(\mathcal{E}_N^2+\mathcal{E}_N^4\right)\left(\mathcal{D}_N^{mac}\right)^2\right\}\\
    &+\widetilde{C}\left\{\frac{1}{\epsilon^2}\left(\mathcal{D}_N^{mic}\right)^2+\left(\mathcal{E}_N^2+\mathcal{E}_N^4\right) \left(\left(\mathcal{D}_N^{mac}\right)^2+\left(\mathcal{D}_N^{mic}\right)^2\right)\right\},
    \end{aligned}
  \end{equation*}
which is
\begin{equation}\label{3eq54}
    \begin{aligned}
    &\frac{\d}{\d t}\mathcal{E}_{N,d}^2
    +\left\{\frac{d-\widetilde{C}}{\epsilon^2}\left(\mathcal{D}_N^{mic}\right)^2+\left(\mathcal{D}_N^{mac}\right)^2\right\}\\
    \leqslant &\frac{Cd+\widetilde{C}}{\epsilon} \left(\mathcal{E}_N+\mathcal{E}_N^2+\mathcal{E}_N^4\right)\left(\mathcal{D}_N^{mic}\right)^2+ \left(\widetilde{C}+Cd\right)\left(\mathcal{E}_N^2+\mathcal{E}_N^4\right)\left(\mathcal{D}_N^{mac}\right)^2.
    \end{aligned}
  \end{equation}

We point out that for some constant $c_1,c_2>0$, there holds
\begin{equation}\label{3eq55}
  c_1\mathcal{E}_{N}(g)(t)\leqslant \mathcal{E}_{N,d}(g)(t)\leqslant c_2\mathcal{E}_{N}(g)(t).
\end{equation}
Indeed, this is due to
\begin{equation}\label{3eq56}
  \begin{aligned}
    \sum_{|\alpha|\leqslant N-1}|G_\alpha(t)|\leqslant &C \sum_{|\alpha|\leqslant N-1}\sum_{i=1}^3G_\alpha^i (t)\\
    \leqslant & C\sum_{|\alpha|\leqslant N-1} \left\|\partial_x^{\alpha}\{\mathbf{I}-\mathbf{P}\}g\right\|^{2}
    +C\sum_{|\alpha|\leqslant N-1}\left\|\partial_x^{\alpha}\mathbf{P}g\right\|^{2}\\
    \leqslant & C\mathcal{E}_N^2(g)(t).
  \end{aligned}
\end{equation}
Choosing $d >0$ suitably such that $d -\widetilde{C} > 0$, then the relation \eqref{3eq55} holds. Hence we prove the following theorem
\begin{theorem}\label{thGEE}
  If $g$ is a solution of the scaled Boltzmann-Fermi-Dirac equation \eqref{3eq32}, then there exists a constant $c_0>0$ independent of $\epsilon$ such that if $\mathcal{E}_{N,d}^2\leqslant 1$, then
  \begin{equation}\label{3eq57}
  \frac{\d}{\d t}\mathcal{E}_{N,d}^2+\frac{1}{\epsilon^2}\left(\mathcal{D}_N^{mic}\right)^2+\left(\mathcal{D}_N^{mac}\right)^2\leqslant c_0\mathcal{E}_{N,d}\left\{\frac{1}{\epsilon}\left(\mathcal{D}_N^{mic}\right)^2+\left(\mathcal{D}_N^{mac}\right)^2\right\}.
  \end{equation}
\end{theorem}

Now, with the above preparations, we are ready to prove Theorem \ref{th1} by the usual continuation arguments.

\begin{proof}[{\bf Proof of Theorem \ref{th1}}]
  Recall the constants $c_0, c_1, c_2$ in \eqref{3eq57} and \eqref{3eq55}, and $\delta_0$ appeared in Theorem \ref{lex}, we can define
  \begin{equation}\label{3eq58}
    M=\min\left\{\frac{c_1\delta_0}{2c_2},\frac{1}{2c_2},\frac{1}{4c_0c_2}\right\}.
  \end{equation}
  Let the initial data $g_{0,\epsilon}$ satisfies
  \begin{equation*}
    \mathcal{E}_N(0)=\left\|g_{0,\epsilon}\right\|_{H_x^NL_v^2}\leqslant M.
  \end{equation*}
   Then Theorem \ref{lex} shows that for some constant $T > 0$, there exists a solution $g\in L^{\infty}\left([0,T];\,H_x^NL_v^2\right)$ satisfying $\mathcal{E}_N(t)\leqslant 2M$ for $0 < t < T$. We further define
  \begin{equation*}
    T^*=\sup\left\{t\in \mathbb{R}^+\,\big|\,\mathcal{E}_N(t)\leqslant2M\frac{c_2}{c_1}\right\}>0.
  \end{equation*}
  Note that

  \begin{equation*}
    \mathcal{E}_{N,d}(t)\leqslant c_2\mathcal{E}_N(t)\leqslant 2c_2M<1, \quad\forall 0\leqslant t\leqslant T.
  \end{equation*}
  The global energy estimate \eqref{3eq57} implies that
  \begin{equation*}
    \frac{\d}{\d t}\mathcal{E}_{N,d}^2+ (1-2c_0c_2M)\left\{\frac{1}{\epsilon^2}\left(\mathcal{D}_N^{mic}\right)^2+\left(\mathcal{D}_N^{mac}\right)^2\right\}\leqslant 0,
  \end{equation*}
  from the choice of $M$ such that $1-2c_0c_2M > \frac{1}{2}$. Thus
  \begin{equation*}
    \mathcal{E}_{N,d}^2(T)+ \frac{1}{2}\int_{0}^{T}\left\{\frac{1}{\epsilon^2}\left(\mathcal{D}_N^{mic}\right)^2+\left(\mathcal{D}_N^{mac}\right)^2\right\}dt\leqslant \mathcal{E}_{N,d}^2(0),
  \end{equation*}
  which implies $\displaystyle\mathcal{E}_N(T)\leqslant \frac{c_2}{c_1}M$. Thus $T^*=+\infty$, and we finish the proof of Theorem \ref{th1}.
\end{proof}

\newpage

\section{Limit to Incompressible Navier-Stokes-Fourier Equations}\label{sec4}
\subsection{The limit from the global energy estimate}
First of all, we shall introduce some constants:
\begin{equation*}
  \begin{aligned}
  E_{0}=&\langle 1\rangle, &\quad E_{2}=&\left\langle\left|v_{1}\right|^{2}\right\rangle, \\
  E_{4}=&\left\langle\left|v_{1}\right|^{4}\right\rangle, &\quad E_{22}=&\left\langle\left|v_{1} v_{2}\right|^{2}\right\rangle,\\
  C_{A}=\left\langle\left(\frac{|v|^{2}}{2}-K_{A}\right)^{2} v_{1}^{2}\right\rangle_{1-2\mu}, &\quad K_A=&\frac{E_4+2E_{22}}{2E_2}.
  \end{aligned}
\end{equation*}
Based on Theorem \ref{th1}, there exists a $\delta_0 > 0$, such that for any given initial data
\begin{equation*}
  g_{\epsilon,0}(x,v)=\left\{\rho_0(x)+\mathrm{u}_0(x)\cdot v+\theta_0(x)\left(\frac{|v|^2}{2}-K_g\right)\right\}+\widetilde{g}_{\epsilon,0}(x,v),
\end{equation*}
satisfying
\begin{equation*}
  \left\|\left(\rho_{0}, \mathrm{u}_{0}, \theta_{0}\right)\right\|_{H_x^N} \leqslant \frac{\delta_0}{2C_{E_0,E_2,E_4}},
\end{equation*}
where $C_{E_0,E_2,E_4}>1$ is a constant and
\begin{equation*}
  \widetilde{g}_{\epsilon, 0} \in Null(L)^{\bot}, \quad\left\|\widetilde{g}_{\epsilon, 0}\right\|_{H_x^NL_v^2} \leqslant \frac{\delta_0}{2},
\end{equation*}
the scaled Boltzmann-Fermi-Dirac equation \eqref{1eq8} admits a global solution $g_{\epsilon}$, where $K_g=K_A-1$. Furthermore, there is a constant $C>0$ independent of $\epsilon$ such that the global energy estimate \eqref{1eq18} holds, thus
\begin{equation}\label{4eq1}
  \sup _{t \geqslant 0} \mathcal{E}_{N}^{2}(t)=\sup _{t \geqslant 0} \sum_{|\alpha| \leqslant N} \int_{\mathbb{R}_{x, v}^{6}}\left|\partial_{x}^{\alpha} g_{\epsilon}(t)\right|^{2} \mu(1-\mu)\d v \d  x \leqslant C,
\end{equation}
and
\begin{equation}\label{4eq2}
   \int_{0}^{\infty} \left(\mathcal{D}_{N}^{mic}\right)^{2}(t) \d  t=\sum_{|\alpha| \leqslant N} \int_{0}^{\infty} \int_{\mathbb{R}_{x, v}^{6}}\left|\partial_{x}^{\alpha}\{\mathbf{I}-\mathbf{P}\} g_{\epsilon}(t)\right|^{2} \nu\mu(1-\mu)\d v\d  x \d t \leqslant C \epsilon^{2},
\end{equation}
and
\begin{equation}\label{4eq3}
  \int_{0}^{\infty} \left(\mathcal{D}_{N}^{mac}\right)^{2}(t) \d  t=\sum_{0<|\alpha|\leqslant N} \int_{0}^{\infty} \int_{\mathbb{R}_{x, v}^{6}}\left|\partial_{x}^{\alpha} \mathbf{P} g_{\epsilon}(t)\right|^{2} \nu\mu(1-\mu)\d v \d x \d t \leq C.
\end{equation}

We deduce from the energy bound \eqref{4eq1} that there exists a $g_{0}$ in the functional space $L^{\infty}\left([0,+\infty) ; H^{N}\left(\d x ; L^{2}(\mu(1-\mu)\d v)\right)\right),$ such that
\begin{equation}\label{4eq4}
  g_{\epsilon} \rightarrow g_{0} \quad \text{ as }\quad \epsilon \rightarrow 0,
\end{equation}
where the convergence is weak-$\star$ for $t,$ strongly in $H^{N-\eta}\left(\d x\right)$ for any $\eta>0,$ and weakly in $L^{2}(\mu(1-\mu)\d v)$.

From the energy dissipation bound \eqref{4eq2} we have
\begin{equation}\label{4eq5}
  \{\mathbf{I}-\mathbf{P}\} g_{\epsilon} \rightarrow 0, \text{ in } L^{2}\left([0,+\infty) ; H^{N}\left(\d x ; L^{2}(\mu(1-\mu)\d v)\right)\right) \text{ as } \epsilon \rightarrow 0.
\end{equation}
By combining the convergence \eqref{4eq4} and \eqref{4eq5} we have $\{\mathbf{I}-\mathbf{P}\} g_{0}=0 .$ Thus, there exists $(\rho, \mathrm{u}, \theta) \in L^{\infty}\left([0,+\infty) ; H^{N}\left(\d x\right)\right),$ such that
\begin{equation*}
  g_{0}(t, x, v)=\rho(t, x)+\mathrm{u}(t, x) \cdot v+\theta(t, x)\left(\frac{|v|^{2}}{2}-K_g\right).
\end{equation*}

\subsection{The limiting equations}
If we write
\begin{equation*}
  A(v)=\left(\frac{|v|^2}{2}-K_A\right)v,\qquad B(v)=v\otimes v-\frac{|v|^2}{3}I,
\end{equation*}
then from \cite{Zakrevskiy}, there exist unique $A^\prime(v)=\Big(A_i^\prime(v)\Big)$, $B^\prime(v)= \Big(B_{ij}^\prime(v)\Big),\;i,j=1,2,3$, such that
\begin{equation*}
  L(A_i^\prime)=A_i^\prime,\qquad L(B_{ij}^\prime)=B_{ij}.
\end{equation*}
Moreover,
\begin{equation*}
  A^\prime(v)=-\alpha_{L}(|v|)A(v),\qquad B^\prime(v)=-\beta_{L}(|v|)B(v).
\end{equation*}
for some positive functions $\alpha_{L}(|v|),\;\beta_{L}(|v|)$.

Now we define the fluid variables as follows:
\begin{equation*}
\left\{
  \begin{aligned}
  &\rho_{\epsilon}=\frac{1}{K_AE_0-3E_2/2}\left\langle g_{\epsilon}, \left[1+\left(\frac{2E_0}{3E_2}K_g-1\right)\frac{|v|^2}{2}\right]\right\rangle, \\
  &\mathrm{u}_{\epsilon}=\frac{1}{E_2}\left\langle g_{\epsilon}, v\right\rangle, \\
  &\theta_{\epsilon}=\frac{1}{K_AE_0-3E_2/2}\left\langle g_{\epsilon}, \left(\frac{E_0}{3E_2}|v|^{2}-1\right)\right\rangle.
  \end{aligned}
\right.
\end{equation*}
Then it follows from \eqref{4eq4} that
\begin{equation}\label{4eq6}
  \left(\rho_{\epsilon}, \mathrm{u}_{\epsilon}, \theta_{\epsilon}\right) \rightarrow(\rho, \mathrm{u}, \theta) \text{ as } \epsilon \rightarrow 0,
\end{equation}
where the convergence is weak-$\star$ for $t,$ strongly in $H^{N-\eta}\left(\d x\right)$ for any $\eta>0$.

Taking inner products with the Boltzmann-Fermi-Dirac equation \eqref{1eq8} in $L_{v}^{2}$ by
\begin{equation*}
  \frac{1+\left(\frac{2E_0}{3E_2}K_g-1\right)\frac{|v|^2}{2}}{K_AE_0-3E_2/2}, \qquad \frac{v}{E_2}, \qquad \frac{\frac{E_0}{3E_2}|v|^{2}-1}{K_AE_0-3E_2/2}
\end{equation*}
respectively gives the local conservation laws:
\begin{equation}\label{4eq7}
  \left\{
  \begin{aligned}
  &\partial_{t} \rho_{\epsilon}+\frac{2K_g}{3\epsilon}\nabla_{x} \cdot \mathrm{u}_{\epsilon}+\frac{\frac{2E_0}{3E_2}K_g-1}{K_AE_0-3E_2/2}\nabla_{x} \cdot \left\langle A^\prime, \frac{1}{\epsilon} L g_{\epsilon}\right\rangle=0, \\
  &\partial_{t} \mathrm{u}_{\epsilon}+\frac{1}{\epsilon} \nabla_{x}\left(\rho_{\epsilon}+\theta_{\epsilon}\right)+\frac{1}{E_2}\nabla_{x} \cdot\left\langle B^\prime, \frac{1}{\epsilon} L g_{\epsilon}\right\rangle=0, \\
  &\partial_{t} \theta_{\epsilon}+\frac{2}{3\epsilon}\nabla_{x} \cdot \mathrm{u}_{\epsilon}+\frac{2E_0/3E_2}{K_AE_0-3E_2/2} \nabla_{x} \cdot\left\langle A^\prime, \frac{1}{\epsilon} L g_{\epsilon}\right\rangle=0.
  \end{aligned}
  \right.
\end{equation}
Here we use the properties of $A,B,A^\prime\text{ and }B^\prime$, the self-adjointness of $L$, to obtain the following calculations:
\begin{equation*}
  \begin{aligned}
  \left\langle v\cdot\nabla_x g_\epsilon,v\right\rangle
  &=\nabla_x\cdot\left\langle B,g_\epsilon\right\rangle+\nabla_x \left\langle \frac{|v|^2}{3}g_\epsilon, 1\right\rangle\\
  &=\nabla_x\cdot\left\langle LB^\prime,g_\epsilon\right\rangle +E_2\nabla_x(\rho_\epsilon+\theta_\epsilon)\\
  &=\nabla_x\cdot\left\langle B^\prime,Lg_\epsilon\right\rangle +E_2\nabla_x(\rho_\epsilon+\theta_\epsilon),
  \end{aligned}
\end{equation*}
and
\begin{equation*}
  \begin{aligned}
  &\left\langle v\cdot\nabla_x g_\epsilon,\frac{1}{K_AE_0-3E_2/2}\left(\frac{E_0}{3E_2}|v|^{2}-1\right) \right\rangle\\
  =&\frac{1}{K_AE_0-3E_2/2}\nabla_x\cdot\left\langle g_\epsilon, \frac{2E_0}{3E_2}\left(\frac{|v|^{2}}{2}-K_A\right)v+ \left(\frac{2E_0}{3E_2}K_A-1\right)v\right\rangle\\
  =&\frac{2E_0/3E_2}{K_AE_0-3E_2/2}\nabla_x\cdot\left\langle A,g_\epsilon\right\rangle +\frac{2}{3}\nabla_x\cdot\mathrm{u}_{\epsilon}\\
  =&\frac{2E_0/3E_2}{K_AE_0-3E_2/2} \nabla_{x} \cdot\left\langle A^\prime, L g_{\epsilon}\right\rangle +\frac{2}{3}\nabla_x\cdot\mathrm{u}_{\epsilon}.
  \end{aligned}
\end{equation*}

{\bf Incompressibility and Boussinesq relation. } From the first equation of \eqref{4eq7},
\begin{equation*}
  \frac{2K_g}{3} \nabla_{x} \cdot \mathrm{u}_{\epsilon}=-\epsilon\partial_{t} \rho_{\epsilon}-\frac{\frac{2E_0}{3E_2}K_g-1}{K_AE_0-3E_2/2}\nabla_{x} \cdot \left\langle LA^\prime, \{\mathbf{I}- \mathbf{P}\}g_{\epsilon}\right\rangle.
\end{equation*}
From the global energy bound \eqref{4eq1} and the global energy dissipation \eqref{4eq2}, it is easy to deduce
\begin{equation}\label{4eq8}
  \nabla_x\cdot \mathrm{u}_{\epsilon}\rightarrow 0 \text{ in the sense of distributions as } \epsilon\rightarrow 0.
\end{equation}
By combining with the convergence \eqref{4eq6}, we have
\begin{equation}\label{4eq9}
  \nabla_{x} \cdot \mathrm{u}=0.
\end{equation}
From the second equation of \eqref{4eq7},
\begin{equation*}
    \nabla_{x}\left(\rho_{\epsilon}+\theta_{\epsilon}\right)=-\epsilon \partial_{t} \mathrm{u}_{\epsilon}-\frac{1}{E_2}\nabla_{x} \cdot\left(LB^\prime,\{\mathbf{I}-\mathbf{P}\} g_{\epsilon}\right).
\end{equation*}
From the global energy dissipation \eqref{4eq2}, it follows that
\begin{equation}\label{4eq10}
  \nabla_{x}\left(\rho_{\epsilon}+\theta_{\epsilon}\right) \rightarrow 0 \text{ in the sense of distributions as } \epsilon \rightarrow 0,
\end{equation}
which gives the Boussinesq relation
\begin{equation}\label{4eq11}
    \nabla_{x}(\rho+\theta)=0.
\end{equation}

{\bf Convergence of $\displaystyle\frac{K_g\theta_\epsilon-\rho_\epsilon}{K_g+1}$.} The third equation times $K_g$ and then minus  the first equation in \eqref{4eq7} gives
\begin{equation}\label{4eq12}
  \partial_t\left(\frac{K_g\theta_\epsilon-\rho_\epsilon}{K_g+1}\right) +\frac{1}{K_A(K_AE_0-3E_2/2)}\nabla_{x} \cdot \left\langle A^\prime, \frac{1}{\epsilon} L g_{\epsilon}\right\rangle=0.
\end{equation}
From the global energy estimate \eqref{4eq1}, we have that
\begin{equation*}
  \left\|\frac{K_g\theta_\epsilon-\rho_\epsilon}{K_g+1}(t)\right\|_{H^N(\d x)}\leqslant C \text{ for almost every } t\in [0,+\infty),
\end{equation*}
Then there exists a $\widetilde{\theta}\in L^\infty\left([0,+\infty);H^N(\d x)\right)$, so that
\begin{equation*}
  \frac{K_g\theta_\epsilon-\rho_\epsilon}{K_g+1}(t)\rightarrow  \widetilde{\theta}(t) \text{ in } H^{N-\eta}(\d x),
\end{equation*}
for any $\eta>0$ as $\epsilon\to 0$. Furthermore, using the equation \eqref{4eq12}, we can show the equi-continuity in $t$. Indeed, for any $\left[t_{1}, t_{2}\right] \subset[0, \infty),$ and any test function $\chi(x)$ and $|\alpha| \leqslant N-1$
\begin{equation*}
\begin{aligned}
    & \int_{\mathbb{R}^3}\left[\partial_{x}^{\alpha} \left(\frac{K_g\theta_\epsilon-\rho_\epsilon}{K_g+1}\left(t_{2}\right)\right) -\partial_{x}^{\alpha}\left(\frac{K_g\theta_\epsilon-\rho_\epsilon}{K_g+1}\right)\left(t_{1}\right) \right] \chi(x) \d  x \\
    =& -\frac{1}{K_A(K_AE_0-3E_2/2)}\int_{t_{1}}^{t_{2}} \int_{\mathbb{R}^3}\left\langle A^\prime, \frac{1}{\epsilon} L\{\mathbf{I}-\mathbf{P}\} \nabla_{x} \partial_{x}^{\alpha} g_{\epsilon}\right\rangle \chi(x) \d  x \d  t\\
    \leqslant & \frac{C}{\epsilon} \left(\int_{t_{1}}^{t_{2}}\left(\mathcal{D}_{N}^{mic}\right) ^{2}\left(g_{\epsilon}\right)(t) \d  t\right)^{1/2}.
\end{aligned}
\end{equation*}
Thus the energy dissipation estimate \eqref{4eq2} implies the equi-continuity in $t .$ From the Arzelà-Ascoli Theorem,
\begin{equation*}
  \widetilde{\theta} \in C\left([0, \infty) ; H^{N-1-\eta}\left(\d x\right)\right) \cap L^{\infty}\left([0, \infty) ; H^{N-\eta}\left(\d x\right)\right),
\end{equation*}
and
\begin{equation}\label{4eq13}
  \frac{K_g\theta_\epsilon-\rho_\epsilon}{K_g+1} \rightarrow \widetilde{\theta}
  \qquad\text{ in } C\left([0, \infty) ; H^{N-1-\eta}\left(\d x\right)\right) \cap L^{\infty}\left([0, \infty) ; H^{N-\eta}\left(\d x\right)\right)
\end{equation}
as $\varepsilon \rightarrow 0$ for any $\eta>0$. Note that $\widetilde{\theta}=\frac{K_g\theta-\rho}{K_g+1}$ and $\theta=\frac{K_g\theta-\rho}{K_g+1}+\frac{1}{K_g+1}(\rho+\theta),$ and the relation \eqref{4eq11}, we get $\widetilde{\theta}=\theta$ and $\rho+\theta=0$

{\bf Convergence of $\mathcal{P} u_{\varepsilon}$.} Taking the Leray projection operator $\mathcal{P}$ on the second equation of \eqref{4eq7} gives
\begin{equation*}
  \partial_{t} \mathcal{P} \mathrm{u}_{\epsilon}+\frac{1}{E_2}\mathcal{P} \nabla_{x} \cdot\left\langle B^\prime, \frac{1}{\epsilon} L g_{\epsilon}\right\rangle=0.
\end{equation*}

Similar arguments as above deduce that there exists a divergence free $\widetilde{\mathrm{u}} \in L^{\infty}\left([0, \infty) ; H^{N}\left(\d x\right)\right)$ such that
\begin{equation}\label{4eq14}
  \mathcal{P} \mathrm{u}_{\epsilon} \rightarrow \widetilde{\mathrm{u}} \quad \text { in } \quad C\left([0, \infty) ; H^{N-1-\eta}\left(\d x\right)\right) \cap L^{\infty}\left([0, \infty) ; H^{N-\eta}\left(\d x\right)\right)
\end{equation}
as $\epsilon \rightarrow 0$ for any $\eta>0 .$ Note that $\widetilde{\mathrm{u}}=\mathcal{P} \mathrm{u}$ and \eqref{4eq9}, we have $\widetilde{\mathrm{u}}=\mathrm{u}$.

Similar to the standard calculations in \cite{ref6}, the local conservation laws can be rewritten as
\begin{equation}\label{4eq15}
  \left\{\begin{array}{l}
    \partial_{t} \rho_{\epsilon}+\frac{2K_g}{3}\frac{1}{\epsilon} \nabla_{x} \cdot \mathrm{u}_{\epsilon}+K_A\left(\frac{2K_gE_0}{3E_2}-1\right)\nabla_{x} \cdot\left(\mathrm{u}_{\epsilon} \theta_{\epsilon}\right)=\kappa_1\nabla_{x} \cdot\left[\nabla_{x} \theta_{\epsilon}\right]+\nabla_{x} \cdot R_{\epsilon, \theta}^{(1)}\\
    \partial_{t} \mathrm{u}_{\epsilon}+\frac{1}{\epsilon} \nabla_{x}\left(\rho_{\epsilon}+\theta_{\epsilon}\right)+\nabla_{x} \cdot\left(\mathrm{u}_{\epsilon} \otimes \mathrm{u}_{\epsilon}-\frac{\left|\mathrm{u}_{\epsilon}\right|^{2}}{3} I\right)= \frac{\nu_{*}}{E_2}\Sigma(\mathrm{u}_\epsilon)+\nabla_{x} \cdot R_{\epsilon, \mathrm{u}}\\
    \partial_{t} \theta_{\epsilon}+\frac{2}{3} \frac{1}{\epsilon} \nabla_{x} \cdot \mathrm{u}_{\epsilon}+\frac{2K_AE_0}{3E_2}\nabla_{x} \cdot\left(\mathrm{u}_{\epsilon} \theta_{\epsilon}\right)=\kappa_2 \nabla_{x} \cdot\left[\nabla_{x} \theta_{\epsilon}\right]+\nabla_{x} \cdot R_{\epsilon, \theta}^{(2)}
\end{array}\right.
\end{equation}
where
\begin{equation*}
  \begin{aligned}
  \kappa_1&=\frac{\frac{2E_0}{3E_2}K_g-1}{K_AE_0-3E_2/2}\left\langle\alpha_L(|v|) \left(\frac{|v|^2}{2}-K_A\right)^2v_1^2\right\rangle,\\
  \kappa_2&=\frac{2E_0/3E_2}{K_AE_0-3E_2/2} \left\langle\alpha_L(|v|)\left(\frac{|v|^2}{2}-K_A\right)^2v_1^2\right\rangle,\\
  \Sigma(u)&=\nabla u+(\nabla u)^T-\frac{2}{3}(\nabla\cdot u)I,
  \end{aligned}
\end{equation*}
and $R_{\epsilon, \mathrm{u}}, R_{\epsilon, \theta}^{(1)}, R_{\epsilon, \theta}^{(2)}$ have of the form
\begin{equation}\label{4eq16}
  \begin{aligned}
  &R+S\left(\epsilon\Big\langle\zeta(v), \partial_{t}g_{\epsilon}-T(g_{\epsilon},g_{\epsilon},g_{\epsilon})\Big\rangle\right.\\
  &+\Big\langle\zeta(v), v \cdot \nabla_{x}\{\mathbf{I}-\mathbf{P}\}g_{\epsilon}\Big\rangle -\Big\langle\zeta(v), Q(\{\mathbf{I}-\mathbf{P}\} g_{\epsilon},\{\mathbf{I}-\mathbf{P}\} g_{\epsilon})\Big\rangle\\
  &\left.-\Big\langle\zeta(v), Q(\{\mathbf{I}-\mathbf{P}\} g_{\epsilon}, \mathbf{P} g_{\epsilon})\Big\rangle -\Big\langle\zeta(v), Q(\mathbf{P} g_{\epsilon},\{\mathbf{I}-\mathbf{P}\} g_{\epsilon})\Big\rangle\right).
  \end{aligned}
\end{equation}
For $R_{\epsilon, \theta}^{(1)},$ we take $\zeta(v)= A^\prime$, $R=-K_A\left(\frac{2K_gE_0}{3E_2}-1\right)\mathrm{u}_{\epsilon}(\rho_{\epsilon}+\theta_{\epsilon})$ and $S=\frac{\frac{2E_0}{3E_2}K_g-1}{K_AE_0-3E_2/2}$ ; For $R_{\epsilon, \theta}^{(1)},$ we take $\zeta(v)=A^\prime$, $R=-\frac{2K_AE_0}{3E_2}\mathrm{u}_{\epsilon}(\rho_{\epsilon}+\theta_{\epsilon})$ and $S=\frac{2E_0/3E_2}{K_AE_0-3E_2/2}$ and for $R_{\epsilon, \mathrm{u}}$, we take $\zeta(v)= B^\prime$, $R=0$ and $S=\frac{1}{E_2}$.

{\bf The equations of $\theta$ and $\mathrm{u}$.} Decompose $u_{\epsilon}=\mathcal{P} u_{\epsilon}+\mathcal{Q} u_{\epsilon},$ where $\mathcal{Q}=\nabla_{x} \Delta_{x}^{-1} \nabla_{x} \cdot$ is a gradient. Denote $\widetilde{\theta}_{\epsilon}=\displaystyle\frac{K_g\theta_\epsilon-\rho_\epsilon}{K_g+1}$. Then from \eqref{4eq15}, the following equation is satisfied in the sense of distributions:
\begin{equation*}
  \partial_{t} \widetilde{\theta}_{\epsilon}+\nabla_{x} \cdot\left(\mathcal{P} \mathrm{u}_{\epsilon} \widetilde{\theta}_{\epsilon}\right)-\frac{3E_2}{2K_AE_0} \kappa_2 \Delta_{x} \widetilde{\theta}_{\epsilon}=\frac{1}{C_A}\nabla_{x} \cdot \widetilde{R}_{\epsilon, \theta},
\end{equation*}
where
\begin{equation}\label{4eq17}
  \begin{aligned}
  \widetilde{R}_{\epsilon, \theta}=&\frac{1}{C_A} R_{\epsilon, \theta}
  -\left(\frac{1}{K_A}+1\right)\mathcal{P} \mathrm{u}_{\epsilon}\left(\rho_{\epsilon}+\theta_{\epsilon}\right)\\
  &-\left(\frac{1}{K_A}+1\right)\mathcal{Q} \mathrm{u}_{\epsilon}\left(\rho_{\epsilon}+\theta_{\epsilon}\right)-\mathcal{Q} \mathrm{u}_{\epsilon} \widetilde{\theta}_{\epsilon}
  +\frac{3E_2}{2K_AE_0} \kappa_2 \nabla_{x}\left(\rho_{\epsilon}+\theta_{\epsilon}\right).
  \end{aligned}
\end{equation}

For any $T>0$, let $\varphi(t, x)$ be a text function satisfying
\begin{equation*}
  \varphi(t, x) \in C^{1}\left([0, T], C_{c}^{\infty}\left(\mathbb{R}_x^3\right)\right)
   \text{ with } \varphi(0, x)=1
\end{equation*}
and
\begin{equation*}
  \varphi(t, x)=0 \qquad \text{ for } t \geqslant T^{\prime}, \text{ where } T^{\prime}<T .
\end{equation*}
Note \eqref{4eq16} and use the global bounds \eqref{4eq1}, \eqref{4eq2}, and \eqref{4eq3}. It is easy to show that
\begin{equation}\label{4eq18}
  \int_{0}^{T} \int_{\mathbb{R}^3} \nabla_{x} \cdot R_{\epsilon}(t, x) \varphi(t, x) \d  x \d  t \rightarrow 0, \text{ as } \epsilon \rightarrow 0,
\end{equation}
where $R_{\epsilon}=R_{\epsilon, \mathrm{u}}$ or $R_{\epsilon, \theta}$. For other terms in \eqref{4eq17} noting the convergence \eqref{4eq8} and \eqref{4eq10}, together with \eqref{4eq18}, we have
\begin{equation}\label{4eq19}
  \int_{0}^{T} \int_{\mathbb{R}^3} \nabla_{x} \cdot \widetilde{R}_{\epsilon, \theta}(t, x) \varphi(t, x) \d  x \d  t \rightarrow 0 \quad \text { as } \quad \epsilon \rightarrow 0.
\end{equation}

From the convergence \eqref{4eq13} and \eqref{4eq14}, as $\epsilon \rightarrow 0$
\begin{equation*}
  \begin{aligned}
  \int_{0}^{T} \int_{\mathbb{R}^3} \partial_{t} \widetilde{\theta}_{\epsilon}(t, x) \varphi(t, x) \d  x \d  t \rightarrow &-\int_{\mathbb{R}^3}\left(\frac{K_g\theta_0-\rho_0}{K_g+1}\right)(x) \d  x\\
  &-\int_{0}^{T} \int_{\mathbb{R}^3} \theta(t, x) \partial_{t} \varphi(t, x) \d  x \d  t,
  \end{aligned}
\end{equation*}
\begin{equation*}
  \int_{0}^{T} \int_{\mathbb{R}^3} \Delta_{x} \widetilde{\theta}_{\epsilon} \varphi(t, x) \d  x \d  t \rightarrow \int_{0}^{T} \int_{\mathbb{R}^3} \theta(t, x) \Delta_{x} \varphi(t, x) \d  x \d  t,
\end{equation*}
and
\begin{equation*}
  \begin{aligned}
  &\int_{0}^{T} \int_{\mathbb{R}^3} \nabla_{x} \cdot\left(\mathcal{P} \mathrm{u}_{\epsilon} \widetilde{\theta}_{\epsilon}\right) \varphi(t, x) \d  x \d  t \\
  \rightarrow &-\int_{0}^{T} \int_{\mathbb{R}^3} \mathrm{u}(t, x) \theta(t, x) \cdot \nabla_{x} \varphi(t, x) \d  x \d  t.
  \end{aligned}
\end{equation*}

By taking the Leray projection $\mathcal{P}$ on the second equation of \eqref{4eq15}, we have the following equation:
\begin{equation*}
  \partial_{t} \mathcal{P} \mathrm{u}_{\epsilon}+\mathcal{P} \nabla_{x} \cdot\left(\mathcal{P} \mathrm{u}_{\epsilon} \otimes \mathcal{P} \mathrm{u}_{\epsilon}\right)-\frac{\nu_{*}}{E_2}\Delta_{x} \mathcal{P} \mathrm{u}_{\epsilon}=\mathcal{P} \nabla_{x} \cdot \widetilde{R}_{\epsilon, \mathrm{u}},
\end{equation*}
where
\begin{equation*}
  \widetilde{R}_{\epsilon, \mathrm{u}}=\frac{1}{E_2}R_{\epsilon, \mathrm{u}}-\mathcal{P} \cdot\left(\mathcal{P} \mathrm{u}_{\epsilon} \otimes \mathcal{Q} \mathrm{u}_{\epsilon}+\mathcal{Q} \mathrm{u}_{\epsilon} \otimes \mathcal{P} \mathrm{u}_{\epsilon}+\mathcal{Q} \mathrm{u}_{\epsilon} \otimes \mathcal{Q} \mathrm{u}_{\epsilon}\right).
\end{equation*}

Similar as above we can take the vector-valued test function $\psi(t, x)$ with $\nabla_{x} \cdot \psi=0,$ and prove that as $\epsilon \rightarrow 0$,
\begin{equation*}
  \begin{aligned}
  &\int_{0}^{T} \int_{\mathbb{R}^3}\left(\partial_{t} \mathcal{P} \mathrm{u}_{\epsilon}+\mathcal{P} \nabla_{x} \cdot\left(\mathcal{P} \mathrm{u}_{\epsilon} \otimes \mathcal{P} \mathrm{u}_{\epsilon}\right)-\frac{\nu_{*}}{E_2}\Delta_{x} \mathcal{P} \mathrm{u}_{\epsilon}\right) \cdot \psi(t, x) \d  x \d  t\\
  &\rightarrow-\int_{\mathbb{R}^3} \mathcal{P}u_{0}(x) \cdot \psi(0, x) \d  x\\
  &\qquad-\int_{0}^{T} \int_{\mathbb{R}^3}\left(u \cdot \partial_{t} \psi+\mathrm{u} \otimes \mathrm{u}: \nabla_{x} \psi-\frac{\nu_{*}}{E_2} \mathrm{u} \cdot \Delta_{x} \psi \right)\d  x \d  t,
  \end{aligned}
\end{equation*}
and
\begin{equation*}
  \int_{0}^{T} \int_{\mathbb{R}^3} \mathcal{P} \nabla_{x} \cdot \widetilde{R}_{\epsilon, \mathrm{u}}(t, x) \phi(t, x) \d  x \d  t \rightarrow 0 \quad \text { as } \quad \epsilon \rightarrow 0.
\end{equation*}

By collecting all above convergence results, we have shown that
\begin{equation*}
  (\mathrm{u}, \theta) \in C\left([0, \infty) ; H^{N-1}\left(\d x\right)\right) \cap L^{\infty}\left([0, \infty) ; H^{N}\left(\d x\right)\right)
\end{equation*}
satisfies the following incompressible Navier-Stokes equations
\begin{equation*}
\left\{
  \begin{array}{l}
    E_2\partial_{t} \mathrm{u}+E_2\mathrm{u} \cdot \nabla_{x} \mathrm{u}+\nabla_{x} p=\nu_{*} \Delta_{x} \mathrm{u}, \\
    \nabla_{x} \cdot \mathrm{u}=0, \\
    C_A\partial_{t} \theta+C_A\mathrm{u} \cdot \nabla_{x} \theta=\kappa_{*} \Delta_{x} \theta,
  \end{array}
\right.
\end{equation*}
with initial data:
\begin{equation*}
  \mathrm{u}(0, x)=\mathcal{P} \mathrm{u}_{0}(x), \quad \theta(0, x)=\frac{K_g\theta_0(x)-\rho_0(x)}{K_g+1},
\end{equation*}


\newpage
\begin{thebibliography}{99}
\bibitem{ref1} R. Alexandre,  On some related non homogeneous 3D Boltzmann models in the non cutoff case, {\em J. Math. Kyoto Univ.} {\bf 40} (2000), no. 3, 493-524.

\bibitem{AMUXY-1} R. Alexandre, Y. Morimoto, S. Ukai, C. J. Xu, and T. Yang,  Global existence and full regularity of the Boltzmann equation without angular cutoff, {\em Comm. Math. Phys.} {\bf 304} (2011), no. 2, 513-581.

\bibitem{AMUXY-2} R. Alexandre, Y. Morimoto, S. Ukai, C. J. Xu, and T. Yang, The Boltzmann equation without angular cutoff in the whole space I: Global existence for soft potential, {\em J. Funct. Anal.} {\bf 262} (2012), no. 3, 915-1010.

\bibitem{AMUXY-3} R. Alexandre, Y. Morimoto, S. Ukai, C. J. Xu, and T. Yang, The Boltzmann equation without angular cutoff in the whole space II: Global existence for hard potential, {\em Anal. Appl. (Singap.)} {\bf 9} (2011), no. 2, 113-134.

\bibitem{ref5} T. Allemand, Existence and conservation laws for the Boltzmann-Fermi-Dirac equation in a general domain, {\em C. R. Math. Acad. Sci. Paris} {\bf 348} (2010), no. 13-14, 763-767.

\bibitem{ref6} C. Bardos, F. Golse and C. D. Levermore, Fluid dynamic limits of kinetic equations I: formal derivation, {\em J. Stat. Phys.} {\bf 63} (1991), no. 1-2, 323-344.

\bibitem{ref7} C. Bardos, F. Golse and C. D. Levermore, Fluid dynamic limits of kinetic equations II: convergence proof for the Boltzmann equation, {\em Commun. Pure and Appl. Math.} {\bf 46} (1993), no. 5, 667-753.

\bibitem{ref8} C. Bardos and S. Ukai, The classical incompressible Navier-Stokes limit of the Boltzmann equation, {\em Math. Models Methods Appl. Sci. } {\bf 1} (1991), no. 2, 235-257.

\bibitem{ref9} D. Benedetto, M. Pulvirenti, F. Castella and R. Esposito,  On the weak-coupling limit for bosons and fermions, {\em Math. Models Methods Appl. Sci.} {\bf 15} (2005), no. 12, 1811-1843.

\bibitem{ref10} M. Briant, From the Boltzmann equation to the incompressible Navier-Stokes equations on the torus: A quantitative error estimate, {\em J. Differential Equations.} {\bf 259} (2015), no. 11, 6072-6141.

\bibitem{ref11} M. Briant, Perturbative theory for the Boltzmann equation in bounded domains with different boundary conditions, {\em Kinet. Relat. Models.} {\bf 10} (2017), no. 2, 329-371.

\bibitem{ref12} M. Briant and Y. Guo, Asymptotic stability of the Boltzmann equation with Maxwell boundary conditions, {\em J. Differential Equations.} {\bf 261} (2016), no. 12, 7000-7079.

\bibitem{ref13} M. Briant, S. Merino-Aceituno and C. Mouhot, From Boltzmann to incompressible Navier-Stokes in Sobolev spaces with polynomial weight, {\em Anal. Appl. (Singap.)} {\bf 17} (2019), no. 1, 85-116.

\bibitem{ref14} R. E. Caflisch, The fluid dynamic limit of the nonlinear Boltzmann equation, {\em Comm. Pure Appl. Math.} {\bf 33} (1980), no. 5, 651-666.

\bibitem{ref15} C. Cercignani, {\it The Boltzmann equation and its applications}, Springer-Verlag, New York, 1988.

\bibitem{ref16} C. Cercignani, R. Illner, and M. Pulvirenti, {\it The mathematical theory of dilute gases}, Springer-Verlag, New York, 1994, vol. 106.

\bibitem{ref17} P. A. M. Dirac, On the Theory of Quantum Mechanics. {\em Proceedings of the Royal Society of London A: Mathematical, Physical and Engineering Sciences.} {\bf 112} (1926), no. 762, 661-677.

\bibitem{ref18} J. Dolbeault, Kinetic models and quantum effects: A modified Boltzmann equation for Fermi–Dirac particles, {\em Arch. Ration. Mech. Anal.} {\bf 127} (1994) 101-131.

\bibitem{ref19} L. Erd\"{o}s, M. Salmhofer and H.-T. Yau, On the quantum Boltzmann equation, {\em J. Stat. Phys.} {\bf 116} (2004), no. 1-4, 367-380.

\bibitem{ref20} M. Escobedo, S. Mischler and M. A. Valle, Homogeneous Boltzmann equation in quantum relativistic kinetic theory, {\em Electronic Journal of Differential Equations.} Monograph, 4. Southwest Texas State University, San Marcos, TX, 2003.

\bibitem{ref21} M. Escobedo, S. Mischler and J. Velazquez, Asymptotic description of Dirac mass formation in kinetic equations for quantum particles, {\em J. Differential Equations.} {\bf 202} (2004), no. 2, 208-230.

\bibitem{ref22} E. Fermi, Zur Quantelung des idealen einatomigen Gases, {\em Zeitschrift für Physik} {\bf 36} (1926), no. 11-12, 902–912, English translation: A. Zannoni, On the Quantization of the Monoatomic Ideal Gas, Dec. 1999.

\bibitem{ref23} F. Filbet, J. Hu and S. Jin, A numerical scheme for the quantum Boltzmann equation with stiff collision terms, {\em ESAIM Math. Model. Numer. Anal.} {\bf 46} (2012), no. 2, 443-463.

\bibitem{ref24} R. T. Glassey, {\it The Cauchy problem in kinetic theory}, Society for Industrial and Applied Mathematics (SIAM), Philadelphia, PA, 1996.

\bibitem{GS-JAMS2011} P. T. Gressman and R. M. Strain, Global classical solutions of the Boltzmann equation without angular cut-off, {\em J. Amer. Math. Soc.} {\bf 24} (2011), no. 3, 771-847.

\bibitem{ref26} Y. Guo, The Vlasov-Maxwell-Boltzmann system near Maxwellians, {\em Comm. Pure Appl. Math.} {\bf 55} (2002), no. 9, 1104-1135.

\bibitem{ref27} Y. Guo, Classical solutions to the Boltzmann equation for molecules with an angular cutoff, {\em Arch. Ration. Mech. Anal.} {\bf 169} (2003) 305-353.

\bibitem{ref28} Y. Guo,  The Boltzmann equation in the whole space, {\em Indiana Univ. Math. J.} {\bf 53} (2004), no. 4, 1081-1094.

\bibitem{Guo-CPAM2006} Y. Guo. Boltzmann diffusive limit beyond the Navier-Stokes approximation. {\em Comm. Pure Appl. Math.}, {\bf 59} (2006), no. 5, 626-687.

\bibitem{ref29} Y. Guo, Decay and continuity of the Boltzmann equation in bounded domains, {\em Arch. Ration. Mech. Anal.} {\bf 197} (2010), no. 3, 713-809.

\bibitem{ref30} Y. Guo, J. Jang and N. Jiang, Local Hilbert expansion for the Boltzmann equation, {\em Kinet. Relat. Models.} {\bf 2} (2009), no. 1, 205-214.

\bibitem{ref31} Y. Guo, J. Jang and N. Jiang, Acoustic limit for the Boltzmann equation in optimal scaling, {\em Comm. Pure Appl. Math.} {\bf 63} (2010), no. 3, 337-361.

\bibitem{ref32} Y. Guo, C. Kim, D. Tonon and A. Trescases, Regularity of the Boltzmann equation in convex domains, {\em Invent. Math.} {\bf 207} (2017), no. 1, 115-290.

\bibitem{ref33} J. Jang and N. Jiang, Acoustic limit of the Boltzmann equation: Classical solutions. {\em Discrete Contin. Dyn. Syst.} {\bf 25} (2009), no. 3, 869-882.

\bibitem{Jiang-Xiong-2015} N. Jiang and L. Xiong, Diffusive limit of the Boltzmann equation with fluid initial layer in the periodic domain, {\em SIAM J. Math. Anal.} {\bf 47} (2015), no. 3, 1747-1777.

\bibitem{ref35} N. Jiang, C.-J. Xu and H. Zhao, Incompressible Navier-Stokes-Fourier limit from the Boltzmann equation: classical solutions, {\em Indiana Univ. Math. J.} {\bf 67} (2018), no. 5, 1817-1855.

\bibitem{ref36} L. D. Landau and E. M. Lifshitz, {\it Course of theoretical physics}, Vol. 5: Statistical physics, Pergamon Press, Oxford-Edinburgh-New York, 1968.

\bibitem{ref37} P. L. Lions, Compactness in Boltzmann’s equation via Fourier integral operators and applications III, {\em J. Math. Kyoto Univ.} {\bf 34} (1994), no. 3, 539-584.

\bibitem{ref38} X. Lu, On spatially homogeneous solutions of a modified Boltzmann equation for Fermi-Dirac particles, {\em J. Statist. Phys.} {\bf 105} (2001), no. 1-2, 353-388.

\bibitem{ref39} X. Lu, On the Boltzmann equation for Fermi-Dirac particles with very soft potentials: Global existence of weak solutions, {\em J. Differential Equations.} {\bf 245} (2008), pp. 1705-1761.

\bibitem{ref40} X. Lu and B. Wennberg, On stability and strong convergence for the spatially homogeneous Boltzmann equation for Fermi-Dirac particles, {\em Arch. Ration. Mech. Anal.} {\bf 168} (2003), no. 1, 1-34.

\bibitem{Nishida-1978} T. Nishida, Fluid dynamical limit of the nonlinear Boltzmann equation to the level of the compressible Euler equation, {\em Comm. Math. Phys.} {\bf 61} (1978), no. 2, 119-148.

\bibitem{ref43} L. W. Nordheim, On the kinetic methods in the new statistics and its applications in the electron theory of conductivity, {\em Proc. R. Soc. Lond. Ser. A} {\bf 119} (1928) 689-698.

\bibitem{Ou-WuL-2021} Z. Ouyang and L. Wu,
On the Quantum Boltzmann Equation near Maxwellian and Vacuum. {\em arXiv:2102.00657 [math.AP]}

\bibitem{ref44} L. D. Pitt, A compactness condition for linear operators on function spaces, {\em Journal of Operator Theory.} {\bf 1} (1979), no. 1, 49-54.

\bibitem{ref45} H. Spohn, Quantum kinetic equations, in: M. Fannes, C. Maes, A. Verbeure (Eds.), On Three Levels: Micro-Meso and Macro Approaches in Physics, in: {\em NATO Adv. Study Inst. Ser. B Phys.} vol. {\bf 324}, 1994, pp. 1-10.

\bibitem{ref46} E. A. Uehling and G.E. Uhlenbeck, Transport phenomena in Einstein–Bose and Fermi–Dirac gases I, {\em Phys. Rev.} {\bf 43} (1933) 552-561.

\bibitem{ref47} S. Ukai, On the existence of global solutions of mixed problem for non-linear Boltzmann equation, {\em Proc. Japan Acad.} {\bf 50} (1974), 179-184.

\bibitem{Zakrevskiy} T. Zakrevskiy, {\it Kinetic Models in the Near-Equilibrium Regime}. Thesis at Polytechnique 2015.

\end{thebibliography}
\end{document}